\documentclass[a4paper]{amsart}
\usepackage[utf8]{inputenc}

\usepackage{graphicx}
\graphicspath{ {./img/} }
\usepackage{amsmath}
\usepackage{hyperref}
\usepackage{amssymb}
\usepackage{amsthm}
\usepackage{url}
\usepackage{xfrac}
\usepackage{textcomp}
\usepackage{wrapfig}
\usepackage{mathtools}
\usepackage{tikz-cd}
\usepackage{caption}
\usepackage{faktor}
\usepackage{enumitem}
\usepackage{MnSymbol}
\usepackage[skip=8pt plus1pt, indent=15pt]{parskip}
\usepackage[style=alphabetic,
  useprefix,%
	giveninits=true,%
	hyperref,%
	doi=false,%
	url=false,%
	isbn=false,%
	backend=bibtex,%
  eprint=false,
	maxbibnames=99]{biblatex}
\usepackage[only,llbracket,rrbracket]{stmaryrd} 

\AtBeginDocument{\addtocontents{toc}{\protect\setlength{\parskip}{0pt}}} 
\graphicspath{ {./images/} }
\hypersetup{
 colorlinks,
 citecolor=black,
 filecolor=black,
 linkcolor=black,
 urlcolor=black
}

\theoremstyle{theorem}
\newtheorem{thm}{Theorem}[section]

\newtheorem*{exer*}{Exercise}
\newtheorem*{thm*}{Theorem}
\newtheorem{lem}[thm]{Lemma}
\newtheorem*{lem*}{Lemma}
\newtheorem{corollary}[thm]{Corollary}
\newtheorem{prop}[thm]{Proposition}
\theoremstyle{definition}
\newtheorem{definition}[thm]{Definition}
\newtheorem{remark}[thm]{Remark}
\newtheorem{example}[thm]{Example}
\newcommand{\dd}{{\rm d}}

\newcommand{\norm}[1]{\|#1\|}

\newcommand{\e}{{\rm e}}
\newcommand{\ad}{{\rm ad}}

\newcommand{\doublebracket}[2]{\llbracket #1,#2 \rrbracket}
\newcommand{\id}{{\rm id}}

\title[Quantitative Pansu and Mitchell Theorems]{Quantitative versions of Pansu Asymptotic Theorem and of Mitchell Tangent Theorem}
\author[Le Donne, Nicolussi Golo, and Tettamanti]{Enrico Le Donne, Sebastiano Nicolussi Golo, and Andrea Tettamanti}
\date{\today}

\begin{document}
\begin{abstract}
We quantitatively study the speed of convergence of geodesic Lie groups to their metric limits. For nilpotent geodesic Lie groups, we give estimates on the difference of the original metrics and the asymptotic metrics, while for general geodesic Lie groups, we give similar estimates for the difference of the original metrics and the tangent metrics. In both settings, our results sharpen existing bounds in the literature.
\end{abstract}
\maketitle
\setcounter{tocdepth}{3}
\tableofcontents

\section{Introduction}

Limits of groups equipped with distance functions have been extensively studied across several areas of mathematics, ranging from geometric group theory to harmonic analysis and control theory; see, as examples, 
\cite{Roth:Stein, Gromov-polygrowth, Folland-Stein, nagelstwe, Mit85,  Pansu, Stein:book, Bellaiche, gromov, Mon02, Bonfiglioli:et:al, Breuillard_MR3267520, Agrachev_Barilari_Boscain:book, Don25}. 
In many situations, limits of metric groups retain a group structure.
 Fundamental examples occur in geometric group theory, where the asymptotic cones of finitely generated nilpotent groups, equipped with word metrics, are Carnot groups -- simply connected nilpotent Lie groups endowed with geodesic distances admitting families of intrinsic dilations; see \cite{Gromov-polygrowth, Pan83}. 
These same metric Lie groups also arise as tangent cones of sub-Riemannian Lie groups and, more generally, of Carnot–Carathéodory spaces; see \cite{Mit85, Bellaiche, jeancontrol}.

Roughly speaking, given a metric space, the asymptotic cone can be thought of as the space viewed from infinitely far away, while the tangent cone can be interpreted as the space at an infinitesimal scale. Formally, given a metric space $(X,d)$ we define its {\em asymptotic cone} (respectively its {\em tangent cone} at $p\in X$) as the Gromov--Hausdorff limit, if it exists, of the sequence $ (X,\epsilon d, p)$ as $\epsilon \rightarrow 0$ (respectively as $\epsilon \rightarrow \infty)$.

In this work, we consider Lie groups equipped with left-invariant geodesic distances that induce the manifold topology; we refer to such objects concisely as {\em geodesic Lie groups}.
In fact, these are precisely what are also called sub-Finsler Lie groups; see Section~\ref{subsection: sF metrics on lie groups}.
Our goal is to study sequences of geodesic Lie groups and to analyze the rate of convergence toward their limits.
We focus on two cases:
\begin{enumerate}[label=(\alph*)]
    \item scaled-down versions of simply connected nilpotent geodesic Lie groups converging to their asymptotic cone;
    \item scaled-up versions of geodesic Lie groups converging to their tangent cone.
\end{enumerate}
These two situations correspond precisely to those studied qualitatively by Pansu and Mitchell in \cite{Pan83, Mit85}, respectively.
In both settings, the sequence of distance functions and the limiting distance function can be viewed as distance functions on the same underlying set, and the convergence is, in fact, uniform on compact sets.
This paper aims to provide quantitative estimates for the speed of this convergence. As a consequence, we obtain quantitative bounds on the difference between the metric of a geodesic Lie group and its tangent metric (or its asymptotic metric, in the case where the group is nilpotent).

The identification of a geodesic Lie group with its tangent (or asymptotic) cone is not canonical: classical constructions depend on the choice of particular linear subspaces of the corresponding Lie algebra.
We aim to demonstrate that different and opportune identifications can yield more precise quantitative estimates.
These choices consist of the selection of gradings, which we recall next following \cite{Don25} as a reference.

Given a vector space $V$, a {\em linear grading} of $V$ is a collection of subspaces $(V_i)_{i\in I}$, where $I$ is a set, such that $V=\bigoplus_{i\in I}V_i$. If $I\subseteq \mathbb R$, we set $V_{\leq j}:=\bigoplus\{V_i : i\in I,\; i\leq j\}$. An {\em asymptotic grading} (also called {\em compatible linear grading}) of a nilpotent Lie algebra $\mathfrak{g}$ is a linear grading $(V_i)_{i\in \mathbb N}$ of $\mathfrak{g}$, seen as a vector space, such that
 \begin{equation*}
  \mathfrak{g}^{(j)} = V_j \oplus \mathfrak{g}^{(j+1)}, \quad \forall j\in \mathbb{N},
 \end{equation*}
where $\mathfrak{g}^{(j)}$ denotes the $j$-th term of the lower central series of $\mathfrak{g}$.
Fix now a geodesic Lie group $(G,d)$ that is simply connected and nilpotent.
Provided a choice of an asymptotic grading of the Lie algebra $\mathfrak{g}$ of $G$, one can diffeomorphically identify the asymptotic cone $(G_\infty,d_\infty)$ of $(G,d)$ with $G$ itself; see Section~\ref{subsection: asymptotic carnot group} for the identification. The asymptotic distance can therefore be seen as a distance function on $G$, and we refer to it as the {\em Pansu limit metric} associated to the grading, following \cite{Pan83} as initial reference.
 The celebrated Pansu Asymptotic Theorem, see page~417 in \cite{Pan83}, can be expressed as a metric comparison theorem:
\begin{align*}
 \frac{d(1,g)}{d_\infty(1,g)} \rightarrow 1, \quad \text{as }  g \rightarrow \infty.
\end{align*}
One can then study the error term $|d(1,g)- d_\infty(1,g)|$ to quantify the speed of convergence toward the asymptotic cone. Clearly, by Pansu Asymptotic Theorem, this error is sublinear as a function of $d(1,g)$. Tashiro proved in \cite[Theorem~4.1]{Tas22} that if $G$ is $2$-step nilpotent and non-singular, then the error term is bounded by a constant. In \cite{BL12} Breuillard and Le Donne gave a bound on the error term for groups of arbitrary nilpotency step. They showed that if $G$ is $s$-step nilpotent, then
\begin{align} \label{intro: eq: BL12_bound}
  |d(1,g)- d_\infty(1,g)| = O\left(d(1,g)^{1-1/s}\right), \quad {\rm as}\; g \rightarrow \infty.
\end{align}
In Riemannian step-2 groups, we have sharper results, see \cite{LeNaNiRy25}. However, in general sub-Finsler groups, the bound \eqref{intro: eq: BL12_bound} is sharp.

Let $(G,d)$ be a simply connected nilpotent geodesic Lie group and $(V_i)_i$ be an asymptotic grading of $\mathfrak{g}$. We introduce two constants, $\alpha_\infty$ and $\beta$, to improve the bound \eqref{intro: eq: BL12_bound} of Breuillard and Le Donne. The constant $\alpha_\infty$ quantifies how close $V_1$ is to being the first stratum of a stratification of a Carnot group, in the standard sense as recalled in Section~\ref{subsection: gradings}. We begin by defining
$$\alpha_{(1,\infty)} \coloneqq \max \left\{j \in \mathbb{N}: [V_{p_1},\ldots,V_{p_k}] \subseteq V_{|p|} \oplus V_{\geq |p|+j}, \; \forall k \in \mathbb{N}, \; \forall p \in \mathbb{N}^k\right\},$$
where $[V_a,V_b]$ denotes the space spanned by brackets of the form $[v_a,v_b]$ with $v_a \in V_a$, $v_b \in V_b$ and $[V_{p_1},\ldots,V_{p_k}]$ denotes the left-iterated bracket $[V_{p_1},[\cdots,[V_{p_{k-1}},V_{p_k}]\cdots]]$. 
Next, denoting by $\pi_1$ the projection $\mathfrak{g}=V_1 \oplus [\mathfrak{g},\mathfrak{g}] \rightarrow V_1$ 
with kernel $[\mathfrak{g},\mathfrak{g}] $,
and by $\Delta\subseteq \mathfrak{g}$ the distribution associated to $d$ (see Section~\ref{subsection: sF metrics on lie groups} for the definition), we define
 \begin{align*}
  \alpha_{(2,\infty)} &\coloneqq \max \left\{j \in \mathbb{N} \, : \, v-\pi_1(v) \in \mathfrak{g}^{(j+1)}, \, \forall v \in \Delta \right\}. 
 \end{align*}
Finally, we introduce the following constant
\begin{eqnarray}\nonumber
   \alpha_{\infty}&\coloneqq & \alpha_{\infty}(\Delta, (V_i)_i) \\\nonumber
   &\coloneqq &\min\{\alpha_{(1,\infty)}, \alpha_{(2,\infty)}\}
   \\\label{def_alpha_inf}
   &=& 
   \max \left\{j \in \mathbb{N} \, : 
   \begin{cases}
    k \in \mathbb{N}, \; p \in \mathbb{N}^k, \;\\ v \in \Delta \end{cases}\implies \begin{cases}
     [V_{p_1},\ldots,V_{p_k}] \subseteq V_{|p|} \oplus V_{\geq |p|+j}, \,\\
   \, v-\pi_1(v) \in \mathfrak{g}^{(j+1)} 
   \end{cases}\right\}
   .
\end{eqnarray}
The constant $\alpha_{(1,\infty)}$ is a variation of the constant $e_D$ introduced by Cornullier in \cite{Cor17}. 
Observe that $(G, \Delta)$ is a Carnot group if and only if $\alpha_{\infty} = + \infty$. In this latter case, the distances $d$ and $d_\infty$ coincide. We refer to Section~\ref{subsection: gronwall pansu} for other properties of $\alpha_\infty$.

For a Lie group $H$ equipped with a distribution $\Delta$ and a linear grading $(D_i)_{i}$ of its Lie algebra $\mathfrak{h} $, we introduce a class of ideals of $\mathfrak{h} $, which we call Carnot quotient ideals, so that the quotient Lie algebras are Carnot in the following sense. A subspace $\mathfrak{i}\subseteq \mathfrak{h}$ is a \textit{Carnot quotient ideal of} $(H,\Delta)$ \textit{with respect to the grading $(D_i)_{i}$} if 
 \begin{itemize}
  \item $\mathfrak{i}  $ is an ideal of $\mathfrak{h}$;
  \item $\mathfrak{h}/\mathfrak{i}$ is stratified by $(D_i + \mathfrak{i})_{i}$;
  \item \label{def: CQI projection} 
  $\Delta \subseteq D_1 + \mathfrak{i}.$
 \end{itemize}
Next, we define $\beta $ as the minimal integer $k$ such that there exists a Carnot quotient ideal contained in the subspace $D_1 \oplus \dots \oplus D_k$: 
\begin{equation}\label{def_beta}
 \beta:=\beta(\Delta,(D_i)_i):= \min\{k\in \mathbb Z\,:\, \exists \mathfrak{i} \subseteq D_{\leq k} \, \text{Carnot quotient ideal}
 \}.
\end{equation}
The example that one should keep in mind for the constant $\beta$ is the following. If $G$ is a Carnot group of step $s_1$ and $H$ is a simply connected $s_2$-step nilpotent geodesic Lie group, then there exists an asymptotic grading of $G\times H$ such that $\beta \leq s_2$; see Proposition~\ref{prop: pansu convergence doesnt depend on carnot factor}.
Note that $\beta=0$ if and only if $(G,\Delta)$ is a Carnot group, in which case the following bounds will be trivial.
We are ready to state our first main result.

\begin{thm}[Quantitative Pansu]\label{main thm: pansu}
 Let $(G,d)$ be a geodesic Lie group. Assume $G$ is simply connected and nilpotent.
 After fixing an asymptotic grading, consider the associated Pansu limit metric $d_\infty$ on $G$ and the constants $\alpha_\infty$ and $\beta$ defined in \eqref{def_alpha_inf} and \eqref{def_beta}, respectively. Then, for some $C>0$, we have 
 \begin{align}\label{eq: main pansu1}
  \left|d(p,q)-d_\infty(p,q)\right| \leq C\max\left\{d_\infty(1,p),d_\infty(1,q)\right\}^{1-\alpha_\infty/\beta}, \quad \forall p\in G, \;\forall q \in G \setminus B_{d_\infty}(1,1).
 \end{align}
\end{thm}

  The degenerate case when $\alpha_\infty/\beta=\infty$ coincides with $(G, \Delta)$ being a Carnot group, in which case, we trivially have $d=d_\infty$.
In all cases, we have  $\alpha_\infty/\beta\geq 1/s$, where $s$ is the nilpotency step of $G$, and therefore
Theorem~\ref{main thm: pansu} is an improvement of the bound~\eqref{intro: eq: BL12_bound}. From the bound~\eqref{eq: main pansu1} we directly see that $(G,d_\infty)$ is the asymptotic cone of $(G,d)$, see Theorem~\ref{thm: Pansu asymptotic thm} and Theorem~\ref{thm: quantitative Pansu}.

The second main result is about a similar bound for the rate of convergence toward the tangent cones of geodesic Lie groups. Let $\mathfrak{g}$ be a Lie algebra equipped with a bracket-generating distribution $\Delta$.
A {\em tangent grading} (also called {\em adapted linear grading}) of $(\mathfrak{g},\Delta)$ is a grading $(W_i)_{i}$ of $\mathfrak{g}$ satisfying
\[\sum_{i=1}^j \ad_\Delta^{i-1}\Delta = W_j \oplus \sum_{i=1}^{j-1} \ad_\Delta^{i-1}\Delta , \quad \forall j\in \mathbb{N},\]
where we set $\ad^k_A B\coloneqq [A,\ldots,A,B]$ with $A$ appearing $k$ times.
Notice that $W_1=\Delta$ and one such grading exists because $\Delta$ is assumed bracket generating.
Fix now a geodesic Lie group $(G,d)$, with Lie algebra $\mathfrak{g}$ and distribution $\Delta$.
After choosing a tangent grading of $(\mathfrak{g},\Delta)$, one can locally identify $G$ with the tangent cone $(G_0,d_0)$ of $(G,d)$ at 1; see Section~\ref{subsection: osculating carnot group} for the identification. Thus, we can consider the metric $d_0$, which we call \textit{tangent metric} associated to the grading, as a distance function in a neighborhood of 1 in $G$. Mitchell Tangent Theorem can be expressed as a metric comparison theorem:
\[\frac{d(1,g)}{d_0(1,g)} \rightarrow 1, \quad {\rm as }\; g \rightarrow 1.\]
Similarly to the asymptotic case, we quantify the error $|d(1,g) -d_0(1,g)|$ in terms of two constants: $\alpha_0$ and $\beta$. The constant $\beta$ coincides with the one appearing in \eqref{def_beta} for the asymptotic grading, whereas $\alpha_0$ is the analogue of $\alpha_\infty$, adapted to the context of tangent gradings. It is defined as
\begin{equation}\label{def_alpha0}
 \alpha_{0} \coloneqq \alpha_{0}((W_i)_i) \coloneqq \max \left\{j \in \mathbb{N}: [W_{p_1},\ldots,W_{p_k}] \subseteq W_{|p|} \oplus W_{\leq |p|-j}, \; \forall k \in \mathbb{N}, \; \forall p \in \mathbb{N}^k\right\}.
 \end{equation}
 Here $\alpha_0 = \infty$ if and only if $(G,\Delta)$ is a Carnot group. 
Our second main result is the following bound:
\begin{thm}[Quantitative Mitchell]\label{main thm: mitchell}
 Let $(G,d)$ be a Lie group equipped with a left-invariant geodesic metric. After fixing a tangent grading, consider the associated tangent metric $d_0$ on a bounded neighborhood $U$ of the identity in $G$ and the constants $\alpha_0$ and $\beta$ defined in \eqref{def_alpha0} and \eqref{def_beta}, respectively. Then, for some $C>0$, we have 
  \begin{align*}
  &-Cd(p,q)^{1/\beta}d_0(1,p)^{1+(\alpha_0 -1)/\beta}\\
     &\leq d(p,q)-d_0(p,q)\\
     &\leq Cd_0(p,q)^{1/\beta}d_0(1,p)^{1+(\alpha_0 -1)/\beta}  \quad \forall p,q \in U \text{ with } d_0(1,p) \geq d_0(1,q).
 \end{align*}
 In particular, we have 
\begin{align}\label{eq: cor mitchell1}
  |d(p,q)-d_0(p,q)| \leq C\max\{d_0(1,p),d_0(1,q)\}^{1+\alpha_0/\beta}, \quad \forall p,q \in U.
 \end{align}
\end{thm}
From the bound~\eqref{eq: cor mitchell1}, we directly see that $(G,d_0)$ is the tangent cone of $(G,d)$, see Theorem~\ref{thm: Mitchell tangent thm} and Theorem~\ref{thm: quantitative Mitchell}.

The degenerate case $\alpha_0/\beta=\infty$ corresponds to $(G,\Delta)$ being a Carnot group, in which case $d$ and $d_0$ coincide. While Theorem~\ref{main thm: pansu} strengthens the result by  Breuillard and Le Donne, Theorem~\ref{main thm: mitchell} is a quantification of Mitchell's result that is sharper than the ones obtained in \cite{Bellaiche}. Indeed, denoting by $s$ the nilpotency step of the tangent cone $G_0$, Bellaiche obtains in \cite[Theorem 6.4]{Bellaiche}, for $\epsilon$ small enough, the bounds 
\begin{equation}\label{eq: bellaiche bound}
    -C\epsilon d(p,q)^{1/s} \leq d(p,q)-d_0(p,q) \leq C\epsilon d_0(p,q)^{1/s}, \quad \forall p,q \in B_{d_0}(1,\epsilon).
\end{equation}
Our new bound in Theorem~\ref{main thm: mitchell} implies \eqref{eq: bellaiche bound} since $\beta \leq s$ and $(\alpha_0 -1)/\beta \geq 0$.

We briefly summarize the strategy to obtain the theorems. The two proofs are similar, so we focus on Theorem~\ref{main thm: pansu}. We consider a one-parameter family of Lie group products $(*_\epsilon)_{\epsilon\in \mathbb{R}}$ on $G$ such that $G=(G, *_1)$ and $(G, *_0)$ is isomorphic to the asymptotic cone of $G$. We proceed by introducing a one-parameter family of geodesic metrics $(\rho_\epsilon)_{\epsilon\in \mathbb{R}}$ on $G$ such that $\rho_\epsilon$ is $*_\epsilon$-left-invariant for every $\epsilon$, $\rho_1 = d$, and $\rho_0 = d_\infty.$ Theorem~\ref{main thm: pansu} is a consequence of Theorem~\ref{thm: quantitative Pansu}, which states that the metrics $\rho_\epsilon$ quantitatively converge uniformly to $\rho_0$ on compact sets. Namely, for every compact $K \subseteq G$, there exists a constant $C$ such that 
\begin{equation*}
  \left|\rho_0(p,q)-\rho_\epsilon(p,q)\right| \leq C \epsilon^{\alpha_{\infty}/\beta}, \quad \forall \epsilon \in [0,1], \; \forall p,q \in K.
 \end{equation*}
We explain how we prove the inequality $\rho_0(p,q)-\rho_\epsilon(p,q)\lesssim \epsilon^{\alpha_{\infty}/\beta}$, for a fixed compact $K$; the proof of the other inequality uses the same techniques. We start with a $\rho_\epsilon$-geodesic $\gamma$ from $p$ to $q$, we construct a curve $\gamma_0$ such that $\gamma_0(0)=p$ and $\text{Length}_{\rho_0}(\gamma_0) \leq \text{Length}_{\rho_\epsilon}(\gamma)$, and then bound the $\rho_0$-distance of the endpoints of $\gamma$ and $\gamma_0$. To bound the distance of the endpoints, we observe that for every Carnot quotient ideal $\mathfrak{i} \triangleleft \text{Lie(G)}$ and for every $\epsilon \in [0,1]$, the curve $t \mapsto\gamma_\epsilon^{-1}(t)*_\epsilon\gamma(t)$ lies in $\mathfrak{i}$. One can then use the Ball-Box Theorem and the definition of $\beta$, as well as a Grönwall-type lemma and the definition of $\alpha_\infty$, to bound the $\rho_0$-distance of the endpoints of $\gamma$ and $\gamma_0$.

This work is structured as follows. In Section~\ref{section: sub-Finsler Lie groups}, we introduce the basic definitions and known results about geodesic Lie groups. In Section~\ref{section: large and small}, we present the necessary notions and results about asymptotic cones and tangent spaces of Lie groups. Section~$\ref{section: gronwall}$ contains the proof of two Grönwall-type lemmata that are needed for the proofs of the main theorems. Finally, in Sections~\ref{section: quantitative pansu} and~\ref{section: quantitative mitchell}, we prove Theorems~\ref{main thm: pansu} and~\ref{main thm: mitchell}, respectively. We conclude with some applications and examples in Section~\ref{section: applications}.

\newpage
\section{Sub-Finsler Lie groups}\label{section: sub-Finsler Lie groups}
\subsection{Sub-Finsler metrics on Lie groups}\label{subsection: sF metrics on lie groups}
Our main objects of study are Lie groups equipped with left-invariant geodesic metrics inducing the manifold topology. By a result of Berestovskij \cite{Ber88}, these metric Lie groups are precisely the sub-Finsler Lie groups, which we now present, following the style of the monograph \cite{Don25}.

Let $G$ be a Lie group with Lie algebra $\mathfrak{g}$, viewed as the tangent space of $G$ at the identity element $1$. A \textit{left-invariant polarization} on $G$ is a sub-bundle $\Delta \subset TG$ such that for all $g\in G$ we have
$\Delta_g \coloneqq \Delta \cap T_gG = (L_g)_*\Delta_{1} \subset T_gG$, where $L_g \colon G \rightarrow G$ denote the left translation by $g$ and $\Delta_{1}$ is a vector subspace of $\mathfrak{g}$. Note that space $\Delta_{1}$ uniquely determines the left-invariant polarization $\Delta$, hence we have a bijection between the set of vector subspaces of $\mathfrak{g}$ and the set of left-invariant polarizations on $G$. To simplify the notation, we also denote $\Delta_1= \Delta$ when there is no possibility of confusion. A left-invariant polarization $\Delta$ is \textit{bracket-generating} if the Lie algebra generated by the vector fields tangent to $\Delta$ contains a frame for $TG$. Since the polarization $\Delta$ is left-invariant, it is bracket-generating if and only if the Lie subalgebra of $\mathfrak{g}$ generated by $\Delta_1$ is $\mathfrak{g}.$ 

Fix a norm $\|\cdot \|$ on $\mathfrak{g}$.
Let $\Delta_1 \subseteq \mathfrak{g}$ be the subspace corresponding to a left-invariant bracket-generating polarization $\Delta$ on $G$. We use the norm on $\mathfrak{g}$ to define a left-invariant norm on the polarization. Let $\| \cdot \|_1$ be the restriction to $\Delta_1$ of the norm on $\mathfrak{g}.$ Then for $g \in G$ we define the norm on $\Delta_g$ to be the pushforward by the left translation of the norm on $\Delta_1$
\begin{align*}
 \|v\|_g \coloneqq \|(\dd L_{g^{-1}})_gv \|_1,\quad \forall g \in G, \; \forall v \in \Delta_g.
\end{align*}

The triple $(G, \Delta, \| \cdot \|)$ is called \textit{sub-Finsler Lie group}. Every sub-Finsler Lie group has an associated left-invariant metric, known as \textit{Carnot-Carathéodory metric} or \textit{sub-Finsler metric} defined as follows. An absolutely continuous curve $\gamma : I \rightarrow G$ defined on an interval $I$ is \textit{$\Delta$-horizontal} if, for almost every $t$, its derivative at time $t$ is contained in $\Delta_{\gamma(t)}.$ Then, the sub-Finsler distance $d_{sF}$ of two points $x,y \in G$ is defined as
\begin{align*}
 d_{sF}(x,y) \coloneqq \inf \big\{L(\gamma) : \gamma \; \Delta\text{-horizontal curve from } x \text{ to } y \big\},
\end{align*}
where $L(\gamma) \coloneqq \int_I \|\Dot{\gamma}(t)\|_{\gamma(t)} \text{d}t.$ One could also consider sub-Finsler metrics associated to non-bracket-generating polarizations. However, we shall restrict our study to metrics corresponding to bracket-generating polarizations because, by Chow's theorem (also known as Chow--Rashevsky Theorem, see \cite{Cho40} and \cite{Ras38}), such metrics are finite, induce the manifold topology, and are geodesic. Therefore, every two points can be joined by a geodesic path. If the norm is induced by a scalar product, the metric is said to be \textit{sub-Riemannian}. If the space $\Delta$ is the whole Lie algebra $\mathfrak{g}$, the metric is said to be \textit{Finsler} (or \textit{Riemannian} if in addition the norm comes from a scalar product). We also refer to the metric group $(G,d_{sF})$ as a \textit{sub-Finsler Lie group}.

An important result in this context is the Ball-box Theorem, which locally compares Riemannian and sub-Finsler metrics. Suppose that $\Delta$ together with all Lie brackets of order at most $s$ of elements of $\Delta$ span $\mathfrak{g}$ linearly. The Ball-box Theorem (see for example \cite[Theorem~2.4.2]{Mon02}) implies that for every sub-Finsler metric $d_{sF}$ and every Riemannian metric $d_e$ we have, on every compact $K \subset G$:
\begin{align} \label{eq: ball-box}
 \frac{1}{C}d_{sF}(x,y)^s \leq d_e(x,y) \leq Cd_{sF}(x,y), \quad \forall x,y \in K,
\end{align}
for some constant $C \geq 1.$

\subsection{Baker--Campbell--Hausdorff formula}\label{subsection: BCH}
The Baker--Campbell--Hausdorff formula, BCH formula in short, enables us to reconstruct the product of a Lie group in a neighborhood of the identity from the structure of its Lie algebra. For simply connected nilpotent Lie groups the BCH formula allows us to identify the group with its Lie algebra and read the product only in terms of the Lie algebra operations. 

Let $G=(G,\cdot)$ be a Lie group with Lie algebra $(\mathfrak{g}, [\cdot, \cdot])$. Then there exists a neighborhood of the origin $V \subseteq \mathfrak{g}$ and a neighborhood of the identity $U \subseteq G$ such that $\exp: V \rightarrow U$ is a diffeomorphism. The BCH formula is given by
\begin{align} \label{def: BCH}
 \log(\exp(x_1)\cdot \exp(x_2)) = x_1 + x_2 + \sum_{k=2}^\infty \sum_{q \in \{1,2\}^k}b_{k,q} \lbrack x_{q_1},\ldots,x_{q_k} \rbrack, \quad \forall x_1,x_2 \in V,
\end{align}
where $b_{k,q}$ are universal constants and $\lbrack x_{q_1},\ldots,x_{q_k} \rbrack$ denotes the left-iterated bracket: $\lbrack x_{q_1},\ldots,x_{q_k} \rbrack:=\lbrack x_{q_1},\lbrack\cdots,\lbrack x_{q_{k-1}},x_{q_k} \rbrack\cdots\rbrack\rbrack.$

The BCH formula defines a local group operation $*$ on $V$, called the \textit{Dynkin product}:
\begin{align*} \label{def: Dynkin product}
 x_1 * x_2 \coloneqq \log(\exp(x_1) \cdot \exp(x_2)).
\end{align*}
We obtain a local group isomorphism (and, hence, an identification) via the exponential map:
\begin{align*} 
 \exp : V \rightarrow U, \quad (V,*)\simeq (U,\cdot).
\end{align*}

If the group $G$ is simply connected and nilpotent, the exponential map is a global diffeomorphism \cite[Theorem~1.2.1]{CG90}, and thus $(\mathfrak{g}, *) \simeq (G,\cdot)$. Therefore in this case $(\mathfrak{g}, *)$ is a Lie group with Lie algebra $(\mathfrak{g}, [\cdot, \cdot])$. Moreover, its exponential map is the identity map \cite[Proposition~9.4.4]{Don25}:
\[\id = \exp:(\mathfrak{g}, [\cdot, \cdot]) \rightarrow (\mathfrak{g}, *).\]

\subsection{Gradings of vector spaces and Lie algebras}\label{subsection: gradings}
Let $V$ be a real vector space and $k\in \mathbb{N}$. A collection of subspaces $(D_i)_{i=1}^k$ is a \textit{(linear) grading of V} if $V= \bigoplus_{i=1}^kD_i.$ For $i \in \{1,\ldots,k\}$, we call $D_i$ the \textit{$i$-th layer of the grading}. We denote by $\pi_i:V \rightarrow D_i$ the projection on the $i$-th layer with respect to the grading, and for $x \in V$ we use the notation $(x)_i \coloneqq \pi_i(x).$ We denote by $D_{\leq i}$ (respectively $D_{\geq i}$) the subspace $D_1 \oplus\cdots\oplus D_i$ (respectively $D_i \oplus\cdots\oplus D_k$).

Gradings induce linear maps as follows. Let $(D_i)_{i=1}^k$ be a grading of $V$. For $\epsilon \in \mathbb{R}$, define $\delta_\epsilon: V \rightarrow V$ to be the linear map with the property that $\delta_\epsilon(x)=\epsilon^jx$, for all $j\in \{1,\ldots,k\}$, for all $x\in D_j$. The family $(\delta_\epsilon)_{\epsilon \in \mathbb{R}}$ is called the \textit{one-parameter family of dilations associated to $(D_i)_{i=1}^k$}. 

We consider two kinds of linear gradings of polarized Lie groups, which we call asymptotic and tangent linear gradings, respectively. Asymptotic gradings were first introduced by Cornulier in \cite{Cor17}, under the name of compatible linear gradings, and were used to identify a nilpotent Lie group with its asymptotic cone. Similarly, a choice of tangent grading provides a (local) identification of a sub-Finsler Lie group with its tangent cone. These identifications are explained in Sections~\ref{subsection: asymptotic carnot group} and~\ref{subsection: osculating carnot group}.

Let $\mathfrak{g}$ be a Lie algebra. Denote by $(\mathfrak{g}^{(i)})_{i\geq1}$ the lower central series of $\mathfrak{g},$ which is the series of subspace of $\mathfrak{g}$ defined inductively by $\mathfrak{g}^{(1)} \coloneqq \mathfrak{g}$ and $\mathfrak{g}^{(i+1)} \coloneqq \lbrack \mathfrak{g}, \mathfrak{g}^{(i)} \rbrack.$
\begin{definition}[Asymptotic grading]\label{def: compatible linear grading}
 Let $\mathfrak{g}$ be an $s$-step nilpotent Lie algebra. A grading $(V_j)_{j=1}^s$ of $\mathfrak{g}$ is an \textit{asymptotic grading} if
 \begin{equation}\label{def: eq: compatible linear grading}
  \mathfrak{g}^{(j)} = V_j \oplus \mathfrak{g}^{(j+1)}, \quad \forall j\in \{1,\ldots,s\}.
 \end{equation}
\end{definition}

\begin{definition}[Tangent grading]\label{def: adapted linear grading}
 Let $\mathfrak{g}$ be a Lie algebra, let $\Delta$ be a bracket generating subspace. Define, for $k \in \mathbb{N}$:
 \[\Delta^k \coloneqq \ad_\Delta^{k-1}\Delta \coloneqq \underbrace{[\Delta,\ldots,\Delta]}_{k\text{ times}},\]
 and
 \[\Delta^{[k]} \coloneqq \sum_{i=1}^k \Delta^k.\]
 Let $s \coloneqq \min\{n\in \mathbb{N}: \Delta^{[n]} = \mathfrak{g}\}$ be the step of $\Delta$. A grading $(W_j)_{j=1}^s$ of $\mathfrak{g}$ is an \textit{tangent grading of $\mathfrak{g}$ for $\Delta$} if 
 \[\Delta^{[j]} = W_j \oplus \Delta^{[j-1]} , \quad \forall j\in \{1,\ldots,s\}.\]
\end{definition}

In what follows, a particular kind of linear gradings, called stratifications, play an important role. We say that a Lie algebra $\mathfrak{g}$ \textit{admits an s-step stratification}, for a positive integer $s$, if there is a grading $\mathfrak{g}=\bigoplus_{i=1}^s D_i$ such that $\lbrack D_1, D_j \rbrack = D_{j+1}$, for $1 \leq j \leq s-1$, $D_s \neq \{0\}$, and $\lbrack D_1, D_s \rbrack = \{0\}$. We call such a grading a \textit{stratification}. Stratifications are unique up to Lie algebra isomorphisms. The spaces $D_1,\ldots, D_s$ in the definition of a stratification are called \textit{strata} and for every $i$ we refer to the space $D_i$ as the \textit{$i$-th stratum}. Note that every Lie algebra that admits an $s$-step stratification is nilpotent with nilpotency step $s$.

An \textit{(algebraic) Carnot group} is a simply connected Lie group whose Lie algebra admits a stratification, and one such stratification is fixed. Let $G$ be a Carnot group and $\Delta$ be the first stratum of $\text{Lie}(G)$. If $d$ is a sub-Finsler metric associated to the polarization $\Delta$, we call the metric group $(G,d)$ a $\textit{(metric) Carnot group}.$

Metric Carnot groups $(G,d)$ have many nice properties. Among them are self-similarity, that is, they admit isometries $(G,d) \rightarrow (G,rd)$ for every $r>0$, and isometric homogeneity, that is, isometries act on $G$ transitively. A consequence is that the Ball-Box Theorem holds globally:
\begin{thm}[Ball-Box Theorem for Carnot groups {\cite[Theorem~11.2.3]{Don25}}]\label{thm: Ball-Box for Carnot}
 Let $(G,d)$ be a Carnot group stratified by $(V_i)_{i=1}^s$, which we identify with $(\mathfrak{g}, *)$ as explained in Section~\ref{subsection: BCH}. Then, for every norm $\norm{\cdot}$ on $\mathfrak{g}$ there is a constant $C \geq 1$ such that
 \begin{align}\label{thm: eq: Ball box for Carnot}
  \frac{1}{C}d(0,x) \leq \sum_{i=1}^s\norm{(x)_i}^{1/i}\leq Cd(0,x), \quad \forall x \in \mathfrak{g}.
 \end{align}
\end{thm}

\begin{prop}[Properties of linear gradings]\label{prop: properties of linear gradings}
 Let $(\mathfrak{g}, \Delta)$ be a polarized Lie algebra and let $q \in \mathbb{N}$ be the step of $\Delta$. Let $(W_i)_{i=1}^q$ be a tangent grading for $\mathfrak{g}$. Then
 \begin{enumerate}
     \item \label{prop: bracket of ALG containment}$[W_i,W_j] \subseteq W_{\leq i+j}, \; \forall i,j\in \{1,\ldots,q\}.$
     \item If $[W_1,W_j] \subseteq W_{1+j}, \; \forall j \in \{1,\ldots,q-1\}$, and $[W_1,W_q] = \{0\}$, then $(W_i)_{i=1}^q$ is a stratification.
 \end{enumerate}
 If $\mathfrak{g}$ is nilpotent of step $s\in \mathbb{N}$ and $(V_i)_{i=1}^s$ is an asymptotic grading for $\mathfrak{g}$, then
\begin{enumerate}\setcounter{enumi}{2}
    \item $q\leq s$.
    \item \label{prop: CLG inside ALG}$\mathfrak{g}^{(j)} = \Delta^j + \mathfrak{g}^{(j+1)}, \quad \forall j \in \{1,\ldots,s\}.$
    \item $\mathfrak{g} = \Delta^{[j]} + \mathfrak{g}^{(j+1)}, \quad \forall j \in \{1,\ldots,s\}.$
    \item \label{prop: bracket of CLG containment}$[V_i,V_j] \subseteq V_{\geq i+j}, \; \forall i,j\in \{1,\ldots,s\}.$
    \item \label{prop: CLG + algebra grading is stratifictaion} If $[V_1, V_j] \subseteq V_{1+j}, \; \forall j \in \{1,\ldots,s-1\}$, then $(V_i)_{i=1}^s$ is a stratification.
\end{enumerate}
\end{prop}
\begin{proof}
(1). Let $i,j\in \{1,\ldots,q\}.$ Then
 \begin{align*}
  [W_i,W_j] \subseteq [\Delta^{[i]},\Delta^{[j]}] \subseteq \Delta^{[i+j]} = W_{\leq i+j}.
 \end{align*}

(2). We prove by induction on $k$ that
 \begin{equation*}\label{prop: eq: CLG + algebra grading is stratifictaion}
  [W_1, W_k] = W_{k+1}, \quad \forall k \in \{1,\ldots,q-1\}.
 \end{equation*}
 
 For $k = 1$, by the definition of tangent grading:
 \begin{align*}
  \Delta^{[2]} &= W_2 \oplus \Delta, \text{ and}\\
  \Delta^{[2]} &\overset{\text{def}}{=} [\Delta, \Delta] + \Delta \\
  &= [W_1,W_1] + \Delta.
 \end{align*}
 Thus $\dim(W_2)\leq \dim([W_1,W_1])$. But $[W_1,W_1] \subseteq W_2$, hence $W_2= [W_1,W_1].$

 Now fix $k\in \{1,\ldots,q-2\}$ and suppose that the claim holds up to $k$. Again, by the definition of tangent grading:
 \begin{align*}
  \Delta^{[k+2]} &= W_{k+2} \oplus \Delta^{[k+1]}, \text{ and}\\
  \Delta^{[k+2]} &\overset{\text{def}}{=} \Delta^{[k+1]} + \Delta^{k+2} \\
  &= \Delta^{[k+1]} + [W_1,W_{k+1}],
 \end{align*}
 where in the last equation we use the induction hypothesis to deduce that
 \[\Delta^{k+2} = \underbrace{[\Delta,\ldots,\Delta]}_{k+1\text{ brackets}} = [W_1,\underbrace{[W_1,\ldots,W_1]}_{k \text{ brackets}}] \overset{\text{ind. hyp.}}{=} [W_1, W_{k+1}].\]
 Therefore $\dim(W_{k+2}) \leq \dim([W_1,W_{k+1}])$. But $[W_1,W_{k+1}] \subseteq W_{k+2}$, and thus $W_{k+2}= [W_1,W_{k+1}].$

(3). This is clear from the definitions of asymptotic and tangent gradings.

 (4). We induct on $j$. For $j=1$ it is clear, since the projection on the abelianization $\mathfrak{g}/[\mathfrak{g}, \mathfrak{g}]$ restricted to bracket generating subspaces is surjective. So, suppose the claim holds for some $j \in \{1,\ldots,s-1\}$. Then
 \begin{align*}
  \mathfrak{g}^{(j+1)} = [\mathfrak{g}^{(j)}, \mathfrak{g}] = [\Delta^j+\mathfrak{g}^{(j+1)}, \Delta+\mathfrak{g}^{(2)}] \subseteq \Delta^{j+1} + \mathfrak{g}^{(j+2)},
 \end{align*}
 and since $\Delta^{j+1}, \mathfrak{g}^{(j+2)} \subseteq \mathfrak{g}^{(j+1)}$, we have $\mathfrak{g}^{(j+1)} = \Delta^{j+1} + \mathfrak{g}^{(j+2)}.$

 (5). We induct on $j$. We proved the case $j=1$ in \eqref{prop: CLG inside ALG}. So, suppose the claim holds for some $j \in \{1,\ldots,s-1\}$. Then
 \begin{align*}
  \mathfrak{g} = \Delta^{[j]} + \mathfrak{g}^{(j+1)} \overset{\eqref{prop: CLG inside ALG}}{=}\Delta^{[j]} + \Delta^{j+1}+\mathfrak{g}^{(j+2)} = \Delta^{[j+1]} + \mathfrak{g}^{(j+2)}.
 \end{align*}

 (6). Let $i,j\in \{1,\ldots,s\}.$ Then
 \begin{align*}
  [V_i,V_j] \subseteq [\mathfrak{g}^{(i)},\mathfrak{g}^{(j)}] \subseteq \mathfrak{g}^{(i+j)} = V_{\geq i+j}.
 \end{align*}

 (7). Since $V_s = \mathfrak{g}^{(s)}$, then $[V_1,V_s]=\{0\}$. We prove by induction on $k$ that
 \begin{equation*}\label{prop: eq: CLG + algebra grading is stratifictaion}
  [V_1, V_k] = V_{k+1}, \quad \forall k \in \{1,\ldots,s-1\}.
 \end{equation*}
 
 For $k = 1$, by definition of asymptotic grading:
 \begin{align*}
  \mathfrak{g}^{(2)} &= V_2 \oplus \mathfrak{g}^{(3)}, \text{ and}\\
  \mathfrak{g}^{(2)} &= [\mathfrak{g}, \mathfrak{g}] \\
  &= [V_1 \oplus \mathfrak{g}^{(2)}, V_1 \oplus \mathfrak{g}^{(2)}]\\
  &\overset{\ref{prop: properties of linear gradings}\eqref{prop: bracket of CLG containment}}{\subseteq} [V_1,V_1] + \mathfrak{g}^{(3)}.
 \end{align*}
 Thus $\dim(V_2)\leq \dim([V_1,V_1])$. But $[V_1,V_1] \subseteq V_2$, and therefore $V_2= [V_1,V_1].$

 Now fix $k\in \{1,\ldots,s-2\}$ and suppose that the claim holds up to $k$. Again, by the definition of asymptotic grading:
 \begin{align*}
  \mathfrak{g}^{(k+2)} &= V_{k+2} \oplus \mathfrak{g}^{(k+3)}, \text{ and}\\
  \mathfrak{g}^{(k+2)} &= [\mathfrak{g}, \mathfrak{g}^{(k+1)}] \\
  &= [V_1 \oplus \mathfrak{g}^{(2)}, V_{k+1} \oplus \mathfrak{g}^{(k+2)}]\\
  &\overset{\ref{prop: properties of linear gradings}\eqref{prop: bracket of CLG containment}}{\subseteq} [V_1,V_{k+1}] + \mathfrak{g}^{(k+3)}.
 \end{align*}
 Thus $\dim(V_{k+2})\leq \dim([V_1,V_{k+1}])$. But $[V_1,V_{k+1}] \subseteq V_{k+2}$, and therefore $V_{k+2}= [V_1,V_{k+1}].$
\end{proof}

\section{Large and small scale geometry of sub-Finsler groups}\label{section: large and small}
We recall some basic notions regarding Gromov--Hausdorff convergence of pointed metric spaces.
\begin{definition}[Hausdorff approximating sequence]
 Let $(X_j,x_y)$ and $(Y_j,y_j)$, for $j\in \mathbb{N}$, be two sequences of pointed metric spaces. A sequence of maps $\phi_j : (X_j,x_j) \rightarrow (Y_j,y_j)$ is said to be \textit{Hausdorff approximating} if for all $R>0$ and all $\delta >0$ there exists a sequence $\epsilon_j$ such that
 \begin{itemize}
  \item $\epsilon_j \rightarrow 0$, as $j\rightarrow \infty$;
  \item $\phi_j|_{B(x_j,R)}$ is a $(1,\epsilon_j)$-quasi-isometric embedding;
  \item $\phi_j(B(x_j,R))$ is an $\epsilon_j$-net for $B(y_j,R-\delta).$
 \end{itemize}
 We remark that if $\phi_j : (X_j,x_j) \rightarrow (Y_j,y_j)$ is an Hausdorff approximating sequence of geodesic metric spaces, then $\phi_j(B(x_j,R))$ is an $\epsilon_j$-net for $B(y_j,R).$
 
\end{definition}
\begin{definition}[Gromov-Hausdorff limit]
 We say that a pointed metric space $(Y,y)$ is the \textit{Gromov-Hausdorff limit} of a sequence of pointed metric spaces $(X_j,x_j)$, with $j\in \mathbb{N}$, if there exists an Hausdorff approximating sequence $\phi_j : (X_j,x_j) \rightarrow (Y,y)$.
\end{definition}
We will need the following classical criterion, which can be found in \cite[Proposition~12.1.3]{Don25}:
\begin{prop}[Criterion for GH convergence]\label{prop: criterion_GH_conv}
 Let $d_j$ be a sequence of distances on a set $X$ that converges to a distance $d_\infty$ uniformly on $d_\infty$-bounded sets. Let $x_0 \in X$. If 
 \[\text{diam}_{d_\infty}\left(\bigcup_{j\in \mathbb{N}}B_{d_j}(x_0,R)\right) < \infty, \quad \forall R > 0,\]
 then ${\rm id}:(X,d_j,x_0) \rightarrow (X,d_\infty,x_0)$, is an Hausdorff approximating sequence.
\end{prop}
\subsection{Asymptotic Carnot group and Pansu Asymptotic Theorem}\label{subsection: asymptotic carnot group}
\begin{definition}[Asymptotic Carnot algebra and group]
The \textit{asymptotic Carnot algebra} of a simply connected $s$-step nilpotent Lie group $G$ with Lie algebra $(\mathfrak{g}, \lbrack \cdot, \cdot \rbrack)$ is the Lie algebra $(\mathfrak{g}_\infty, \lbrack \cdot, \cdot \rbrack^{(0)})$, where 
\begin{align*}
 \mathfrak{g}_\infty \coloneqq \bigoplus_{i=1}^s \mathfrak{g}^{(i)}/\mathfrak{g}^{(i+1)}
\end{align*}
and $\lbrack \cdot, \cdot \rbrack^{(0)}$ is the unique Lie bracket on $\mathfrak{g}_\infty$ such that for $v \in \mathfrak{g}^{(i)}$ and $w \in \mathfrak{g}^{(j)}$:
\begin{align*}
 \lbrack v + \mathfrak{g}^{(i+1)}, w + \mathfrak{g}^{(j+1)} \rbrack^{(0)} \coloneqq \lbrack v, w \rbrack + \mathfrak{g}^{(i+j+1)}.
\end{align*} 
The \textit{asymptotic Carnot group} $G_\infty$ of $G$ is the simply connected nilpotent Lie group with Lie algebra $(\mathfrak{g}_\infty, \lbrack \cdot, \cdot \rbrack^{(0)})$.
\end{definition}
 Notice that $\mathfrak{g}_\infty$ is stratifiable and the spaces $(\mathfrak{g}^{(i)}/\mathfrak{g}^{(i+1)})_{i=1}^s$ canonically provide a stratification on $\mathfrak{g}_\infty.$

If the group $G$ is equipped with a sub-Finsler metric $d$, then we can define a sub-Finsler metric $d_\infty$ on $G_\infty$ as follows. Suppose $d$ is associated to a polarization $\Delta$ and a norm $\norm{\cdot}$ on $\mathfrak{g}$, and denote by $\rho: \mathfrak{g} \rightarrow \mathfrak{g}/[\mathfrak{g}, \mathfrak{g}]$ the projection with kernel $[\mathfrak{g}, \mathfrak{g}]$. Since $\Delta$ is bracket generating, the restriction of $\rho$ to $\Delta_1$ is surjective. Define the \textit{Pansu limit norm} $\norm{\cdot}_\infty$ to be the norm on $\mathfrak{g}/[\mathfrak{g}, \mathfrak{g}]$ with unit ball at the origin $B_{\norm{\cdot}_\infty}(0,1) \coloneqq \rho(B_{\norm{\cdot}}(0,1) \cap \Delta_1).$ The \textit{Pansu limit metric} $d_\infty$ is the sub-Finsler metric on $G_\infty$ associated to $(\mathfrak{g}/[\mathfrak{g}, \mathfrak{g}], \norm{\cdot}_\infty)$. We remark that $(G_\infty, d_\infty)$ is a (metric) Carnot group, which, by the following theorem, is the asymptotic cone of $(G,d)$.
\begin{thm}[Pansu Asymptotic Theorem \cite{Pan83}]\label{thm: Pansu asymptotic thm}
 The asymptotic cone of a simply connected nilpotent sub-Finsler Lie group $(G,d)$ is the Carnot group $(G_\infty,d_\infty).$
\end{thm}

One can pointwise identify a simply connected nilpotent sub-Finsler Lie group with its asymptotic Carnot group and quantitatively study the speed of convergence of the metric $d$ to the metric $d_\infty$, obtaining a quantitative version of Pansu's theorem. The identification depends on the non-canonical choice of an asymptotic grading of $\mathfrak{g}.$ Let $(V_i)_{i=1}^s$ be an asymptotic grading of $\mathfrak{g}$ and let $j \in \{1,2,\ldots,s\}$. Recall that the asymptotic grading condition means that $\mathfrak{g}^{(j)} = V_j \oplus \mathfrak{g}^{(j+1)}$. Then $\pi_j \colon \mathfrak{g}^{(j)} \rightarrow V_j$ induces an isomorphism between $\mathfrak{g}^{(j)}/\mathfrak{g}^{(j+1)}$ and $V_j$ and these isomorphisms induce a vector space isomorphism $\mathfrak{g}\simeq \mathfrak{g}_\infty.$
Under this identification we can pull back the Lie bracket $\lbrack \cdot, \cdot \rbrack^{(0)}$ from $\mathfrak{g}_\infty$ to $\mathfrak{g}$.

Using exponential maps, which are global diffeomorphisms for simply connected nilpotent groups, we obtain the identifications
\begin{align*}
 G \simeq \mathfrak{g} \simeq \mathfrak{g}_\infty \simeq G_\infty.
\end{align*}
Observe that in general the identification map $G \rightarrow G_\infty$ is just a diffeomorphism, not a group isomorphism.

\subsection{Osculating Carnot group and Mitchell Tangent Theorem}\label{subsection: osculating carnot group}
Let \\ $(G,\Delta, \|\cdot\nobreak\|)$ be a sub-Finsler Lie group, not necessarily nilpotent. Denote by $[\cdot, \cdot]$ the bracket on $\mathfrak{g} \coloneqq \text{Lie}(G)$ and by $s \coloneqq \min\{n\in \mathbb{N}: \sum_{i=1}^n \ad_{\Delta}^{i-1}(\Delta) = \mathfrak{g}\}$ the step of the distribution.
\begin{definition}[Osculating Carnot algebra and group]
The \textit{osculating Carnot algebra} of $(G,\Delta)$ is the Lie algebra $(\mathfrak{g}_0, \doublebracket{\cdot}{\cdot}^{(0)})$, where, using the notation of Definition~\ref{def: adapted linear grading},
\begin{align*}
 \mathfrak{g}_0 \coloneqq \bigoplus_{i=1}^s \Delta^{[i]}/\Delta^{[i-1]},
\end{align*}
and the bracket $\doublebracket{\cdot}{\cdot}^{(0)}$ is the unique Lie bracket such that for $x\in \Delta^{[i]},\: y \in \Delta^{[j]}$:
\[\doublebracket{x+ \Delta^{[i-1]}}{y+ \Delta^{[j-1]}}^{(0)} \coloneqq [x,y] + \Delta^{[i+j-1]}.\] 
The \textit{osculating Carnot group} $G_0$ of $G$ is the simply connected nilpotent Lie group with Lie algebra $(\mathfrak{g}_0, \doublebracket{\cdot}{\cdot}^{(0)})$.
\end{definition}

The Lie algebra $\mathfrak{g}_0$ is stratifiable and $(\Delta^{[i]}/\Delta^{[i-1]})_{i=1}^s$ is a stratification, observe in addition that $\Delta^{[1]}/\Delta^{[0]} = \Delta$. Hence the pair $(\Delta, \norm{\cdot})$ defines a sub-Finsler distance $d_0$ on $G_0$, which we call the \textit{tangent metric}, and the metric group $(G_0,d_0)$ is a (metric) Carnot group. For every point $p \in G$, the group $(G_0,d_0)$ is the tangent cone of $(G,d)$ at $p$:
\begin{thm}[Mitchell Tangent Theorem]\label{thm: Mitchell tangent thm}
 Let $(G,d)$ be a sub-Finsler Lie group and $p \in G$. Then the tangent cone of $G$ at $p$ is the osculating Carnot group $(G_0,d_0).$
\end{thm}
In fact, we will provide a quantitative statement of Theorem~\ref{thm: Mitchell tangent thm}: for all $R>0$ the Gromov--Hausdorff distance between the $R$-balls admits a bound of the form
\begin{equation}\label{eq: thm Mitchell tangent thm}
    d_{{\rm GH}}\left(\left(B_{\frac{1}{\epsilon}d}(1,R), \frac{1}{\epsilon}d\right), \left(B_{d_0}(1,R), d_0\right)  \right)= O_R(\epsilon^{\alpha_0/\beta}).
\end{equation}

Analogously to the identification of a nilpotent Lie algebra with its asymptotic Carnot algebra, one can locally identify each sub-Finsler Lie group with its osculating Carnot group, provided a choice of tangent grading, to quantitatively study the speed of convergence of the metric $d$ to the metric $d_0$. Let $(W_i)_{i=1}^s$ be a tangent grading of $(\mathfrak{g}, \Delta)$. Then $\Delta^{[i]}/\Delta^{[i-1]}$ is isomorphic as vector space to $W_i$ for every $i\in \{1,\ldots,s\}.$ These isomorphisms induce a vector space isomorphism $\mathfrak{g}\simeq \mathfrak{g}_0$, and under this identification we can pull back the bracket $\llbracket \cdot, \cdot \rrbracket^{(0)}$ to a bracket on $\mathfrak{g}$. Next, using the fact that the exponential map is a local diffeomorphism around the origin, we identify a neighborhood of the identity in $G$ with a neighborhood of the identity in $G_0$:
\begin{align*}
 G \overset{\text{locally}}{\simeq} \mathfrak{g} \simeq \mathfrak{g}_0 \simeq G_0.
\end{align*}
Again, we remark that in general this is not a local group isomorphism.

There is a version of the Ball-Box Theorem in terms of the tangent metric. Let $(G,\Delta, \norm{\cdot})$ be a sub-Finsler Lie group and $(W_i)_{i=1}^s$ be a tangent grading. We identify a neighborhood of the identity in $G$ with a neighborhood of the origin in $\mathfrak{g}$ equipped with the Dynkin product. As a consequence of a general Ball-Box Theorem, as in \cite[Theorem~12.5.3]{Don25}, the functions $d_0(0, \cdot)$ and $d(0, \cdot)$ on $\mathfrak{g}$ are locally biLipschitz comparable, i.e., there are a neighborhood $\Omega$ of $0$ in $\mathfrak{g}$ and a constant $C>0$ such that
 \begin{equation}\label{eq: BB for Lie gps}
  \frac{1}{C}d_0(0,p)\leq d(0,p) \leq C d_0(0,p), \quad \forall p \in \Omega.
 \end{equation}

\subsection{One-parameter families of sub-Finsler Lie group structures}
Let $G$ be a Lie group and let $(D_i)_{i=1}^k$ be a grading of $\mathfrak{g} \coloneqq \text{Lie}(G)$. Using the family of dilations associated to the grading, we can define two one-parameter families of Lie brackets on $\mathfrak{g}$. Let $\epsilon\in \mathbb{R}$, $\epsilon \neq 0$, then we define
\begin{align}\nonumber
 [x,y]^{(\epsilon)} &\coloneqq \delta_\epsilon[\delta_\epsilon^{-1}x,\delta_\epsilon^{-1}y],\\
  \nonumber \llbracket x,y \rrbracket^{(\epsilon)} &\coloneqq [x,y]^{(1/\epsilon)}.
\end{align}
\begin{lem}\label{lem: brackets extend to eps=0}
 Let $G$ be a Lie group and let $(D_i)_{i=1}^s$ be a grading. Then
 \begin{enumerate}
  \item If $(D_i)_{i=1}^s$ is a tangent grading, then the brackets $\nonumber \llbracket \cdot,\cdot \rrbracket^{(\epsilon)}$ depend polynomially on $\epsilon$, and therefore extend to $\epsilon=0$. Moreover
 \[\lim_{\epsilon \rightarrow 0}\llbracket x,y \rrbracket^{(\epsilon)} = \llbracket x,y\rrbracket^{(0)}, \quad \forall x,y \in \mathfrak{g},\]
 where $\nonumber \llbracket \cdot,\cdot \rrbracket^{(0)}$ is the Lie bracket defined in Section~\ref{subsection: osculating carnot group}.
 \item If $G$ is nilpotent and $(D_i)_{i=1}^s$ is an asymptotic grading, then the brackets $[\cdot,\cdot]^{(\epsilon)}$ depend polynomially on $\epsilon$, and therefore extend to $\epsilon=0$. Moreover
 \[\lim_{\epsilon \rightarrow 0}[x,y]^{(\epsilon)} = [x,y]^{(0)}, \quad \forall x,y \in \mathfrak{g},\]
 where $[\cdot, \cdot]^{(0)}$ is the Lie bracket defined in Section~\ref{subsection: asymptotic carnot group}.
 \end{enumerate}
\end{lem}
\begin{proof} We only prove (1), claim (2) follows from a similar argument. Let $x,y \in \mathfrak{g}$, then
 \begin{align*}
  \llbracket x,y \rrbracket^{(\epsilon)} &= \left\llbracket \sum_{i=1}^s(x)_i,\sum_{i=1}^s(y)_i \right\rrbracket^{(\epsilon)} \\
  &= \sum_{i,j} \delta_\epsilon^{-1}[\delta_\epsilon (x)_i,\delta_\epsilon (y)_j]\\
  &\overset{\ref{prop: properties of linear gradings}\eqref{prop: bracket of ALG containment}}{=} \sum_{i,j}\sum_{k\leq i+j}\delta_\epsilon^{-1}([\epsilon^i(x)_i,\epsilon^j(y)_j])_k\\
  &=\sum_{i,j}\sum_{k\leq i+j}\epsilon^{i+j-k}([(x)_i,(y)_i])_k.
 \end{align*}
 So $\llbracket \cdot,\cdot \rrbracket^{(\epsilon)}$ is polynomial in $\epsilon$. Next, let $i,j \in \{1,\ldots,s\}$ and $w\in D_i, z\in D_j$. Then
 \begin{align*}
  \llbracket w,z \rrbracket^{(\epsilon)} \overset{\ref{prop: properties of linear gradings}\eqref{prop: bracket of ALG containment}}{=} \sum_{k\leq i+j}\delta_\epsilon^{-1}([\delta_\epsilon w, \delta_\epsilon z])_k = \sum_{k\leq i+j}\epsilon^{i+j-k}([w,z])_k \overset{\epsilon \rightarrow 0}{\longrightarrow} ([w,z])_{i+j} = \llbracket w,z \rrbracket^{(0)}.
 \end{align*}
 
\end{proof}
\begin{definition} For $\epsilon \in \mathbb{R}$, we denote with $*_\epsilon$ (respectively $\cdot_\epsilon$) the Dynkin product associated with $[\cdot, \cdot]^{(\epsilon)}$ (respectively $\llbracket\cdot, \cdot\rrbracket^{(\epsilon)}$), whenever the Dynkin series converges. 
\end{definition}
 Recall that if the Lie algebra $\mathfrak{g}$ is nilpotent, the BCH formula is polynomial and thus the Dynkin product is a (global) group product. The next lemma shows that the maps $\epsilon \mapsto *_\epsilon$ and $\epsilon \mapsto \cdot_\epsilon$ are analytic under suitable conditions.
\begin{lem}[{\cite[Lemma~12.2.4]{Don25}}]\label{lem: analyticity of product}
 Let $V$ be a vector space, $\Lambda \subseteq \mathbb{R}$ open, and $([\cdot, \cdot]_\lambda)_{\lambda \in \Lambda}$ a family of Lie brackets on $V$. Assume that the map $\lambda \mapsto [\cdot, \cdot]_\lambda$ is analytic. Denote by $\star_\lambda$ the Dynkin product associated with $[\cdot, \cdot]_\lambda$. Then, the maps
 \begin{align}\nonumber
  (\lambda,x,y) \mapsto x \star_\lambda y \quad \text{and} \quad (\lambda,x,y) \mapsto \left.\frac{\dd}{\dd s} x \star_\lambda (sy)\right|_{s=0}
 \end{align}
 are analytic on some open subset of $\Lambda \times V \times V$ that contains $\Lambda \times \{0\}\times \{0\}.$
\end{lem}

We also introduce two one-parameter families of sub-Finsler metrics. 
\subsubsection{Contracted metrics on simply connected nilpotent sub-Finsler Lie groups}
Suppose that $(G,d)$ is a simply connected nilpotent sub-Finsler Lie group associated to $(\Delta, \norm{\cdot})$, endowed with an asymptotic grading $(V_i)_{i=1}^s$. We define the $\textit{contracted merics}$ as 
\begin{align}\label{def: contracted metric}
 \rho_\epsilon(x,y) \coloneqq \begin{cases}
  |\epsilon|d(\delta_\epsilon^{-1}x, \delta_\epsilon^{-1}y),\; \epsilon \neq 0,\\
  d_\infty(x,y), \; \epsilon = 0.
 \end{cases}\quad,\quad \forall x,y \in G.
\end{align}
The functions $\rho_\epsilon$ are in fact metrics and, for $\epsilon \neq 0$, the metric $\rho_\epsilon$ is the $*_\epsilon$-left-invariant sub-Finsler metric on $G$ associated to 
\begin{equation}\label{def: contracted structure}
 \Delta^{(\epsilon)} \coloneqq \delta_\epsilon \Delta, \quad \norm{v}^{(\epsilon)} \coloneqq \epsilon \norm{\delta_\epsilon^{-1}v}, \;\forall v \in \Delta^{(\epsilon)}.
\end{equation}
For completeness, we also consider the case $\epsilon = 0$ and define 
\begin{equation}\label{def: contracted structure for eps = 0}
 \Delta^{(0)} \coloneqq V_1, \text{ and } \norm{\cdot}^{(0)}\coloneqq \norm{\cdot}_\infty.
\end{equation}
\begin{remark}\label{rem: dilations are isometries asymptotic}
The definition of contracted metrics implies that, for $\epsilon \neq 0$, the dilation $\delta_\epsilon : (G, |\epsilon| d) \rightarrow (G, \rho_\epsilon)$ is an isometry.
\end{remark}

\begin{prop}\label{prop: if G carnot the the contracted metric are equal}
 If $(G,d)$ is a Carnot group and $(\rho_\epsilon)_{\epsilon}$ are the contracted metrics defined in \eqref{def: contracted metric}, then 
 \begin{equation}\label{eq: prop: if G carnot the the contracted metric are equal}
  \rho_\epsilon = d, \quad \forall \epsilon \in [0,1].
 \end{equation}
\end{prop}
\begin{proof}
 The dilations $(\delta_\epsilon)$ associated to the stratification $(V_i)_{i=1}^s$ are Lie algebra isomorphisms of $(\mathfrak{g},[\cdot,\cdot])$, for $\epsilon\neq 0$. Hence, 
 $$[\cdot,\cdot]^{(\epsilon)} = [\cdot, \cdot], \quad \forall \epsilon \neq 0.$$
 It follows that $*_\epsilon = *_1$, $\forall \epsilon \neq 0$. Moreover, since the horizontal subspace of $d$ is $V_1$, we also have 
 $$\Delta^{(\epsilon)} = V_1, \quad \norm{\cdot}^{(\epsilon)} = \norm{\cdot}, \quad \forall \epsilon \in (0,1].$$
 Thus,
 $$\rho_\epsilon = d, \quad \forall \epsilon \neq 0.$$
 Next, after identifying $\mathfrak{g}$ with $\mathfrak{g}_\infty$, we also obtain
 $$[\cdot, \cdot]^{(0)} = [\cdot, \cdot], \quad \mathfrak{g}/[\mathfrak{g}, \mathfrak{g}] = V_1, \quad \norm{\cdot}_\infty = \norm{\cdot},$$
 and thus $\rho_0 = d.$
\end{proof} 
The next result shows that the distances from the origin of the contracted metrics $\rho_\epsilon$ and $\rho_0$ are equivalent, up to a universal multiplicative constant, and an additive constant that depends linearly on $\epsilon$.
\begin{thm}[Guivarc'h Theorem in terms of contracted metrics {\cite[Corollary~12.4.4]{Don25}}]\label{thm: Guivarch}
 Let $G$ be a simply connected nilpotent sub-Finsler Lie group equipped with an asymptotic grading $(V_i)_{i=1}^s$. Then, there is a constant $C\geq 1$ such that 
 \begin{align}\label{eq: Guivarch equation}
  \frac{1}{C}\rho_0(1,x)-C\epsilon \leq \rho_\epsilon(1,x) \leq C\rho_0(1,x) + C\epsilon, \quad \forall x\in G, \; \forall \epsilon \geq 0.
 \end{align}
\end{thm}

We can lift $\rho_0$-geodesics to $\rho_\epsilon$-admissible curves (of the same length) using the fact that the map $\pi_1$ is a submetry, as we now explain. A \textit{submetry} $f:X \rightarrow Y$ between metric spaces is a map such that 
\[f(\overline{B}_{d_X}(x,r)) = \overline{B}{_{d_Y}}(f(x),r), \quad \forall x\in X,\; \forall r\geq 0.\]
\begin{prop}\label{prop: pi_1 is a submetry and homo}
 Let $\mathfrak{g}$ be a nilpotent Lie algebra. Then the projection map $\pi_1 : (\mathfrak{g}, *_\epsilon, \norm{\cdot}^{(\epsilon)}) \rightarrow (V_1, +,\norm{\cdot}_\infty)$ is a group homomorphism and a submetry.
\end{prop}
\begin{proof} The projection map $\pi_1$ is a group homomorphism since 
\[(x *_\epsilon y)_1 = (x)_1+(y)_1, \quad \forall x,y \in \mathfrak{g}.\]

Regarding the claim that $\pi_1$ is a submetry, we use the definitions of the norms $\norm{\cdot}_\infty$ and $\norm{\cdot}^{(\epsilon)}$ to calculate
 \begin{align*}
  \pi_1 \left( B_{\norm{\cdot}^{(\epsilon)}}(0,1)\right) &=\pi_1 \left( \delta_\epsilon B_{\norm{\cdot}^{(1)}}(0,1/\epsilon)\right) \\
  &= \epsilon\pi_1 \left( B_{\norm{\cdot}^{(1)}}(0,1/\epsilon)\right) \\
  &= \epsilon B_{\norm{\cdot}_\infty}(0,1/\epsilon)\\
  &=B_{\norm{\cdot}_\infty}(0,1).
 \end{align*}
\end{proof}
It is a standard fact in metric geometry that submetries lift $1$-Lipschitz curves from their target space to their domain:
\begin{prop}\label{prop: lift of curves}
 Let $f: X \rightarrow Y$ be a submetry between metric spaces. Assume that $X$ is boundedly compact. Then for every $1$-Lipschitz curve $\gamma:[0, T] \rightarrow Y$ and every $x \in f^{-1}(\gamma(0))$, there exists a $1$-Lipschitz curve $\tilde{\gamma}:[0,T] \rightarrow X$ such that $\tilde{\gamma}(0)=x$ and $f \circ \tilde{\gamma} = \gamma.$
\end{prop}
\begin{prop}
    Let $G$ be a simply connected nilpotent sub-Finsler Lie group equipped with an asymptotic grading $(V_1,\ldots, V_s)$ such that the sub-Finsler structure associated to the contracted metrics $\rho_\epsilon$ are $(\Delta^{(\epsilon)}, \norm{\cdot}^{(\epsilon)})$. Then for every Euclidean metric $\norm{\cdot}_{\rm E}$ there exists a constant $C>0$ depending only on $G, \Delta, \norm{\cdot}, \norm{\cdot}_{\rm E}$, and $(V_1,\ldots, V_s)$ such that we have 
\begin{equation}\label{eq: upper bound of euclidean norm of layer wrt to norm eps}
 \norm{(u)_k}_{\rm E} \leq C\epsilon^{k-1} \norm{(u)_k}^{(\epsilon)}, \quad \forall k \in \{1,\ldots,s\}, \; \forall \epsilon \in [0,1], \; \forall u \in \Delta^{(\epsilon)},
\end{equation}
using the convention that $0^0=1$, and 
\begin{equation}\label{eq: upper bound of euclidean norm wrt to norm eps}
 \norm{u}_{\rm E} \leq C \norm{u}^{(\epsilon)}, \quad \forall \epsilon \in [0,1] , \; \forall u \in \Delta^{(\epsilon)}.
\end{equation}
\end{prop}

\begin{proof}
    Regarding \eqref{eq: upper bound of euclidean norm of layer wrt to norm eps}, let $\epsilon \in (0,1]$ and $u \in \Delta^{(\epsilon)}$. Since $\norm{\cdot}^{(1)}$ and $\norm{\cdot}_{\rm E}$ are norms on $\mathfrak{g}$, there exists a constant $C\leq 1$ such that $\norm{\cdot}^{(1)} \geq C\norm{\cdot}_{\rm E}$. Thus
\begin{align*}
 \norm{(u)_k}^{(\epsilon)} &\overset{\text{def}}{=} \epsilon \norm{\delta_\epsilon^{-1} (u)_k}^{(1)}\\
 &\geq C \epsilon \norm{\delta_\epsilon^{-1} (u)_k}_{\rm E}\\
 &=C\epsilon^{1-k} \norm{ (u)_k}_{\rm E}.
\end{align*}
For $\epsilon = 0$, since $\Delta^{(0)} \overset{\eqref{def: contracted structure for eps = 0}}{=}V_1$ and $\left.\norm{\cdot}_{\rm E}\right|_{V_1}$ is a norm on $V_1$, there exists a constant $C'$ so that
\begin{equation}\label{eq: upper bound on euclidean norm with 0 norm}
\norm{u}_{\rm E} \leq C' \norm{u}^{(0)} , \quad \forall u \in V_1.
\end{equation}

Regarding \eqref{eq: upper bound of euclidean norm wrt to norm eps}, let $\epsilon \neq 0$ and $u \in \Delta^{(\epsilon)}$. Then
\begin{align*}
 \norm{u}^{(\epsilon)} &\overset{\text{def}}{=} \epsilon \norm{\delta_\epsilon^{-1}u}^{(1)}\\
 &\geq C\epsilon\norm{\delta_\epsilon^{-1}u}_{\rm E}\\
 &=C \norm{(u)_1 + \epsilon^{-1}(u)_2+\cdots+\epsilon^{-s+1}(u)_s}_{\rm E}\\
 &\geq C\norm{(u)_1 + (u)_2+\cdots+(u)_s}_{\rm E}\\
 &= C\norm{u}_{\rm E}.
\end{align*}
If $\epsilon = 0$, then by~\eqref{eq: upper bound on euclidean norm with 0 norm}, $\norm{u}_{\rm E} \leq C' \norm{u}^{(0)} ,\; \forall u \in V_1 = \Delta^{(0)}.$
\end{proof}

\subsubsection{Dilated metrics on sub-Finsler Lie groups}
Suppose that $(G,d)$ is a sub-Finsler Lie group (not necessarily nilpotent) with distance associated to $(\Delta, \|\cdot \nobreak\|)$, endowed with a tangent grading $(W_i)_{i=1}^s$. Let $r>0$ be small enough so that the exponential map restricted to $B_{d_0}(0,r)$ is a diffeomorphism. We define the $\textit{dilated metrics}$ as
\begin{align*}
 d_\epsilon(x,y) \coloneqq \begin{cases}
  \frac{1}{|\epsilon|}d(\delta_\epsilon x, \delta_\epsilon y),\; \epsilon \neq 0,\\
  d_0(x,y), \; \epsilon = 0.
 \end{cases}\quad,\quad \forall x,y \in B_{d_0}(0,r).
\end{align*}
Locally, for $\epsilon \neq 0$, the metric $d_\epsilon$ is the $\cdot_\epsilon$-left-invariant sub-Finsler metric on a neighborhood of the identity in $G$ associated to $(\Delta, \norm{\cdot}).$
\begin{remark}\label{rem: dilations are isometries tangent}
The definition of dilated metrics implies that, for $\epsilon \neq 0$, the dilation $\delta_{\epsilon^{-1}} : (B_{d_0}(0,|\epsilon| r), \frac{1}{|\epsilon|} d) \rightarrow (B_{d_0}(0,r), d_\epsilon)$ is an isometry.
\end{remark}
Applying dilations $\delta_\epsilon$ to~\eqref{eq: BB for Lie gps} we obtain some equicontinuity property of the family of distances: there are $r>0,C\geq 1$ such that
\begin{align}\label{eq: dilated metric are bilipschitz}
  \frac{1}{C}d_0(0,p) \leq d_\epsilon(0,p) \leq C d_0(0,p), \quad \forall p\in B_{d_0}(0, r), \; \forall \epsilon \in (0,1].
 \end{align}
 
Moreover, since $d_\epsilon$ are geodesic distances, we further have
 \begin{equation}\label{cor: equicontinuity of distances}
  B_{d_\epsilon}(0,r'/C) \subseteq B_{d_0}(0,r'), \quad \forall r'\leq r, \forall \epsilon \in (0,1].
 \end{equation}
\subsection{Carnot quotient ideals} In the proofs of Theorem~\ref{thm: quantitative Pansu} and Theorem~\ref{thm: quantitative Mitchell}, we need to bound the norm of the difference of two curves $\gamma_0$ and $\gamma_\epsilon$. To this end, we introduce a class of ideals $\mathfrak{i}$ of $\mathfrak{g}$, which we call Carnot quotient ideals, such that the projection of $\gamma_0$ and $\gamma_\epsilon$ on the quotient $\mathfrak{g}/\mathfrak{i}$ coincide. Thus their difference $v \coloneqq \gamma_0(1)^{-1}*_\epsilon\gamma_\epsilon(1)$ lies in $\mathfrak{i}$. We then use this fact to obtain a bound on the norm of $v$, better than the bound one would obtain by directly applying the Ball-Box Theorem (see Lemma~\ref{lemma: uniform bound on contracted metrics} and Lemma~\ref{lemma: uniform bound on dilated metrics}).

\begin{definition} \label{def: Carnot quotient ideal} 
 Let $(G,\Delta)$ be a polarized group and let $(D_i)_{i=1}^s$ be a grading.
 We say that a subspace $\mathfrak{i}\subseteq \mathfrak{g}$ is a \textit{Carnot quotient ideal of} $(G,\Delta)$ \textit{with respect to the grading $(D_i)_{i=1}^s$} if 
 \begin{enumerate}
  \item \label{def: CQI 1 ideal}$\mathfrak{i}$ is an ideal of $\mathfrak{g}$.
  \item \label{def: CQI 2 stratification}$\mathfrak{g}/\mathfrak{i}$ is stratified by $(\pi(D_j))_{1\leq j\leq s}$, where $\pi:\mathfrak{g}\rightarrow \mathfrak{g}/\mathfrak{i}$ is the quotient map.
  \item \label{def: CQI projection} 
  $\Delta \subseteq D_1 + \mathfrak{i}.$
 \end{enumerate}
\end{definition}

We claim that condition~\eqref{def: CQI projection} is equivalent to
\begin{equation}\label{def: CQI: first equivalence of condition 3}
 \pi(u) = \pi \circ \pi_1(u), \quad \forall u \in \Delta.
\end{equation}
Indeed, let $u=\sum_{i=1}^s(u)_i \in \Delta$. Then
\begin{align*}
 \pi(u) = \pi(\pi_1(u))&\iff \pi\left(\sum_{i=2}^s(u)_i\right) = 0 \iff \sum_{i=2}^s(u)_i\in \mathfrak{i.}
\end{align*}

Moreover, we claim that condition~\eqref{def: CQI projection} is also equivalent to 
\begin{equation}\label{def: CQI: second equivalence of condition 3}
 \Delta + \mathfrak{i} = D_1 + \mathfrak{i}.
\end{equation}
Indeed, if $\Delta + \mathfrak{i} = D_1 + \mathfrak{i}$, then clearly $\Delta \subseteq D_1 + \mathfrak{i}$. Now assume that $\Delta \subseteq D_1 + \mathfrak{i}$, and hence $\Delta + \mathfrak{i} \subseteq D_1 + \mathfrak{i}$. The quotient map $\pi : \mathfrak{g} \rightarrow \mathfrak{g}/\mathfrak{i}$ is a Lie algebra homomorphism, and thus, since $\Delta$ is bracket generating, $\pi(\Delta) = \Delta+\mathfrak{i}$ is also bracket generating. Therefore, $\Delta + \mathfrak{i}$ is a bracket generating subspace of the first stratum $D_1 + \mathfrak{i}$ of the stratified Lie algebra $\mathfrak{g}/\mathfrak{i}$, so it must coincide with $D_1+\mathfrak{i}.$

Note in addition that if $(D_i)_{i=1}^s$ is a tangent grading, condition~\eqref{def: CQI projection} is satisfied for any ideal, since $\Delta = D_1$ and $\pi_1|_{D_1} = \id$.
\begin{example}\label{example: CQI}
\begin{enumerate}
 \item \label{example: CQI 1}Let $(G,\Delta, \norm{\cdot})$ be a sub-Finsler Lie group, and let $(D_i)_{i=1}^s$ be either an asymptotic or a tangent grading. Denote by $d$ the sub-Finsler metric on $G$. Then
 \[(G,d) \text{ is a Carnot group} \iff \{0\} \text{ is a Carnot quotient ideal.}\]
 Indeed, if $(G,d)$ is a Carnot group, then clearly $\{0\}$ is a Carnot quotient ideal. Conversely, if $\{0\}$ is a Carnot quotient ideal, the grading $(D_i)_{i=1}^s$ is a stratification of $\text{Lie}(G)$, and by \eqref{def: CQI: second equivalence of condition 3} we have $\Delta = D_1$.
 \item For every nilpotent Lie algebra $\mathfrak{g}$ with asymptotic grading $(V_1)_{i=1}^s$, the ideal $V_2 \oplus\cdots\oplus V_s$ is a Carnot quotient ideal, for every polarization $\Delta$.
 \item Consider the non-stratifiable Lie algebra $N_{5,2,2}$ from \cite{LDT22}. It is the Lie algebra $\mathbb{R}^5=\text{span}\{e_1,\ldots,e_5\}$ with non-trivial brackets $[e_1,e_2]=e_4$, $[e_1,e_4]=[e_2,e_3]=e_5$. Consider the asymptotic grading $V_1 = \text{span}\{e_1,e_2,e_3\}$, $V_2 = \text{span}\{e_4\}$, $V_3 = \text{span}\{e_5\}$ and distribution at the identity $\Delta = V_1$. Then $\mathfrak{i} = \text{span}\{e_5\}$ is a Carnot quotient ideal.
\end{enumerate}
\end{example}

\begin{definition}\label{def: constant beta}
 Let $(D_i)_{i=1}^s$ be a grading of the Lie algebra of a polarized Lie group $(G,\Delta)$ such that $\text{step}(\Delta) \leq s$. We define 
\[\beta:= \min\{k\in \mathbb Z\,:\, \exists \mathfrak{i} \subseteq D_{\leq k} \, \text{Carnot quotient ideal}
 \}\in \{0,\ldots,s\}.\]
\end{definition}
Note that by Example~\ref{example: CQI}~\eqref{example: CQI 1}, the metric space $(G,d)$ is a Carnot group if and only if $\beta=0$.
\begin{example}\label{example: beta}
 \begin{enumerate}
  \item For every Riemannian group $G$ with Lie algebra $\mathfrak{g}$ and tangent grading $W_1 = \mathfrak{g}$ we have $\beta\leq1$, since $\mathfrak{g}$ is a Carnot quotient ideal of itself. 
  \item\label{example: CQI2} An example where $1 < \beta < s$. Consider the Lie algebra $N_{5,2,2} = \text{span}\{e_1,e_2,e_3,e_4,e_5\}$ and the \textit{filiform algebra of dimension 5}: $N_{5,2,1} = \text{span}\{e'_1,e'_2,e'_3,e'_4,e'_5\}$ from \cite{LDT22}, which is the $5$-dimensional Lie algebra with only non-trivial brackets $[e'_1,e'_2] = e'_3, \;[e'_1,e'_3] = e'_4,\; [e'_1,e'_4] = e'_5$. Let $\mathfrak{g} = N_{5,1} \times N_{5,2,1}$ be the product Lie algebra. Equip $\mathfrak{g}$ with the asymptotic grading $(V_i)_{i=1}^4$ defined by
  \begin{align*}
   V_1 &\coloneqq \text{span}\{e_1,e_2,e_3,e'_1,e'_2\},\\
   V_2 &\coloneqq \text{span}\{e_4,e'_3\},\\
   V_3 &\coloneqq \text{span}\{e_5,e'_4\},\\
   V_4 &\coloneqq \text{span}\{e'_5\},
  \end{align*}
  and the bracket generating subspace $\Delta\coloneqq N_{5,2,2} + V_1$. Then $\mathfrak{i} \coloneqq N_{5,2,2}$ is a Carnot quotient ideal. Indeed, $\mathfrak{i}$ is obviously an ideal of $\mathfrak{g}$, and $\mathfrak{g}/\mathfrak{i} \simeq N_{5,2,1}$ is stratified by $(V_i + \mathfrak{i})_{i=1}^4$. Therefore $\beta \leq 3$.
  Moreover, every Carnot quotient ideal $\mathfrak{j}$ of $(\mathfrak{g}, \Delta)$ with respect to $(V_i)_{i=1}^4$ must contain $e_5$, otherwise $\mathfrak{g}/\mathfrak{j}$ would not be stratifiable (since it would contain $N_{5,2,2}$ as a subalgebra). Hence $\beta = 3$.
 \end{enumerate}
\end{example}

We record some properties of Carnot quotient ideals.
\begin{lem}[Properties of Carnot quotient ideals associated to asymptotic gradings]\label{lem: properties of CQI associated to CLG}
  Let $(\mathfrak{g}, \Delta)$ be a nilpotent polarized Lie algebra and $(V_i)_i$ an asymptotic grading. Let $\mathfrak{i} \triangleleft \mathfrak{g}$ be a Carnot quotient ideal with respect to $(V_i)_i$. Then
 \begin{enumerate}
  \item $\mathfrak{i}$ is dilation invariant.
  \item \label{lem: properties of CQI associated to CLG 2}$\mathfrak{i} \triangleleft (\mathfrak{g},[\cdot, \cdot]^{(\epsilon)})$ is an ideal, for every $\epsilon \in [0,1].$ 
  \item \label{projection of distribution is the same}$\pi|_{\Delta^{(\epsilon)}} = \pi \circ \pi_1|_{\Delta^{(\epsilon)}}$, for every $\epsilon \in (0,1]$.
  \item $(\mathfrak{g},[\cdot,\cdot]^{(\epsilon)})/\mathfrak{i}$ is stratified by $(\pi(V_j))_{1\leq j\leq s}$, for every $\epsilon \in [0,1]$.
  \item \label{lemma equal brackets}Denote by $[\cdot, \cdot]^{(\epsilon, \mathfrak{i})}$ the quotient bracket on $(\mathfrak{g},[\cdot,\cdot]^{(\epsilon)})/\mathfrak{i}$. Then \[[\cdot,\cdot]^{(\epsilon,\mathfrak{i})} = [\cdot,\cdot]^{(\tau,\mathfrak{i})}, \quad \forall \epsilon,\tau \in [0,1].\]
  \item \label{pi homomorphism} For $\epsilon \in [0,1],$ equip $(\mathfrak{g},[\cdot,\cdot]^{(\epsilon)})$ and $(\mathfrak{g}/\mathfrak{i}, [\cdot,\cdot]^{(\epsilon,\mathfrak{i})})$ with the respective Dynkin products $*_\epsilon$ and $*_{\epsilon,\mathfrak{i}}$. Then $\pi:(\mathfrak{g}, *_{\epsilon}) \rightarrow (\mathfrak{g}/\mathfrak{i}, *_{\epsilon,\mathfrak{i}})$ is a Lie group homomorphism.
 \end{enumerate}
\end{lem}
\begin{proof}
 (1). For $\tau >0$, denote by $\overline{\delta_\tau}: \mathfrak{g}/\mathfrak{i} \rightarrow \mathfrak{g}/\mathfrak{i}$ the dilation with respect to $(\pi(V_j))_{j=1}^s$ on the quotient. We claim that
 \begin{align*}
  \overline{\delta_\tau} \circ \pi = \pi \circ \delta_\tau, \quad \forall \tau \in \mathbb{R}_{> 0}.
 \end{align*}
 Indeed, for $v\in V_j$, we have
 \begin{align*}
  \pi(\delta_\tau v)=\pi(\tau^jv)=\tau^j\pi(v)=\overline{\delta_\tau} (\pi(v)),
 \end{align*}
since $\overline{\delta_\tau}|_{\pi(V_j)}$ is a multiplication by $\tau^j$. The claim follows by linearity. Consequently
\[\{0\} = \overline{\delta_\tau}(\pi(\mathfrak{i))=\pi(\delta_\tau(\mathfrak{i))}}.\]
Thus $\delta_\tau(\mathfrak{i})\subseteq \ker\pi = \mathfrak{i}$.

(2). For $\epsilon>0$, it follows from the dilation invariance of $\mathfrak{i}$ that
\begin{align*}
 [\mathfrak{g}, \mathfrak{i}]^{(\epsilon)} &\overset{\text{def}}{=} \delta_\epsilon[\delta_{\epsilon^{-1}}\mathfrak{g}, \delta_{\epsilon^{-1}}\mathfrak{i}] 
 \subseteq \delta_\epsilon[\mathfrak{g}, \mathfrak{i}]
 \subseteq \delta_\epsilon \mathfrak{i}
 \subseteq \mathfrak{i}.
\end{align*}
By Lemma~\ref{lem: brackets extend to eps=0} we also obtain the case $\epsilon = 0.$

(3). Clearly, for every $\epsilon >0$ we have $\pi_1 \circ \delta_\epsilon = \delta_\epsilon \circ \pi_1$. Therefore, for $u \in \Delta$, we have 
\begin{align*}
 \pi \circ \pi_1 \circ \delta_\epsilon (u) &= \pi \circ \delta_\epsilon \circ \pi_1 (u) \\
 &= \overline{\delta_\epsilon} \circ \pi \circ \pi_1 (u) \\
 &\overset{\eqref{def: CQI: first equivalence of condition 3}}{=} \overline{\delta_\epsilon} \circ \pi (u) \\
 &= \pi \circ \delta_\epsilon(u),
\end{align*}
that is, $\pi|_{\Delta^{(\epsilon)}} = \pi \circ \pi_1|_{\Delta^{(\epsilon)}}$.

(4). For $\epsilon=0$, it is clear. For $\epsilon>0$, using that $\mathfrak{g}/\mathfrak{i}$ is stratified by $(\pi(V_j))_{j=1}^s$ and thus $\pi([V_1,V_k]) = \pi(V_{k+1})$ we get
\begin{align*}
 [\pi(V_1),\pi(V_k)]^{(\epsilon, \mathfrak{i})} &= \pi([V_1,V_k]^{(\epsilon)}) \\
 &=\pi(\delta_\epsilon[\delta_{\epsilon^{-1}}V_1,\delta_{\epsilon^{-1}}V_k]) \\
 &=\overline{\delta_\epsilon}(\pi([V_1,V_k])) \\
 &= \overline{\delta_\epsilon}(\pi(V_{k+1}))\\
 &= \pi(V_{k+1}).
\end{align*}

(5). Let $\epsilon \in (0,1]$ and $v\in V_i, w\in V_j$. Then
\begin{align*}
 [\pi(v),\pi(w)]^{(\epsilon,\mathfrak{i})} &= \pi([v,w]^{(\epsilon)}) \\
 &= ([v,w]^{(\epsilon)})_{i+j} + \mathfrak{i} \\
 &= ([v,w])_{i+j} + \mathfrak{i} \\
 &=[v,w]^{(0)} + \mathfrak{i} \\
 &=[\pi(v),\pi(w)]^{(0,\mathfrak{i})},
\end{align*}
where in the second equality we have used that $[\pi(V_1),\pi(V_k)]^{(\epsilon, \mathfrak{i})} = \pi(V_{i+j}).$

(6). Both exponential maps $\exp_{\mathfrak{g}}: \mathfrak{g} \rightarrow \mathfrak{g}$ and $\exp_{\mathfrak{g}/\mathfrak{i}} : \mathfrak{g}/\mathfrak{i} \rightarrow \mathfrak{g}/\mathfrak{i}$ are global diffeomorphisms and are such that $\exp_{\mathfrak{g}} = \text{id}_\mathfrak{g}$ and $\exp_{\mathfrak{g}/\mathfrak{i}} = \text{id}_{\mathfrak{g}/\mathfrak{i}}$. Since $\mathfrak{g}$ is simply connected, there exists a unique Lie group homomorphism $\rho:\mathfrak{g} \rightarrow \mathfrak{g}/\mathfrak{i}$ with $\rho_* = \pi$. We have
\[\rho = \exp_{\mathfrak{g}/\mathfrak{i}} \circ \pi \circ \exp_{\mathfrak{g}}^{-1} = \pi.\]
\end{proof}
Similar properties hold if we replace the asymptotic grading with a tangent grading and the brackets $[\cdot, \cdot]^{(\epsilon)}$ by brackets $\llbracket \cdot, \cdot \rrbracket^{(\epsilon)}$:
\begin{lem}[Properties of Carnot quotient ideals associated to tangent gradings]\label{lem: properties of CQI associated to ALG}
  Let $(\mathfrak{g}, \Delta)$ be a polarized Lie algebra and $(W_i)_i$ a tangent grading with respect to $\Delta$. Let $\mathfrak{i} \triangleleft \mathfrak{g}$ be a Carnot quotient ideal with respect to $(W_i)_i$. Then
 \begin{enumerate}
  \item $\mathfrak{i}$ is dilation invariant.
  \item \label{lem: properties of CQI associated to ALG 2} $\mathfrak{i} \triangleleft (\mathfrak{g},\llbracket \cdot, \cdot\rrbracket^{(\epsilon)})$ is an ideal, for every $\epsilon \in [0,1].$ 
  \item $(\mathfrak{g},\llbracket \cdot,\cdot \rrbracket^{(\epsilon)})/\mathfrak{i}$ is stratified by $(\pi(W_j))_{1\leq j\leq s}$, for every $\epsilon \in [0,1]$.
  \item \label{lem: properties of CQI associated to ALG 4}Denote by $\llbracket\cdot, \cdot\rrbracket^{(\epsilon, \mathfrak{i})}$ the quotient bracket on $(\mathfrak{g},\llbracket\cdot,\cdot\rrbracket^{(\epsilon)})/\mathfrak{i}$. Then \[\llbracket\cdot,\cdot\rrbracket^{(\epsilon,\mathfrak{i})} = \llbracket\cdot,\cdot\rrbracket^{(\tau,\mathfrak{i})}, \quad \forall \epsilon,\tau \in [0,1].\]
  \item \label{lem: properties of CQI associated to ALG 5} For $\epsilon \in [0,1],$ equip $(\mathfrak{g},\llbracket\cdot,\cdot \rrbracket^{(\epsilon)})$ and $(\mathfrak{g}/\mathfrak{i}, \llbracket\cdot,\cdot\rrbracket^{(\epsilon,\mathfrak{i})})$ with the Dynkin products $\cdot_\epsilon$ and $\cdot_{\epsilon,\mathfrak{i}}$. Then $\pi:(\mathfrak{g}, \cdot_{\epsilon}) \rightarrow (\mathfrak{g}/\mathfrak{i}, \cdot_{\epsilon,\mathfrak{i}})$ is a local Lie group homomorphism.
 \end{enumerate}
\end{lem}
\begin{proof}
 Claims (1), (2), (3), and (4) follow from the same arguments of the proofs of Lemma~\ref{lem: properties of CQI associated to CLG} (1), (2), (4), and (5) respectively.

 Regarding (5), both exponential maps $\exp_{\mathfrak{g}}: \mathfrak{g} \rightarrow \mathfrak{g}$ and $\exp_{\mathfrak{g}/\mathfrak{i}} : \mathfrak{g}/\mathfrak{i} \rightarrow \mathfrak{g}/\mathfrak{i}$ are local diffeomorphisms and are such that $\exp_{\mathfrak{g}} = \text{id}_\mathfrak{g}$ and $\exp_{\mathfrak{g}/\mathfrak{i}} = \text{id}_{\mathfrak{g}/\mathfrak{i}}$ near the origin. Hence the map $\exp_{\mathfrak{g}/\mathfrak{i}} \circ \pi \circ \exp_{\mathfrak{g}}^{-1} = \pi$ is a local Lie group homomorphism.
\end{proof}

\section{Grönwall Lemma type results}\label{section: gronwall}
Grönwall Lemma-type results are used to bound from above the distance of the endpoints of two curves starting from the same point, given a uniform bound on the difference of their derivatives. The main technical difficulty in this regard is the fact that we need to compare derivatives at distinct points. The prototypical Grönwall Lemma in our context is the following version, whose proof can be found in \cite[Lemma~A.1]{ALDNG22} or in \cite[Lemma~12.2.2]{Don25}.
\begin{lem}[Grönwall]\label{lem: base Gronwall}
 Let $\Omega \subseteq \mathbb{R}^n$ and $X,Y: \Omega \times[0,T] \rightarrow \mathbb{R}^n$. Let $\norm{\cdot}$ be a norm on $\mathbb{R}^n$. Suppose that there are $E,K>0$ such that for all $p,q \in \Omega$ and all $t\in [0,T]$
 \[\norm{X(p,t)-Y(q,t)} \leq E + K\norm{p-q}.\]
If $\eta,\gamma : [0,T] \rightarrow \Omega$ are integral curves of $X$ and $Y$, respectively, and satisfy $\eta(0)=\gamma(0)$, then
 \[\norm{\gamma(t)-\eta(t)} \leq E \frac{\e^{Kt}-1}{K}, \quad \forall t\in [0,T].\]
\end{lem}
In this section, we prove two sharper versions of Lemma~\ref{lem: base Gronwall} under some more restrictive hypotheses that are sufficient for the application of the results to the proofs of the sharper Pansu's and Mitchell's theorems.

\subsection{Grönwall Lemma for Pansu Asymptotic Theorem}\label{subsection: gronwall pansu}
We prove a version of Grönwall Lemma adapted to the context of Pansu's asymptotic theorem (see Lemma~\ref{lem: Pansu Gronwall}). First, we need to introduce some constants.

\begin{definition}\label{def: constant alpha pansu}
 Let $(\mathfrak{g},\Delta)$ be a polarized nilpotent Lie algebra $\mathfrak{g}$ and $(V_i)_{i=1}^s$ an asymptotic grading. We define
 \begin{itemize}
  \item $\alpha_{(1,\infty)} \coloneqq \max \{j \in \mathbb{N}: [V_{p_1},\ldots,V_{p_k}] \subseteq V_{|p|} \oplus V_{\geq |p|+j}, \; \forall k \in \mathbb{N}, \; \forall p \in \mathbb{N}^k\}.$
  \item $\alpha_{(2,\infty)} \coloneqq \max \{j \in \mathbb{N} \, : \, v-\pi_1(v) \in \mathfrak{g}^{(j+1)}, \, \forall v \in \Delta \}.$
  \item $\alpha_{\infty} \coloneqq \min\{\alpha_{(1,\infty)}, \alpha_{(2,\infty)}\}.$
 \end{itemize}
\end{definition}
The constant $\alpha_{(1,\infty)}$ is a slight variation of the constant $e_D$ introduced by Cornullier in \cite{Cor17}, and can be thought of as quantifying how stratifiable the asymptotic grading is. Observe that $(V_1,\ldots,V_s)$ is a stratification if and only if $\alpha_{(1,\infty)}= + \infty.$ The constant $\alpha_{(2,\infty)}$ can be viewed as quantifying how close $\Delta$ is from being the first strata of a stratification. Note that $\alpha_{(2,\infty)} = + \infty$ if and only if $\Delta = V_1.$ It follows that $(G, \Delta)$ is a (metric) Carnot group if and only if $\alpha_{\infty} = + \infty.$
\begin{example}
    \begin{enumerate}
        \item Let $G$ be a simply connected nilpotent stratifiable Lie group equipped with a stratification $(V_i)_i$. Suppose $G$ is equipped with the polarization $\Delta = V_1 \oplus V_2$. Then 
        \[\alpha_{(1,\infty)} = + \infty, \; \alpha_{(2,\infty)} = \alpha_\infty = 1.\]
        \item Let $G$ be a simply connected nilpotent non-stratifiable Lie group equipped with an asymptotic grading $(V_i)_i$. Suppose $G$ is equipped with the polarization $\Delta = V_1.$ Then
        \[\alpha_\infty= \alpha_{(1,\infty)} < \alpha_{(2,\infty)} = +\infty.\]
    \end{enumerate}
\end{example}

\begin{prop}\label{prop: alpha_inf < beta}
 If a nilpotent sub-Finsler Lie group $(G,\Delta, \norm{\cdot})$ with an asymptotic grading $(V_i)_{i=1}^s$ is not a Carnot group, then
 \begin{equation}\label{prop: eq: alpha_inf < beta}
  \nonumber \alpha_\infty < \beta,
 \end{equation}
 where $\alpha_\infty$ and $\beta$ are the constants from Definition~\ref{def: constant alpha pansu} and~\ref{def: constant beta}.
\end{prop}
\begin{proof}
 Since $(G,\Delta, \norm{\cdot})$ is not Carnot, $\beta \geq 1$. Let $\mathfrak{i} \subseteq V_{\leq \beta}$ be a Carnot quotient ideal. We distinguish two cases.
 
 \textbf{Case 1}: $\Delta \neq V_1$. On the one hand, since $\mathfrak{i}$ is a Carnot quotient ideal, we have
 \[\Delta \overset{\ref{def: Carnot quotient ideal}\eqref{def: CQI projection}}{\subseteq} V_1 + \mathfrak{i} \subseteq V_{\leq \beta}.\]
 On the other hand, by definition of $\alpha_{(2,\infty)}$ we have
 \[\Delta \subseteq V_1 \oplus V_{\geq \alpha_{(2,\infty)} + 1}.\]
 Therefore, since $\Delta \neq V_1$, we have $\alpha_{(2,\infty)} +1\leq \beta$, and thus $\alpha_\infty \leq \alpha_{(2,\infty)} < \beta.$

 \textbf{Case 2}: $\Delta = V_1$. Let $k \in \{2,\ldots,s-1\}$ and $p \in \mathbb{N}^k$. By Definition~\ref{def: Carnot quotient ideal}:
\begin{align*}
 [V_{p_1},\ldots,V_{p_k}] +\mathfrak{i} &\overset{\ref{def: Carnot quotient ideal}\eqref{def: CQI 1 ideal}}{=} [V_{p_1}+ \mathfrak{i},\ldots, V_{p_k}+ \mathfrak{i}]\\
 &\overset{\ref{def: Carnot quotient ideal}\eqref{def: CQI 2 stratification}}{\subseteq} V_{|p|}+ \mathfrak{i}.
\end{align*}
Therefore,
\[[V_{p_1},\ldots,V_{p_k}] \subseteq V_{|p|} + \mathfrak{i} \subseteq V_{|p|} + V_{\leq \beta}.\]
Using Proposition~\ref{prop: properties of linear gradings}, we obtain
\[[V_{p_1},\ldots,V_{p_k}] \overset{\ref{prop: properties of linear gradings}\eqref{prop: bracket of CLG containment}}{\subseteq}V_{|p|} \oplus V_{|p|+1} \oplus\cdots\oplus V_{\beta}.\]
In particular, since $|p|\geq 2$, we have
\[[V_{p_1},\ldots,V_{p_k}] \subseteq V_{|p|} \oplus V_{|p|+1} \oplus\cdots\oplus V_{|p| + \beta-2}.\]
Moreover, by definition of $\alpha_{(1,\infty)}$:
\[[V_{p_1},\ldots,V_{p_k}] \subseteq V_{|p|} \oplus V_{|p|+\alpha_{(1,\infty)}} \oplus\cdots\oplus V_{|p| + \beta-2}.\]
So $\alpha_{(1,\infty)} \leq \beta -2$, otherwise $\text{Lie}(G)$ would be stratifiable (and thus $(G,d)$ would be a Carnot group) by Proposition~\ref{prop: properties of linear gradings}\eqref{prop: CLG + algebra grading is stratifictaion}.
\end{proof}
\begin{lem}\label{lem: difference of brackets pansu}
 Let $\mathfrak{g}$ be a polarized nilpotent Lie algebra equipped with an asymptotic grading which induces projections $x\mapsto(x)_j$. Let $n \in \mathbb{N}, p\in \mathbb{N}^n,$ and $x_1,\ldots, x_n \in \mathfrak{g}.$ Then
 \begin{align}\label{difference of brackets}
  [(x_1)_{p_1},\ldots,(x_n)_{p_n}]-[(x_1)_{p_1},\ldots,(x_n)_{p_n}]^{(0)} = \sum_{j={|p|+\alpha_{(1,\infty)}}}^s([(x_1)_{p_1},\ldots,(x_n)_{p_n}])_j.
 \end{align}
\end{lem}
\begin{proof}
We show by induction on $n\in \mathbb{N}$ that 
\begin{equation}\label{eq: lem: difference of brackets pansu 1}
 [(x_1)_{p_1},\ldots,(x_n)_{p_n}]^{(0)} = ([(x_1)_{p_1},\ldots,(x_n)_{p_n}])_{|p|}, \quad \forall p \in \mathbb{N}^n, \; \forall x_1,\ldots,x_n \in \mathfrak{g}.
\end{equation}
By definition of the bracket $[\cdot, \cdot]^{(0)}$, we have for $x,y \in \mathfrak{g}$ and $j,k \in \mathbb{N}$:
\begin{equation}\label{eq: lem: difference of brackets pansu 2}
 [(x)_j, (y)_k]^{(0)} = ([(x)_j,(y)_k])_{j+k}.
\end{equation}
Now suppose that the claim is true for brackets of length up to $n-1$. Using the properties of asymptotic gradings, we obtain
\begin{align*}
 ([(x_1)_{p_1},\ldots,(x_n)_{p_n}])_{|p|} &\overset{\ref{prop: properties of linear gradings} \eqref{prop: bracket of CLG containment}}{=} \sum_{j=|p|-p_1}^{s}([(x_1)_{p_1},([(x_2)_{p_2},\ldots,(x_n)_{p_n}])_{j}])_{|p|} \\
 &\overset{\ref{prop: properties of linear gradings} \eqref{prop: bracket of CLG containment}}{=} ([(x_1)_{p_1},([(x_2)_{p_2},\ldots,(x_n)_{p_n}])_{|p|-p_1}])_{|p|}\\
 &\overset{\eqref{eq: lem: difference of brackets pansu 1}}{=} ([(x_1)_{p_1},([(x_2)_{p_2},\ldots,(x_n)_{p_n}]^{(0)})_{|p|-p_1}])_{|p|} \\
 &\overset{\eqref{eq: lem: difference of brackets pansu 2}}{=}[(x_1)_{p_1},\ldots,(x_n)_{p_n}]^{(0)}.
\end{align*}
Now, by definition of $\alpha_{(1,\infty)}$, we have
\begin{align*}
  [(x_1)_{p_1},\ldots,(x_n)_{p_n}]-[(x_1)_{p_1},\ldots,(x_n)_{p_n}]^{(0)} &\overset{\ref{prop: properties of linear gradings} \eqref{prop: bracket of CLG containment}, \eqref{eq: lem: difference of brackets pansu 1}}{=} \sum_{j={|p|+1}}^s([(x_1)_{p_1},\ldots,(x_n)_{p_n}])_j \\
  &= \sum_{j={|p|+\alpha_{(1,\infty)}}}^s([(x_1)_{p_1},\ldots,(x_n)_{p_n}])_j.
\end{align*}
\end{proof}
\begin{lem}\label{lem: difference of product Pansu}
 Let $\mathfrak{g}$ be a polarized nilpotent Lie algebra equipped with an asymptotic grading which induces projections $x\mapsto(x)_j$. Let $x_1,x_2 \in \mathfrak{g}$ and $\epsilon \in [0,1].$ Then 
 \begin{align}\label{lem: eq: difference of product Pansu}
   x_1*_\epsilon x_2 - x_1 *_0 x_2 = \sum_{k = 2}^{s} \sum_{q \in \{1,2\}^k} \sum_{p \in \{1,\ldots,s\}^k } \sum_{j=|p|+\alpha_{(1,\infty)}}^s b_{k,q} \epsilon^{j-|p|} (\lbrack (x_{q_1})_{p_1},\ldots,(x_{q_k})_{p_k} \rbrack)_j,
 \end{align}
 where $b_{k,q}$ are the constants in the BCH formula~\eqref{def: BCH}.
\end{lem}
\begin{proof}
 For $y_1,\ldots, y_n \in \mathfrak{g}$ and $n\in \mathbb{N}$, by definition of the bracket $[\cdot,\cdot]^{(\epsilon)}$ we have
 \[[y_1,\ldots,y_n]^{(\epsilon)} = \delta_\epsilon [\delta_\epsilon^{-1} y_1,\ldots,\delta_\epsilon^{-1}y_n],\]
 and from the fact that dilations are Lie algebra automorphisms of $(\mathfrak{g}, [\cdot, \cdot]^{(0)})$, we have 
  \begin{align*}
   [y_1,\ldots,y_n]^{(0)} &= \delta_\epsilon [\delta_\epsilon^{-1} y_1,\ldots,\delta_\epsilon^{-1}y_n]^{(0)}. 
  \end{align*}
  Now we use the BCH formula and Lemma~\ref{lem: difference of brackets pansu} to calculate
  \begin{align*}
   x_1*_\epsilon x_2 &- x_1 *_0 x_2 \overset{\eqref{def: BCH}}{=} \sum_{k = 2}^{s} \sum_{q \in \{1,2\}^k}b_{k,q} (\lbrack x_{q_1},\ldots,x_{q_k} \rbrack^{(\epsilon)}-\lbrack x_{q_1},\ldots,x_{q_k} \rbrack^{(0)}) \\
   &=\sum_{k = 2}^{s} \sum_{q \in \{1,2\}^k}b_{k,q}\delta_\epsilon (\lbrack \delta_{\epsilon^{-1}}x_{q_1},\ldots,\delta_{\epsilon^{-1}}x_{q_k} \rbrack-\lbrack \delta_{\epsilon^{-1}}x_{q_1},\ldots,\delta_{\epsilon^{-1}}x_{q_k} \rbrack^{(0)}) \\
   &\overset{\eqref{difference of brackets}}{=}\sum_{k = 2}^{s} \sum_{q \in \{1,2\}^k} \sum_{p \in \{1,\ldots,s\}^k } \sum_{j=|p|+\alpha_{(1,\infty)}}^s b_{k,q} \delta_{\epsilon} (\lbrack (\delta_{\epsilon^{-1}}x_{q_1})_{p_1},\ldots,(\delta_{\epsilon^{-1}}x_{q_k})_{p_k} \rbrack)_j\\
   &=\sum_{k = 2}^{s} \sum_{q \in \{1,2\}^k} \sum_{p \in \{1,\ldots,s\}^k } \sum_{j=|p|+\alpha_{(1,\infty)}}^s b_{k,q} \epsilon^{j-|p|} (\lbrack (x_{q_1})_{p_1},\ldots,(x_{q_k})_{p_k} \rbrack)_j.
  \end{align*}
\end{proof}

\begin{lem}\label{lem: PANSU: bound on norm of projection}
 Let $(\mathfrak{g},\Delta)$ be a polarized nilpotent Lie algebra equipped with a norm $\norm{\cdot}$ and an asymptotic grading which induces projections $x\mapsto(x)_j = \pi_j(x)$ and dilations $(\delta_\lambda)_{\lambda \in \mathbb{R}}$. For every norm $\|\cdot\|_{\rm E}$ on $\mathfrak{g}$ there exists a constant $C$ such that
 \begin{align}\label{PANSU: bound on norm of projection}
  \|u-\pi_1(u)\|_{\rm E} \leq C \epsilon^{\alpha_{(2,\infty)}}\|u\|^{(\epsilon)}, \quad\forall \epsilon \in [0,1], \forall u \in \Delta^{(\epsilon)}
 \end{align}
\end{lem}
\begin{proof}
By \eqref{eq: upper bound of euclidean norm of layer wrt to norm eps}, there is a $C$ such that
 \[\|(u)_k\|_{\rm E} \overset{\eqref{eq: upper bound of euclidean norm of layer wrt to norm eps}}{\leq} C\epsilon^{k-1}\|u\|^{(\epsilon)}, \quad \forall \epsilon \in (0,1], \; \forall u \in \Delta^{(\epsilon)}, \; \forall k \in \{2,\dots,s\}.\]
 Hence by definition of $\alpha_{(2,\infty)}$ and triangle inequality 
 \begin{align*}
  \|u-\pi_1(u)\|_{\rm E} &\leq \|(u)_{\alpha_{(2,\infty)} + 1}\|_{\rm E} + \|(u)_{\alpha_{(2,\infty)} + 2}\|_{\rm E}+ \dots + \|(u)_{s}\|_{\rm E} \\ 
  &\leq C\epsilon^{\alpha_{(2,\infty)}}\|u\|^{(\epsilon)}+ \dots + C\epsilon^{s-1}\|u\|^{(\epsilon)} \\
  &\leq C\epsilon^{\alpha_{(2,\infty)}}\|u\|^{(\epsilon)}.
 \end{align*}
 For $\epsilon = 0$, since $\Delta^{(0)} = V_1$:
 \[\norm{u-\pi_1(u)}_{\rm E} = 0 \leq C\epsilon^{\alpha_{(2,\infty)}}\|u\|^{(0)}, \quad \forall u \in \Delta^{(0)}.\]
\end{proof}
We are ready to state and prove the version of Grönwall Lemma adapted to Pansu Theorem.

\begin{definition} \label{def: VF Pansu}
 Given $u \in L^{2}(\lbrack 0,1 \rbrack, \mathfrak{g})$ and $\epsilon \in [0,1]$, we define the time-dependent vector field $X^{u,\epsilon}: \; \mathfrak{g} \times \lbrack 0,1 \rbrack \rightarrow \mathfrak{g}$ as
 \begin{align*}
 &X^{u,\epsilon}(g,t) \coloneqq (L_g^{*_\epsilon})_* u(t).
\end{align*}
\end{definition}

\begin{lem}[Grönwall Lemma for Pansu Theorem]\label{lem: Pansu Gronwall}
 Let $(G,\Delta, \norm{\cdot})$ be a nilpotent sub-Finsler Lie group equipped with an asymptotic grading $(V_i)_{i=1}^s$. Let $\|\cdot\|_{\rm E}$ be a norm on $\mathfrak{g}$. There is a constant $C>0$ such that $\forall \epsilon \in [0,1]$, for all $L>0$, and for all $u \in L^{2}(\lbrack 0,1 \rbrack, \Delta^{(\epsilon)})$ with $\|u(t)\|^{(\epsilon)} \leq L$, for almost every $t \in [0,1]$, we have:

 \begin{enumerate}
  \item \label{lem: Pansu Gronwall1}For the vector fields defined in Definition~\ref{def: VF Pansu}, we have that, for $g,h \in B_{\rm E}(0,1)$ and almost every $t \in [0,1]$:
  \[\|X^{u,\epsilon}(g,t) - X^{\pi_1 \circ u,0}(h,t)\|_{\rm E} \leq CL (\epsilon^{\alpha_{\infty}} + \|g-h\|_{\rm E}),\]
  where $\alpha_{\infty}$ is the constant from Definition~\ref{def: constant alpha pansu}.
  \item \label{lem: Pansu Gronwall2} Let $\gamma, \eta \colon [0,1] \rightarrow \mathfrak{g}$ be integral curves of $X^{u,\epsilon}$ and $X^{\pi_1 \circ u,0}$, respectively, with $\gamma(0)=\eta(0)$ and $\gamma(t),\eta(t)\in B_{\rm E}(0,1), \forall t \in [0,1]$. Then 
  \[\|\gamma(1)-\eta(1)\|_{\rm E} \leq \epsilon^{\alpha_\infty}(\e^{CL} -1) \] 
 \end{enumerate}
\end{lem}

\begin{proof}
 \textbf{Proof of \eqref{lem: Pansu Gronwall1}}. Let $g,h \in B_{\rm E}(0,1)$ and $t \in [0,1]$ such that $\|u(t)\|^{(\epsilon)} \leq L$. We have $\norm{\cdot}_{\rm E} \overset{\eqref{eq: upper bound of euclidean norm wrt to norm eps}}{\lesssim} \norm{\cdot}^{(\epsilon)}$, and consequently $\norm{u(t)}_{\rm E} \lesssim L$, where the inequalities are up to a multiplicative constant that does not depend on $\epsilon$ or $u$. We have
 \begin{align*}
  \|X^{u,\epsilon}(g,t) - X^{\pi_1 \circ u,0}(h,t)\|_{\rm E} &\leq \|X^{u,\epsilon}(g,t) - X^{u,0}(g,t)\|_{\rm E} \\
  &+ \|X^{u,0}(g,t) - X^{ u,0}(h,t)\|_{\rm E} \\
  &+\|X^{u,0}(h,t) - X^{ \pi_1 \circ u,0}(h,t)\|_{\rm E}.
 \end{align*}
 We bound the three quantities on the right-hand side separately.

 \textbf{We bound} $\|X^{u,\epsilon}(g,t) - X^{u,0}(g,t)\|_{\rm E}$: Recall that that exponential map $\exp : (\mathfrak{g},[\cdot,\cdot]^{(\tau)}) \rightarrow (\mathfrak{g}, *_\tau)$ is the identity map for every $\tau \in [0,1].$ Therefore
 \begin{align*}
  X^{u,\epsilon}(g,t) - X^{u,0}(g,t) = (L_g^{*_\epsilon})_*u(t)-(L_g^{*_0})_*u(t) = \left.\frac{\dd}{\dd r}g*_\epsilon ru(t) - g*_0 ru(t) \right|_{r=0}.
 \end{align*}
 In order to simplify the notation, we define two functions $x_1,x_2 : (-1,1) \rightarrow \mathfrak{g}$ by $x_1(r) \coloneqq g$, $x_2(r) \coloneqq ru(t).$ We exploit the property of the constant $\alpha_{(1,\infty)}$, and by Lemma~\ref{lem: difference of product Pansu} we get:
 \begin{align*}
  &X^{u,\epsilon}(g,t) - X^{u,0}(g,t) \\ &= \left.\frac{\dd}{\dd r}x_1(r) *_\epsilon x_2(r) - x_1(r) *_0 x_2(r)\right|_{r=0} \\
  &\overset{\eqref{lem: eq: difference of product Pansu}}{=}\left.\frac{\dd}{\dd r} \sum_{k = 2}^{s} \sum_{q \in \{1,2\}^k} \sum_{p \in \{1,\ldots,s\}^k } \sum_{j=|p|+\alpha_{(1,\infty)}}^s b_{k,q} \epsilon^{j-|p|} (\lbrack (x_{q_1}(r))_{p_1},\ldots,(x_{q_k}(r))_{p_k} \rbrack)_j\right|_{r=0}\\
  &= \sum_{k = 2}^{s} \sum_{m=1}^k \sum_{p \in \{1,\ldots,s\}^k } \sum_{j=|p|+\alpha_{(1,\infty)}}^s b_{k,q} \epsilon^{j-|p|} (\lbrack (g)_{p_1},\ldots,(u(t))_{p_m},\ldots,(g)_{p_k} \rbrack)_j,
 \end{align*}
 where we used that after derivation and evaluation at $r = 0$, the only non-zero iterated brackets of $x_1(r)$ and $x_2(r)$ are those where $x_2(r)$ appears exactly once. We point out that since $\epsilon \in [0,1]$ and $j \in \{|p|+\alpha_{(1,\infty)},\ldots,s\}$ then $\epsilon^{j-|p|}$ is at most $\epsilon^{|p|+\alpha_{(1,\infty)}-|p|} = \epsilon^{\alpha_{(1,\infty)}}$. Finally, recalling that $\|[a,b]\|_{\rm E} \lesssim\|a\|_{\rm E}\|b\|_{\rm E}, \; \forall a,b \in \mathfrak{g}$, where the inequality is up to a multiplicative constant depending only on $G$, and $\norm{\cdot}_{\rm E}$, we bound
 \begin{align*}
  \|X^{u,\epsilon}(g,t) - X^{u,0}(g,t)\|_{\rm E} \lesssim \sum_{k = 2}^{s} \sum_{m=1}^k \sum_{p \in \{1,\ldots,s\}^k }L \epsilon^{|p|+\alpha_{(1,\infty)}-|p|} \lesssim L\epsilon^{\alpha_{(1,\infty)}}. 
 \end{align*}

\textbf{We bound} $\|X^{u,0}(g,t) - X^{ u,0}(h,t)\|_{\rm E}$: Notice that the map $\phi: g \mapsto (L_g^{*_0})_* = \dd(L_g^{*_0})_0$ is smooth. Hence, it is Lipschitz on compact sets, and thus there exists a constant $C$ such that 
\[\|\phi(g')-\phi(h')\|_{\rm op} \leq C\| g'-h'\|_{\rm E}, \quad \forall g',h' \in B_{\rm E}(0,1),\]
where $\norm{\cdot}_{\rm op}$ is the operator norm for linear operators $\mathfrak{g}\rightarrow \mathfrak{g}$.
It follows that 
\begin{align*}
 \|X^{u,0}(g,t) - X^{u,0}(h,t)\|_{\rm E} &= \|\phi(g)u(t)-\phi(h)u(t)\|_{\rm E}\\
 &\leq \|\phi(g)-\phi(h)\|_{\rm op}\|u(t)\|_{\rm E}\\
 &\lesssim L\|g-h\|_{\rm E},
\end{align*}
where the inequality is up to a multiplicative factor that does not depend on $\epsilon$, $u$, or $L$.

\textbf{We bound} $\|X^{u,0}(h,t) - X^{ \pi_1 \circ u,0}(h,t)\|_{\rm E}$: Since the map $(L_h^{*_0})_*$ is linear and smooth in $h$, there exists a constant $C$ such that
\[\|(L_{h'}^{*_0})_*(u(t)) - (L_{h'}^{*_0})_*(\pi_1(u(t)))\|_{\rm E} \leq C \|u(t) -\pi_1(u(t))\|_{\rm E}, \quad \forall h' \in B_{\rm E}(0,1).\] We exploit the property of the constant $\alpha_{(2,\infty)}$, and by Lemma~\ref{lem: PANSU: bound on norm of projection} we get:
\begin{align*}
 \|X^{u,0}(h,t) - X^{ \pi_1 \circ u,0}(h,t)\|_{\rm E} &\overset{\text{def}}{=} \|(L_h^{*_0})_*(u(t)) - (L_h^{*_0})_*(\pi_1(u(t))\|_{\rm E} \\
 &\leq C \|u(t) -\pi_1(u(t))\|_{\rm E}\\
 &\overset{\eqref{PANSU: bound on norm of projection}}{\leq} C L \epsilon^{\alpha_{(2,\infty)}}.
\end{align*}

\textbf{We conlude} that there exists a constant $C$ depending only on $G,d,(V_i), L$, and $\|\cdot\|_{\rm E}$ such that
\begin{align*}
 \|X^{u,\epsilon}(g,t) - X^{\pi_1 \circ u,0}(h,t)\|_{\rm E} &\leq \|X^{u,\epsilon}(g,t) - X^{u,0}(g,t)\|_{\rm E} \\
  &\quad+ \|X^{u,0}(g,t) - X^{ u,0}(h,t)\|_{\rm E} \\
  &\quad+\|X^{u,0}(h,t) - X^{ \pi_1 \circ u,0}(h,t)\|_{\rm E}.\\
  &\leq CL(\epsilon^{\min\{\alpha_{(1,\infty)},\alpha_{(2,\infty)}\}} + \|g-h\|_{\rm E}).
\end{align*}

\textbf{Proof of }\eqref{lem: Pansu Gronwall2}. Apply Grönwall Lemma~\ref{lem: base Gronwall} and Lemma~\ref{lem: Pansu Gronwall}\eqref{lem: Pansu Gronwall1}.
\end{proof}

\subsection{Grönwall Lemma for Mitchell Tangent Theorem}
We now prove a version of Grönwall Lemma adapted to the context of Mitchell Theorem (see Lemma~\ref{lem: Mitchell Gronwall}).

\begin{definition}\label{def: constant alpha mitchell}
 Let $(\mathfrak{g},\Delta)$ be a polarized Lie algebra $\mathfrak{g}$ and $(W_i)_{i=1}^s$ a tangent grading. We define
 \begin{align*}
  \alpha_{0} \coloneqq \max \{j \in \mathbb{N}: [W_{p_1},\ldots,W_{p_k}] \subseteq W_{|p|} \oplus W_{\leq |p|-j}, \; \forall k \in \mathbb{N}, \; \forall p \in \mathbb{N}^k\}.
 \end{align*}
\end{definition}
The constant $\alpha_0$ is the analog of the constant $\alpha_{(1,\infty)}$ in the context of tangent gradings. It can be thought of as quantifying how stratifiable a tangent grading is; a tangent grading is a stratification if and only if $\alpha_0=+\infty$. Note that by definition of tangent grading we have $\Delta = W_1$, hence there is no analog of the constant $\alpha_{(2,\infty)}$ in this context.
\begin{lem}\label{lem: difference of brackets mitchell}
 Let $\mathfrak{g}$ be a polarized Lie algebra equipped with a tangent grading which induces projections $x\mapsto(x)_j$. Let $n \in \mathbb{N},\; p\in \mathbb{N}^n,$ and $x_1,\ldots, x_n \in \mathfrak{g}.$ Then
 \begin{align}\label{eq: difference of brackets mitchell}
  [(x_1)_{p_1},\ldots,(x_n)_{p_n}]-\llbracket(x_1)_{p_1},\ldots,(x_n)_{p_n}\rrbracket^{(0)} = \sum_{j={1}}^{|p|-\alpha_{0}}([(x_1)_{p_1},\ldots,(x_n)_{p_n}])_j.
 \end{align}
\end{lem}
\begin{proof}
 We show by induction on $n\in \mathbb{N}$ that 
\begin{equation}\label{eq: lem: difference of brackets mitchell 1}
 \llbracket(x_1)_{p_1},\ldots,(x_n)_{p_n}\rrbracket^{(0)} = ([(x_1)_{p_1},\ldots,(x_n)_{p_n}])_{|p|}, \quad \forall p \in \mathbb{N}^n, \; \forall x_1,\ldots,x_n \in \mathfrak{g}.
\end{equation}
By definition of the bracket $\llbracket\cdot, \cdot\rrbracket^{(0)}$, we have for $x,y \in \mathfrak{g}$ and $j,k \in \mathbb{N}$:
\begin{equation}\label{eq: lem: difference of brackets mitchell 2}
 \llbracket(x)_j, (y)_k\rrbracket^{(0)} = ([(x)_j,(y)_k])_{j+k}.
\end{equation}
Next, suppose that the claim is true for brackets of length up to $n-1$. Using the properties of tangent gradings, we obtain
\begin{align*}
 ([(x_1)_{p_1},\ldots,(x_n)_{p_n}])_{|p|} &\overset{\ref{prop: properties of linear gradings} \eqref{prop: bracket of ALG containment}}{=} \sum_{j=1}^{|p|-p_1}([(x_1)_{p_1},([(x_2)_{p_2},\ldots,(x_n)_{p_n}])_{j}])_{|p|} \\
 &\overset{\ref{prop: properties of linear gradings} \eqref{prop: bracket of ALG containment}}{=} ([(x_1)_{p_1},([(x_2)_{p_2},\ldots,(x_n)_{p_n}])_{|p|-p_1}])_{|p|}\\
 &\overset{\eqref{eq: lem: difference of brackets mitchell 1}}{=} ([(x_1)_{p_1},(\llbracket(x_2)_{p_2},\ldots,(x_n)_{p_n}\rrbracket^{(0)})_{|p|-p_1}])_{|p|} \\
 &\overset{\eqref{eq: lem: difference of brackets mitchell 2}}{=}\llbracket(x_1)_{p_1},\ldots,(x_n)_{p_n}\rrbracket^{(0)}.
\end{align*}

Now, by definition of $\alpha_0$, we have
\begin{align*}
  [(x_1)_{p_1},\ldots,(x_n)_{p_n}]-\llbracket(x_1)_{p_1},\ldots,(x_n)_{p_n}\rrbracket^{(0)} &\overset{\ref{prop: properties of linear gradings} \eqref{prop: bracket of ALG containment}, \eqref{eq: lem: difference of brackets mitchell 1}}{=} \sum_{j=1}^{|p|-1}([(x_1)_{p_1},\ldots,(x_n)_{p_n}])_j \\
  &= \sum_{j=1}^{|p|-\alpha_0}([(x_1)_{p_1},\ldots,(x_n)_{p_n}])_j.
\end{align*}
\end{proof}

\begin{lem}\label{lem: difference of products mitchell}
Let $\mathfrak{g}$ be a polarized Lie algebra equipped with a tangent grading which induces projections $x\mapsto(x)_j$. Let $x_1,x_2 \in \mathfrak{g}$ and $\epsilon \in [0,1].$ Then 
 \begin{align}\label{lem: eq: difference of product Mitchell}
   x_1\cdot_\epsilon x_2 - x_1 \cdot_0 x_2 = \sum_{k = 2}^{s} \sum_{q \in \{1,2\}^k} \sum_{p \in \{1,\ldots,s\}^k } \sum_{j=1}^{|p|-\alpha_0} b_{k,q} \epsilon^{|p|-j} (\lbrack (x_{q_1})_{p_1},\ldots,(x_{q_k})_{p_k} \rbrack)_j,
 \end{align}
 where $b_{k,p}$ are the constants in the BCH formula~\eqref{def: BCH}.
\end{lem}
\begin{proof}
  For $y_1,\ldots, y_n \in \mathfrak{g}$ and $n\in \mathbb{N}$, by definition of the bracket $\llbracket\cdot,\cdot\rrbracket^{(\epsilon)}$ we have
 \[\llbracket y_1,\ldots,y_n\rrbracket^{(\epsilon)} = \delta_\epsilon^{-1} [\delta_\epsilon y_1,\ldots,\delta_\epsilon y_n],\]
 and from the fact that dilations are Lie algebra automorphisms of $(\mathfrak{g}, \llbracket \cdot, \cdot\rrbracket^{(0)})$, we have 
  \begin{align*}
   \llbracket y_1,\ldots,y_n \rrbracket^{(0)} &= \delta_\epsilon^{-1} \llbracket \delta_\epsilon y_1,\ldots,\delta_\epsilon y_n\rrbracket^{(0)}. 
  \end{align*}
  Now we use the BCH formula \eqref{def: BCH} and Lemma~\ref{lem: difference of brackets mitchell} to calculate
  \begin{align*}
   x_1 \cdot_\epsilon x_2 - x_1 \cdot_0 x_2 &\overset{\eqref{def: BCH}}{=} \sum_{k = 2}^{s} \sum_{q \in \{1,2\}^k}b_{k,q} (\llbracket x_{q_1},\ldots,x_{q_k} \rrbracket^{(\epsilon)}-\llbracket x_{q_1},\ldots,x_{q_k} \rrbracket^{(0)}) \\
   &=\sum_{k = 2}^{s} \sum_{q \in \{1,2\}^k}b_{k,q}\delta_\epsilon^{-1} (\lbrack \delta_{\epsilon}x_{q_1},\ldots,\delta_{\epsilon}x_{q_k} \rbrack -\llbracket \delta_{\epsilon}x_{q_1},\ldots,\delta_{\epsilon}x_{q_k} \rrbracket^{(0)}) \\
   &\overset{\eqref{eq: difference of brackets mitchell}}{=}\sum_{k = 2}^{s} \sum_{q \in \{1,2\}^k} \sum_{p \in \{1,\ldots,s\}^k } \sum_{j=1}^{|p|-\alpha_0} b_{k,q} \delta_{\epsilon}^{-1} (\lbrack (\delta_{\epsilon}x_{q_1})_{p_1},\ldots,(\delta_{\epsilon}x_{q_k})_{p_k} \rbrack)_j\\
   &=\sum_{k = 2}^{s} \sum_{q \in \{1,2\}^k} \sum_{p \in \{1,\ldots,s\}^k } \sum_{j=1}^{|p|-\alpha_0} b_{k,q} \epsilon^{|p|-j} (\lbrack (x_{q_1})_{p_1},\ldots,(x_{q_k})_{p_k} \rbrack)_j.
  \end{align*}
\end{proof}
We are ready to state and prove the version of Grönwall Lemma adapted to Mitchell Theorem.

\begin{definition} \label{def: VF Mitchell}
 Given $u \in L^{2}(\lbrack 0,1 \rbrack, \mathfrak{g})$ and $\epsilon \in [0,1]$, we define the time-dependent vector field $Y^{u,\epsilon}: \; \mathfrak{g} \times \lbrack 0,1 \rbrack \rightarrow \mathfrak{g}$ as
 \begin{align*}
 &Y^{u,\epsilon}(g,t) \coloneqq (L_g^{\cdot_\epsilon})_* u(t).
\end{align*}
\end{definition}

\begin{lem}[Grönwall Lemma for Mitchell Theorem]\label{lem: Mitchell Gronwall}
Let $(G,\Delta, \norm{\cdot})$ be a sub-Finsler Lie group equipped with a tangent grading $(W_i)_{i=1}^s$. Let $\|\cdot\|_{\rm E}$ be a norm on $\mathfrak{g}$. There is a constant $C>0$ such that for every $\epsilon \in [0,1]$, for every $L>0$, and for all $u \in L^{2}(\lbrack 0,1 \rbrack, \Delta)$ with $\|u(t)\|_{\rm E} \leq L$, for almost every $t \in [0,1]$, we have:

 \begin{enumerate}
  \item \label{lem: Mitchell Gronwall1}For the vector fields from Definition~\ref{def: VF Mitchell}, we have that, for $g,h \in B_{\rm E}(0,1)$ and almost every $t\in [0,1]$:
  \[\|Y^{u,\epsilon}(g,t) - Y^{u,0}(h,t)\|_{\rm E} \leq CL(\epsilon^{\alpha_{0}} + \|g-h\|_{\rm E}),\]
  where $\alpha_{0}$ is the constant from Definition~\ref{def: constant alpha mitchell}.
  \item \label{lem: Mitchell Gronwall2} Let $\gamma, \eta \colon [0,1] \rightarrow \mathfrak{g}$ be integral curves of $Y^{u,\epsilon}$ and $Y^{u,0}$, respectively, with $\gamma(0)=\eta(0)$ and $\gamma(t),\eta(t)\in B_{\rm E}(0,1), \forall t\in [0,1]$. Then 
  \[\|\gamma(1)-\eta(1)\|_{\rm E} \leq  \epsilon^{\alpha_0}(\e^{CL}-1).\]
 \end{enumerate}
\end{lem}
\begin{proof}
 \textbf{Proof of }\eqref{lem: Mitchell Gronwall1}. We follow the same strategy of the proof of Lemma~\ref{lem: Pansu Gronwall}. Let $g,h \in B_{\rm E}(0,1)$ and $t\in [0,1]$ such that $\|u(t)\|_{\rm E} \leq L$. We have
 \begin{align*}
  \|Y^{u,\epsilon}(g,t) - Y^{u,0}(h,t)\|_{\rm E} &\leq \|Y^{u,\epsilon}(g,t) - Y^{u,0}(g,t)\|_{\rm E} + \|Y^{u,0}(g,t) - Y^{ u,0}(h,t)\|_{\rm E}.
 \end{align*}
 We separately bound the two quantities on the right-hand side.

 \textbf{We bound} $\|Y^{u,\epsilon}(g,t) - Y^{u,0}(g,t)\|_{\rm E}$: Recall that that exponential map $\exp : (\mathfrak{g},\llbracket\cdot,\cdot\rrbracket^{(\tau)}) \rightarrow (\mathfrak{g}, \cdot_\tau)$ is the identity map for every $\tau \in [0,1].$ Therefore
 \begin{align*}
  Y^{u,\epsilon}(g,t) - Y^{u,0}(g,t) = (L_g^{\cdot_\epsilon})_*u(t)-(L_g^{\cdot_0})_*u(t) = \left.\frac{\dd}{\dd r}g\cdot_\epsilon ru(t) - g\cdot_0 ru(t) \right|_{r=0}.
 \end{align*}
 In order to simplify the notation, we define two functions $x_1,x_2 : (-1,1) \rightarrow \mathfrak{g}$ by $x_1(r) \coloneqq g$, $x_2(r) \coloneqq ru(t).$ We exploit the property of the constant $\alpha_0$, and by Lemma~\ref{lem: difference of products mitchell} we get: 
 \begin{align*}
  &Y^{u,\epsilon}(g,t) - Y^{u,0}(g,t) \\ &= \left.\frac{\dd}{\dd r}x_1(r) \cdot_\epsilon x_2(r) - x_1(r) \cdot_0 x_2(r)\right|_{r=0} \\
  &\overset{\eqref{lem: eq: difference of product Mitchell}}{=}\left.\frac{\dd}{\dd r} \sum_{k = 2}^{s} \sum_{q \in \{1,2\}^k} \sum_{p \in \{1,\ldots,s\}^k } \sum_{j=1}^{|p|-\alpha_0} b_{k,q} \epsilon^{|p|-j} (\lbrack (x_{q_1}(r))_{p_1},\ldots,(x_{q_k}(r))_{p_k} \rbrack)_j\right|_{r=0}\\
  &= \sum_{k = 2}^{s} \sum_{m=1}^k \sum_{p \in \{1,\ldots,s\}^k } \sum_{j=1}^{|p|-\alpha_0} b_{k,q} \epsilon^{|p|-j} (\lbrack (g)_{p_1},\ldots,(u(t))_{p_m},\ldots,(g)_{p_k} \rbrack)_j,
 \end{align*}
 where we used that after derivation and evaluation at $r = 0$, the only non-zero iterated brackets of $x_1(r)$ and $x_2(r)$ are those where $x_2(r)$ appears exactly once. We point out that since $\epsilon \in [0,1]$ and $j\in \{1,\ldots,|p|-\alpha_0\}$ then $\epsilon^{|p|-j}$ is at most $\epsilon^{|p|-(|p|-\alpha_0)}= \epsilon^{\alpha_0}$. Finally, recalling that $\|[a,b]\|_{\rm E} \lesssim \|a\|_{\rm E}\|b\|_{\rm E}, \; \forall a,b \in \mathfrak{g}$, where the inequality is up to a multiplicative constant depending only on $G$ and $\norm{\cdot}_{\rm E}$, we bound
 \begin{align*}
  \|Y^{u,\epsilon}(g,t) - Y^{u,0}(g,t)\|_{\rm E} \lesssim L\sum_{k = 2}^{s} \sum_{m=1}^k \sum_{p \in \{1,\ldots,s\}^k } \epsilon^{|p|-(|p|- \alpha_0)} \lesssim L\epsilon^{\alpha_0}, 
 \end{align*}
 where the inequalities are up to a multiplicative constant depending only on $G,(W_i)$, and $\|\cdot\|_{\rm E}$.

\textbf{We bound} $\|Y^{u,0}(g,t) - Y^{ u,0}(h,t)\|_{\rm E}$: Notice that the map $\phi: g \mapsto (L_g^{\cdot_0})_* = \dd(L_g^{\cdot_0})_0$ is smooth. Hence, it is Lipschitz on compact sets, and thus there exists a constant $C$ such that 
\[\|\phi(g)-\phi(h)\|_{\rm op} \leq C\| g-h\|_{\rm E}. \]
It follows that 
\begin{align*}
 \|Y^{u,0}(g,t) - Y^{u,0}(h,t)\|_{\rm E} &= \|\phi(g)u(t)-\phi(h)u(t)\|_{\rm E}\\
 &\leq \|\phi(g)-\phi(h)\|_{\rm op}\|u(t)\|_{\rm E}\\
 &\leq CL\|g-h\|_{\rm E}.
\end{align*}

\textbf{We conclude} that there exists a constant $C'$ depending only on $G,d,(W_i)$, and $\|\cdot\|_{\rm E}$ such that 
\begin{align*}
  \|Y^{u,\epsilon}(g,t) - Y^{u,0}(h,t)\|_{\rm E} &\leq \|Y^{u,\epsilon}(g,t) - Y^{u,0}(g,t)\|_{\rm E} 
  + \|Y^{u,0}(g,t) - Y^{ u,0}(h,t)\|_{\rm E} \\
  &\leq C'L(\epsilon^{\alpha_0}+ \norm{g-h}_{\rm E}).
 \end{align*}

\textbf{Proof of }\eqref{lem: Mitchell Gronwall2}. Apply Grönwall Lemma~\ref{lem: base Gronwall} and Lemma~\ref{lem: Mitchell Gronwall}\eqref{lem: Mitchell Gronwall1}.
\end{proof}

\section{Quantitative Pansu Asymptotic Theorem}\label{section: quantitative pansu}
Theorem~\ref{main thm: pansu} is a consequence of the following equivalent result, which compares contracted metrics. After completing the proof of Theorem~\ref{main thm: pansu}, we show how it implies Theorem~\ref{thm: Pansu asymptotic thm}.

\begin{thm}\label{thm: quantitative Pansu}
 Let $G$ be a simply connected nilpotent sub-Finsler Lie group. With respect to an asymptotic grading, consider the Pansu limit metric $\rho_0$ on $G$ and the contracted metrics $\rho_\epsilon$. Then, for every compact set $K \subseteq G$, there is a constant $C>0$ such that
 \begin{equation}\label{eq: thm: quantitative pansu}
  \left|\rho_0(p,q)-\rho_\epsilon(p,q)\right| \leq C \epsilon^{\alpha_{\infty}/\beta}, \quad \forall \epsilon \in [0,1], \; \forall p,q \in K,
 \end{equation}
  where $\alpha_{\infty}$ and $\beta$ are the constants from Definition~\ref{def: constant alpha pansu} and~\ref{def: constant beta}.
\end{thm}

From now to the end of this section, we are in the assumption of Theorem~\ref{thm: quantitative Pansu}. In addition, we consider the sub-Finsler structures $(\Delta^{(\epsilon)}, \norm{\cdot}^{(\epsilon)})$ defined in \eqref{def: contracted structure}, which are associated to the contracted metrics $\rho_\epsilon$. We denote the asymptotic grading by $(V_1,\ldots,V_s).$

The strategy of the proof is to take a geodesic for one of the two structures, modify it so that it becomes an admissible curve for the other structure, maintaining the length controlled, and then bound the distance of the endpoints. The next lemma is used to show that the difference between these two curves lies in a Carnot quotient ideal.

\begin{lem}\label{lem: projection of curves pansu}
 Let $\epsilon \in (0,1]$, let $\gamma_\epsilon$ be a $\rho_\epsilon$-admissible curve with control $u \in L^2([0,1],\Delta^{(\epsilon)})$, let $\gamma_0$ be the solution of \begin{equation*}
 \begin{cases}
  \gamma_0(0) = \gamma_\epsilon(0)\\
  \dot{\gamma}_0(t) = (L^{*_0}_{\gamma_0(t)})_*\pi_1(u(t)),
 \end{cases}
\end{equation*}
and let $\mathfrak{i}$ be a Carnot quotient ideal (according to Definition~\ref{def: Carnot quotient ideal}). Denote by $\pi : \mathfrak{g} \rightarrow \mathfrak{g}/\mathfrak{i}$ the quotient map. Then
\begin{enumerate}
 \item $\pi \circ \gamma_0 = \pi \circ \gamma_\epsilon$.
 \item $\gamma_0(1)-\gamma_\epsilon(1), \; \gamma_0(1)^{-1} *_0\gamma_\epsilon(1) , \;\gamma_0(1)^{-1} *_\epsilon\gamma_\epsilon(1) \in \mathfrak{i}$.
\end{enumerate} 
\end{lem}
\begin{proof}
 (1). It is sufficient to show that the projected curves are solutions to the same Cauchy problem. Recall that $\mathfrak{i}$ is an ideal for each Lie algebra $(\mathfrak{g},[\cdot,\cdot]^{(\epsilon)})$, see Lemma~\ref{lem: properties of CQI associated to CLG}\eqref{lem: properties of CQI associated to CLG 2}. Moreover, the quotient bracket is independent of $\epsilon$, see Lemma~\ref{lem: properties of CQI associated to CLG}\eqref{lemma equal brackets}. We denote by $*_\mathfrak{i}$ the Dynkin product on $\mathfrak{g}/\mathfrak{i}$, and we denote the left translation by an element $v \in \mathfrak{g}/\mathfrak{i}$ with respect to this product by $L_v^{\mathfrak{i}}$. Observe also that by Lemma~\ref{lem: properties of CQI associated to CLG}\eqref{pi homomorphism}, the quotient map $\pi : \mathfrak{g} \rightarrow \mathfrak{g}/\mathfrak{i}$ is a Lie group homomorphism. Therefore we have 
 \begin{align*}
  \frac{\dd}{\dd t}{(\pi \circ \gamma_\epsilon)}(t) &= (\dd\pi)\dot{\gamma_\epsilon}(t) \\
  &= (\dd\pi)(\dd L_{\gamma_\epsilon(t)}^{*_\epsilon})u(t) \\
  &=\dd(\pi \circ L_{\gamma_\epsilon(t)}^{*_\epsilon})u(t) \\
  &=\dd(L_{\pi(\gamma_\epsilon(t))}^{\mathfrak{i}}\circ \pi)u(t)\\
  &=\dd(L_{\pi(\gamma_\epsilon(t))}^{\mathfrak{i}})\circ \pi(u(t)).
 \end{align*}
 And
 \begin{align*}
  \frac{\dd}{\dd t}{(\pi \circ \gamma_0)}(t) &= (\dd\pi)\dot{\gamma_0}(t) \\
  &= (\dd\pi)(\dd L_{\gamma_0(t)}^{*_0})\pi_1(u(t)) \\
  &=\dd(\pi \circ L_{\gamma_0(t)}^{*_0})\pi_1(u(t)) \\
  &= \dd(L_{\pi(\gamma_0(t))}^{\mathfrak{i}}\circ \pi)\pi_1(u(t)) \\
  &=\dd(L_{\pi(\gamma_0(t))}^{\mathfrak{i}})\circ \pi \circ \pi_1( u(t)),\\
  &=\dd(L_{\pi(\gamma_0(t))}^{\mathfrak{i}})\circ \pi (u(t)),
 \end{align*}
 where in the last step we used Lemma \ref{lem: properties of CQI associated to CLG}\eqref{projection of distribution is the same}. Hence, $\pi \circ \gamma_0 = \pi \circ \gamma_\epsilon$.

 (2). The fact that $\gamma_0(1)-\gamma_\epsilon(1) \in \mathfrak{i}$ is an immediate consequence of (1). Using (1) and the fact that $\pi$ is a Lie group homomorphism, we obtain for $\tau \in \{0, \epsilon\}$:
 \begin{align*}
  (\pi \circ\gamma_0(1))^{-1} &*_{\mathfrak{i}} (\pi \circ\gamma_\epsilon(1)) = 0 \\
  \iff \pi(\gamma_0(1)^{-1}&*_\tau \gamma_\epsilon(1)) = 0,
 \end{align*}
 so $\gamma_0(1)^{-1} *_\tau \gamma_\epsilon(1) \in \mathfrak{i}.$
\end{proof}

For the proof of Theorem~\ref{thm: quantitative Pansu}, we need one more result.
\begin{lem}\label{lemma: uniform bound on contracted metrics}
 For every $L \geq 1$ and for every norm $\norm{\cdot}$ on $\mathfrak{g}$, there is a constant $C>0$ such that we have
 \begin{align}\label{eq: lemma unif bound}
  \rho_\epsilon(0,v) \leq C \norm{v}^{1/\beta}, \quad \forall \epsilon \in [0,1], \; \forall v \in (B_{\rho_0}(0,L)\setminus B_{\rho_0}(0,\epsilon)) \cap V_{\leq \beta}.
 \end{align}
 where $\beta$ is the constant from Definition~\ref{def: constant beta}.
\end{lem}
\begin{proof}If $\beta = 0$, then $V_{\leq \beta} = \{0\}$ and \eqref{eq: lemma unif bound} trivially holds. So suppose $\beta \geq 1$, fix $\epsilon \in [0,1]$, and let $B_\epsilon \coloneqq (B_{\rho_0}(0,L)\setminus B_{\rho_0}(0,\epsilon)) \cap V_{\leq \beta}$.
Since $B_\epsilon \subseteq V_{\leq \beta}$, by the Ball-Box Theorem for Carnot groups, see Theorem~\ref{thm: Ball-Box for Carnot}, there is a constant $C_1>0$ such that 
\begin{equation}\label{lem: bound on contracted metric equation 1}
 \rho_0(0,v) \overset{\eqref{thm: eq: Ball box for Carnot}}{\leq} C_1 \norm{v}^{1/\beta}, \quad \forall v \in B_\epsilon.
\end{equation}

 We now prove~\eqref{eq: lemma unif bound} by contradiction. Assume that there exist sequences $(\epsilon_n)_{n \in \mathbb{N}} \subseteq [0,1]$ and $(v_n)_{n \in \mathbb{N}}$, $v_n \in B_{\epsilon_n}$, such that for every $n \in \mathbb{N}$:
 \begin{equation}\label{exercise: weak uniform bound on contracted metrics equation}
  \rho_{\epsilon_n}(0,v_n) > n\norm{v_n}^{1/\beta}.
 \end{equation}
 Then, by Guivarc'h Theorem~\ref{thm: Guivarch}, there is a constant $D \geq 1$ such that 
 \begin{align*}
  n \overset{\eqref{exercise: weak uniform bound on contracted metrics equation}}{<} \frac{\rho_{\epsilon_n}(0,v_n)}{\norm{v_n}^{1/\beta}} \overset{\eqref{lem: bound on contracted metric equation 1}}{\leq } C_1 \frac{\rho_{\epsilon_n}(0,v_n)}{\rho_0(0,v_n)} \overset{\eqref{eq: Guivarch equation}}{\leq} C_1D + \frac{C_1D\epsilon_n}{\rho_0(0,v_n)}.
 \end{align*}
 Hence $\frac{\epsilon_n}{\rho_0(0,v_n)} \rightarrow \infty.$ In particular for $n$ large enough 
 \[\epsilon_n > \rho_0(0,v_n),\]
 contradicting the fact that $v_n \in B_{\epsilon_n}$.
\end{proof}
We obtain the following corollary.

\begin{corollary}\label{cor: lemma uniform bound on rho_eps}
 For every $L \geq 1$ and for every norm $\norm{\cdot}$ on $\mathfrak{g}$, there is a constant $C>0$ such that we have
 \begin{align}\label{cor: eq: unif bound on rho_eps}
  \rho_\epsilon(0,v) \leq C\max\{\norm{v}^{1/\beta}, \epsilon\}, \quad \forall \epsilon \in [0,1], \; \forall v \in B_{\rho_0}(0,L) \cap V_{\leq \beta}.
 \end{align}
 where $\beta$ is the constant from Definition~\ref{def: constant beta}.
\end{corollary}
\begin{proof}
 Fix $\epsilon \in [0,1]$ and let $v \in B_{\rho_0}(0,\epsilon) \cap V_{\leq \beta}$. Then by Guivarc'h Theorem~\ref{thm: Guivarch}, there exists a constant $C\geq 1$ such that
 \[\rho_\epsilon(0,v) \overset{\eqref{eq: Guivarch equation}}{\leq} C(\rho_0(0,v)+C\epsilon) \leq C^2\epsilon.\]
 This fact, combined with~\eqref{eq: lemma unif bound}, completes the proof.
\end{proof}
\subsection{Proofs of Theorem~\ref{thm: quantitative Pansu} and Theorem~\ref{main thm: pansu}}
\begin{proof}[Proof of Theorem~\ref{thm: quantitative Pansu}]
If $(G,d)$ is a Carnot group, then by Proposition $\ref{prop: if G carnot the the contracted metric are equal}$ we have
\[\rho_\epsilon \overset{\eqref{eq: prop: if G carnot the the contracted metric are equal}}{=} d, \quad \forall \epsilon \in [0,1].\]

So assume that $(G,d)$ is not Carnot. We begin by establishing some bounds. By definition of $\beta$, there is a Carnot quotient ideal (according to Definition~\ref{def: Carnot quotient ideal}) $\mathfrak{i} \triangleleft \mathfrak{g}$ with $\mathfrak{i} \subseteq V_{\leq \beta}$. Let $C \geq 1$ be the constant from Guivarc'h Theorem~\ref{thm: Guivarch}. Let $\| \cdot \|$ be a norm on $\mathfrak{g}$ with unit ball so large that
\begin{equation}\label{eq: Euclidean ball large Pansu}
 B_{\rho_0}(0,2+4C+6C^2) \subseteq B_{\|\cdot\|}(0,1).
\end{equation}
By continuity of the operations and the compactness of the domains, there is an $L\geq 1$ such that
\begin{align*}
 \rho_0(0,v^{-1}*_\epsilon w) \leq L, \quad \forall \epsilon \in [0,1], \;\forall v,w \in \overline{B}_{\norm{\cdot}}(0,1).
\end{align*}
Consequently, by Corollary~\ref{lemma: uniform bound on contracted metrics}, there is a $C'>0$ such that $ \forall\epsilon, \tau \in [0,1]$, $\forall v,w \in B_{\norm{\cdot}}(0,1)$ with $v^{-1}*_\epsilon w \in V_{\leq \beta}$, we have 
\begin{align}\label{eq: Pansu: bound of rho}
 \rho_\tau(0,v^{-1}*_\epsilon w) \overset{\eqref{cor: eq: unif bound on rho_eps}}{\leq} C'\max\{ \norm{v^{-1}*_\epsilon w}^{1/\beta}, \tau\}.
\end{align}
By Lemma~\ref{lem: analyticity of product}, the map $(\lambda,x,y) \mapsto x*_\lambda y$ is analytic on $[0,1] \times B_{\norm{\cdot}}(0,2)^{-1} \times B_{\norm{\cdot}}(0,2)$ (since the series in the Dynkin product is a polynomial, and thus converges globally). Therefore, since analytic maps are Lipschitz on compact sets, there is a constant $\tilde{C}>0$ such that
\begin{align}\label{eq: Pansu: bound difference of products}
 \norm{v^{-1}*_\tau w - v^{-1}*_\tau z} \leq \tilde{C}\norm{w-z}, \quad \forall \tau \in [0,1],\;\forall v,w,z \in \overline{B}_{\norm{\cdot}}(0,1).
\end{align}

Let $K \subseteq \mathfrak{g}$ be a compact set, $\epsilon \in [0,1]$, and $p,q \in K$. We define $L \coloneqq \max\{\rho_\epsilon(p,q), \rho_0(p,q)\}$. Without loss of generality, we may assume $K \subseteq B_{\rho_0}(0,1).$

We prove the \textbf{first inequality}: $\rho_0(p,q) \leq \rho_\epsilon(p,q) +C\epsilon^{\alpha_{\infty}/\beta}.$ Let $\gamma_\epsilon:[0,1]\rightarrow\mathfrak{g}$ be a $\rho_\epsilon$-geodesic from $p$ to $q$ parametrized by constant speed with control $u \in L^2([0,1], \Delta^{(\epsilon)})$. Since $\gamma_\epsilon$ is parametrized by constant speed, we have $\norm{u(t)}^{(\epsilon)} = \rho_\epsilon(p,q) \leq L$ for almost every $t$. We have the uniform bound:
\begin{align}\label{eq: uniform bound of rho_epsilon distance}
 \nonumber\rho_\epsilon(p,q) &\leq \rho_\epsilon(0,p) +\rho_\epsilon(0,q)\\
 \nonumber&\overset{\eqref{eq: Guivarch equation}}{\leq} C\rho_0(0,p)+C\epsilon +\rho_0(0,q)+C\epsilon\\
 &\leq 4C.
\end{align}
We check that independently on $\epsilon \in (0,1)$, the curve $\gamma_\epsilon$ is inside the bounded set $B_{\rho_0}(0,6C^2+C)$:
\begin{align*}\label{eq: bound of rho_0 distance of difference}
 \rho_0(0, \gamma_\epsilon(t)) &\overset{\eqref{eq: Guivarch equation}}{\leq} C\rho_\epsilon(0,\gamma_\epsilon(t))+C\epsilon \\
 &\leq C(\rho_\epsilon(0,p) + \rho_\epsilon(p,\gamma_\epsilon(t)) + C\epsilon \\
 &\overset{\eqref{eq: Guivarch equation}}{\leq} C(\rho_0(0,p) + C\epsilon+\rho_\epsilon(p,q))+C\epsilon\\
 &\overset{\eqref{eq: uniform bound of rho_epsilon distance}}{\leq} 6C^2+C.
\end{align*}

Let $\gamma_0:[0,1] \rightarrow \mathfrak{g}$ be the solution of \begin{equation*}
 \begin{cases}
  \gamma_0(0) = p\\
  \dot{\gamma}_0(t) = (L^{*_0}_{\gamma_0(t)})_*\pi_1(u(t)).
 \end{cases}
\end{equation*}
Note that $\gamma_0$ is a $V_1$-admissible curve in $(\mathfrak{g}_\infty, *_0)$ with control $\pi_1 \circ u$. In particular, the curves $\gamma_0$ and $\pi_1\circ \gamma_\epsilon$ have the same length in their respective spaces: $(\mathfrak{g}_\infty, \rho_0)$ and $(V_1, \norm{\cdot}_\infty)$. Therefore, using in addition that $\pi_1 : (\mathfrak{g}, \rho_\epsilon) \rightarrow (V_1, \norm{\cdot}_\infty)$ is a $1$-Lipschitz map by Proposition~\ref{prop: pi_1 is a submetry and homo}, we have
\begin{align}
 \nonumber L_{\rho_0}(\gamma_0) &= L_{\norm{\cdot}_\infty}(\pi_1 \circ \gamma_\epsilon)\\
 \nonumber &\leq L_{\rho_\epsilon}(\gamma_\epsilon)\\
 \label{eq: length rho_0 < length rho_epsilon}&= \rho_\epsilon(p,q)\\
 \label{eq: lenght gamma_0 wrt rho_0 is bounded} &\overset{\eqref{eq: uniform bound of rho_epsilon distance}}{\leq}4C.
\end{align}
Moreover, since $\gamma_\epsilon$ is parametrized by constant speed, 
\begin{align*}\label{eq: bound on epsilon norm}
 \norm{u(t)}^{(\epsilon)} \overset{\eqref{eq: uniform bound of rho_epsilon distance}}{\leq} 4C, \quad \text{for a.e. } t\in [0,1]
\end{align*}
and, also, 
\begin{align*}
 \rho_0(0,\gamma_0(t)) &\leq \rho_0(0,p)+\rho_0(p,\gamma_0(t)) \\
 &\leq 1+L_{\rho_0}(\gamma_0)\\
 &\overset{\eqref{eq: lenght gamma_0 wrt rho_0 is bounded}}{\leq} 1+4C.
\end{align*}
We have established that the curves $\gamma_\epsilon$ and $\gamma_0$ lie in the bounded set 
\[B_{\rho_0}(0,1+4C+6C^2) \overset{\eqref{eq: Euclidean ball large Pansu}}{\subseteq} B_{\norm{\cdot}}(0,1).\]
In addition we observe that $\gamma_0$ and $\gamma_\epsilon$ are integral curves of the vector fields $X^{\pi_1 \circ u,0}$ and $X^{u,\epsilon}$, respectively, as in Definition~\ref{def: VF Pansu}. Hence, we may apply Grönwall Lemma~\ref{lem: Pansu Gronwall} to bound 
\begin{equation}\label{eq: Pansu: bound of Euclidean difference 1}
 \|\gamma_0(1)-\gamma_\epsilon(1)\| \leq \epsilon^{\alpha_\infty}(\e^{\hat{C}L}-1)
\end{equation}

Next, by Lemma~\ref{lem: projection of curves pansu}, we have $v\coloneqq\gamma_0(1)^{-1}*_0\gamma_\epsilon(1) \in \mathfrak{i}\subseteq V_{\leq \beta}$, and by~\eqref{eq: Pansu: bound difference of products}:
\begin{align} \label{eq: Pansu: bound of v}
 \nonumber\|v\| &= \|L^{*_0}_{\gamma_0(1)^{-1}}(\gamma_0(1)*_0v)-L^{*_0}_{\gamma_0(1)^{-1}}(\gamma_0(1))\| \\
 \nonumber&\overset{\eqref{eq: Pansu: bound difference of products}}{\leq} \tilde{C} \|\gamma_0(1)*_0v - \gamma_0(1)\|\\
 & = \tilde{C}\|\gamma_\epsilon(1)-\gamma_0(1)\|.
\end{align}

Since $v \in V_{\leq \beta}$, by Theorem~\ref{thm: Ball-Box for Carnot} there exists a constant $\bar{C}\geq 1$ depending only on $G, (V_i)_{i=1}^s$, and $\norm{\cdot}$ such that
\begin{equation}\label{eq: Pansu: bound of rho_0}
 \rho_0(0,v) \overset{\eqref{thm: eq: Ball box for Carnot}}{\leq} \bar{C} \norm{v}^{1/\beta}.
\end{equation}
Finally, we obtain the first bound of~\eqref{eq: thm: quantitative pansu}:
\begin{align*}
 \rho_0(p,q) &\leq \rho_0(p,\gamma_0(1))+ \rho_0(\gamma_0(1),\gamma_\epsilon(1)) \\
 &\leq L_{\rho_0}(\gamma_0) + \rho_0(0,v) \\
 &\overset{\eqref{eq: length rho_0 < length rho_epsilon},\eqref{eq: Pansu: bound of rho_0}}{\leq} \rho_\epsilon(p,q) + \bar{C}\|v\|^{1/\beta}\\
 &\overset{\eqref{eq: Pansu: bound of v}}{\leq} \rho_\epsilon(p,q) + \bar{C}\tilde{C}\|\gamma_0(1)-\gamma_\epsilon(1)\|^{1/\beta}\\
 &\overset{\eqref{eq: Pansu: bound of Euclidean difference 1}}{\leq} \rho_\epsilon(p,q) + \bar{C}\tilde{C}(\e^{\hat{C}L}-1)^{1/\beta}\epsilon^{\alpha_\infty/\beta}. 
\end{align*}


We prove the \textbf{second inequality}: $\rho_\epsilon(p,q) \leq \rho_0(p,q) +C\epsilon^{\alpha_{\infty}/\beta}.$ Let $\gamma_0:[0,1] \rightarrow \mathfrak{g}$ be a $\rho_0$-geodesic from $p$ to $q$ parametrized by constant speed with control $v\in L^2([0,1],V_1)$. Since $\gamma_0$ is parametrized by constant speed, we have $\norm{u(t)}^{(0)} = \rho_0(p,q) \leq L$ for almost every $t$. By Proposition~\ref{prop: pi_1 is a submetry and homo}, the projection $\pi_1 : (\mathfrak{g}, \rho_\epsilon) \rightarrow (V_1, \|\cdot\|_\infty)$ is a submetry and therefore, by Proposition~\ref{prop: lift of curves}, we can lift $\pi_1 \circ \gamma_0$ to a curve $\gamma_\epsilon : I \rightarrow \mathfrak{g}$ such that
\begin{equation}\label{eq: Pansu part 2: projection of gamma coincide}
 \pi_1 \circ \gamma_0 = \pi_1 \circ \gamma_\epsilon.
\end{equation} 

We claim that the control of $\gamma_\epsilon$ with respect to $*_\epsilon$ is a function $u \in L^\infty([0,1], \Delta^{(\epsilon)})$ such that
\begin{equation}\label{eq: Pansu part 2: control of gamma_eps}
 \pi_1 \circ u = v, \text{ and }\|u\|^{(\epsilon)} = \|v\| \text{ almost everywhere.}
\end{equation}
Indeed, since $\gamma_\epsilon$ is a rectifiable curve in a sub-Finsler Lie group, it has a measurable control $u : [0,1] \rightarrow \Delta^{(\epsilon)}.$ We show that $\pi_1 \circ u = v$ by proving that both functions are controls of the curve $\pi_1 \circ \gamma_0$ for the abelian product on $V_1$. Recalling that $\pi_1:(\mathfrak{g},*_\epsilon) \rightarrow (V_1, +)$ is a homomorphism by Proposition~\ref{prop: pi_1 is a submetry and homo}, we calculate:
\begin{align*}
 \frac{\dd}{\dd t}\pi_1 \circ \gamma_0(t) &= \dd(\pi_1 \circ L_{\gamma_0(t)}^{*_0})_0v(t)\\
 &=\dd(L^+_{\pi_1(\gamma_0(t))})_0\pi_1\circ v(t)\\
 &=\dd(L^+_{\pi_1(\gamma_0(t))})_0 v(t),
\end{align*}
and
\begin{align*}
 \frac{\dd}{\dd t}\pi_1 \circ \gamma_0(t) &\overset{\eqref{eq: Pansu part 2: projection of gamma coincide}}{=} \frac{\dd}{\dd t}\pi_1 \circ \gamma_\epsilon(t)\\
 &= \dd(\pi_1 \circ L_{\gamma_0(t)}^{*_\epsilon})_0u(t)\\
 &=\dd(L^+_{\pi_1(\gamma_\epsilon(t))})_0\pi_1\circ u(t)\\
 &\overset{\eqref{eq: Pansu part 2: projection of gamma coincide}}{=}\dd(L^+_{\pi_1(\gamma_0(t))})_0 \pi_1\circ u(t).
\end{align*}
Since controls in sub-Finsler Lie groups are unique, we have $\pi_1 \circ u = v$.
Lastly, $\|u\|^{(\epsilon)} = \|v\|$ since the curves $\pi_1 \circ \gamma_0$ and $\gamma_\epsilon$ have the same length in their respective spaces $(V_1, +, \norm{\cdot}_\infty)$ and $(\mathfrak{g}, *_\epsilon, \rho_\epsilon)$. Consequently $u\in L^\infty([0,1], \Delta^{(\epsilon)}).$

In particular, the curves $\gamma_0$ and $\gamma_\epsilon$ have the same length in their respective spaces $(\mathfrak{g}, *_0, \rho_0)$ and $(\mathfrak{g}, *_\epsilon, \rho_\epsilon)$. Hence, we have
\begin{align}\label{eq: length gamma_eps = length gamma_0 <= 2}
 L_{\rho_\epsilon}(\gamma_\epsilon)=L_{\rho_0}(\gamma_0) = \rho_0(p,q) \leq 2,
\end{align}
and in particular $\norm{u}^{(\epsilon)} \overset{\eqref{eq: Pansu part 2: control of gamma_eps}}{=} \norm{\pi_1 \circ u} \leq 2$ almost everywhere.

We claim that the curves $\gamma_0$ and $\gamma_\epsilon$ lie within $B_{\norm{\cdot}}(0,1)$. Indeed, regarding the first curve, we obviously have that $\gamma_0(t) \in B_{\rho_0}(0,2)\subseteq B_{\norm{\cdot}}(0,1)$, since $p,q \in B_{\rho_0}(0,1)$ and $\gamma_0$ is a $\rho_0$-geodesic. Regarding the second curve, we bound:
\begin{align*}
 \rho_0(0, \gamma_\epsilon(t)) &\overset{\eqref{eq: Guivarch equation}}{\leq}C(\rho_\epsilon(0, \gamma_\epsilon(t)) + C\epsilon)\\
 &\leq C(\rho_\epsilon(0,p) + \rho_\epsilon(p,\gamma_\epsilon(t))+ C\epsilon)\\
 &\overset{\eqref{eq: Guivarch equation}}{\leq}C(C\rho_0(0,p)+C\epsilon+L_{\rho_\epsilon}(\gamma_\epsilon)+ C\epsilon)\\
 &\overset{\eqref{eq: length gamma_eps = length gamma_0 <= 2}}{\leq} C(C+C+2+C).
\end{align*}
Therefore $\gamma_\epsilon(t) \in B_{\rho_0}(3C^2+2C) \subseteq B_{\norm{\cdot}}(0,1).$

We proceed similarly as in the first part of the proof. We observe that $\gamma_0$ and $\gamma_\epsilon$ are integral curves of the vector fields $X^{\pi_1 \circ u,0}$ and $X^{u,\epsilon}$, respectively, as in Definition~\ref{def: VF Pansu}. Hence, we may apply Grönwall Lemma~\ref{lem: Pansu Gronwall} to bound 
\begin{equation}\label{eq: Pansu: bound of Euclidean difference 2}
 \|\gamma_0(1)-\gamma_\epsilon(1)\| \leq \epsilon^{\alpha_\infty}(\e^{\hat{C}L}-1).
\end{equation}
By Lemma~\ref{lem: projection of curves pansu}, $w\coloneqq \gamma_0(1)^{-1}*_\epsilon\gamma_\epsilon(1) \in \mathfrak{i} \subseteq V_{\leq \beta}$, and again by~\eqref{eq: Pansu: bound difference of products} we obtain
\begin{align}\label{eq: Pansu: bound of w}
 \nonumber\|w\| &= \norm{L_{\gamma_0^{-1}(1)}^{*_\epsilon}(\gamma_0(1)*_\epsilon w) -L_{\gamma_0^{-1}(1)}^{*_\epsilon}(\gamma_0(1))}\\
 \nonumber&\overset{\eqref{eq: Pansu: bound difference of products}}{\leq}\tilde{C}\norm{\gamma_0(1)*_\epsilon w - \gamma_0(1)}\\
 &=\tilde{C}\|\gamma_0(1)-\gamma_\epsilon(1)\|.
\end{align}
We obtain the second bound of~\eqref{eq: thm: quantitative pansu}: 
\begin{align*}
 \rho_\epsilon(p,q) &\leq \rho_\epsilon(p,\gamma_\epsilon(1))+ \rho_\epsilon(\gamma_0(1),\gamma_\epsilon(1)) \\
 &\leq L_{\rho_\epsilon}(\gamma_\epsilon) + \rho_\epsilon(0,w) \\
 &\overset{\eqref{eq: length gamma_eps = length gamma_0 <= 2},\eqref{eq: Pansu: bound of rho}}{\leq} \rho_0(p,q) + C'\max\{\|w\|^{1/\beta}, \epsilon\}\\
 &\overset{\eqref{eq: Pansu: bound of w}}{\leq} \rho_0(p,q) + C' \tilde{C}\max \{\|\gamma_0(1)-\gamma_\epsilon(1)\|^{1/\beta}, \epsilon\}\\
  &\overset{\eqref{eq: Pansu: bound of Euclidean difference 2}}{\leq} \rho_0(p,q) + C'\tilde{C}\max\{(\e^{\hat{C}L}-1)^{1/\beta}\epsilon^{\alpha_\infty/\beta}, \epsilon\}\\
  &\leq \rho_0(p,q) + C'\tilde{C}\max\{1,(\e^{\hat{C}L}-1)^{1/\beta}\}\max\{\epsilon^{\alpha_\infty/\beta}, \epsilon\}\\
  &\leq \rho_0(p,q) + C'\tilde{C}\max\{1,(\e^{\hat{C}L}-1)^{1/\beta}\}\epsilon^{\alpha_\infty/\beta},
\end{align*}
where we used that $\epsilon \leq 1$, and since $(G,d)$ is not a Carnot group, then by Proposition~\ref{prop: alpha_inf < beta} we have $\alpha_\infty/\beta < 1.$
\end{proof}

Next, we use Theorem~\ref{thm: quantitative Pansu} to prove Theorem~\ref{main thm: pansu}. We stress that an analogous argument can deduce Theorem~\ref{thm: quantitative Pansu} from Theorem~\ref{main thm: pansu}.
\begin{proof}[Proof of Theorem~\ref{main thm: pansu}]
 Let $p\in G, q \in G \setminus B_{d_\infty}(1,1)$, and define $$\epsilon \coloneqq \max\{d_\infty(1,p),d_\infty(1,q)\}^{-1} \leq 1.$$ Then $\delta_\epsilon p, \delta_\epsilon q \in B_{d_\infty}(1,1)$, and thus by Theorem~\ref{thm: quantitative Pansu}, there exists a constant $C>0$ such that
 \begin{align*}
  |d(p,q)- d_\infty(p,q)| &= \frac{1}{\epsilon}|\epsilon d(p,q)- \epsilon d_\infty(p,q)|\\
  &=\frac{1}{\epsilon}|\rho_\epsilon(\delta_\epsilon p,\delta_\epsilon q)- \rho_0(\delta_\epsilon p,\delta_\epsilon q)|\\
  &\overset{\eqref{eq: thm: quantitative pansu}}{\leq} \frac{1}{\epsilon}C \epsilon^{\alpha_\infty/\beta}\\
  &= C\max\{d_\infty(1,p),d_\infty(1,q)\}^{1-\alpha_\infty/\beta},
 \end{align*}
 proving~\eqref{eq: main pansu1}.
\end{proof}

We now explain how Theorem~\ref{thm: quantitative Pansu} implies Pansu Asymptotic Theorem (Theorem~\ref{thm: Pansu asymptotic thm}).
\begin{proof}[Proof of Theorem~\ref{thm: Pansu asymptotic thm}]
  Choose an asymptotic grading of $G$ and consider the Pansu limit metric $d_\infty$ and the contracted metrics $\rho_\epsilon$ on $G$.
  We start by proving that 
 \begin{align}\label{eq: bound_on_balls}
  \text{diam}_{d_\infty}\left(\bigcup_{0< \epsilon < 1}B_{\rho_{\epsilon}}(1,R)\right) < \infty, \quad \forall R >0.
 \end{align}
 
 Indeed, fix $R,\epsilon > 0$, and let $p \in B_{\rho_\epsilon}(1,R)$. By Guivarc'h Theorem~\ref{thm: Guivarch}, there exists a constant $C>0$, which does not depend on $R$ and $\epsilon$, so that we have
 \begin{align*}
  d_\infty(1,p) &\overset{\eqref{eq: Guivarch equation}}{\leq} C \rho_\epsilon(1, p) + C\epsilon
  \leq CR+C\epsilon.
 \end{align*}
 Thus $\bigcup_{0< \epsilon < 1}B_{\rho_{\epsilon}}(1,R) \subseteq B_{d_\infty}(1,C(R+1))$, proving claim~\eqref{eq: bound_on_balls}.

 Next, by Theorem~\ref{thm: quantitative Pansu}, for every non-negative sequence $\epsilon \searrow 0$, the metrics $\rho_\epsilon$ converge uniformly to $\rho_0$ on compact sets, and thus also on $d_\infty$-bounded sets.
 Hence, Proposition~\ref{prop: criterion_GH_conv} implies that 
 $$(G,\rho_\epsilon,1) \xrightarrow{\rm GH} (G,d_\infty,1), \quad {\rm as }\; \epsilon \rightarrow 0.$$
 We conclude by recalling that, by Remark~\ref{rem: dilations are isometries asymptotic}, $(G,\rho_\epsilon,1)$ and $(G,\epsilon d,1)$ are isometric, for every $\epsilon \in (0,1).$
\end{proof}

\section{Quantitative Mitchell Tangent Theorem}\label{section: quantitative mitchell}
Theorem~\ref{main thm: mitchell} follows from the following result comparing dilated metrics. After completing the proof of Theorem~\ref{main thm: mitchell}, we show how it implies Theorem~\ref{thm: Mitchell tangent thm}.
\begin{thm}\label{thm: quantitative Mitchell}
 Let $G$ be a sub-Finsler Lie group. With respect to a tangent grading, consider the tangent metric $d_0$ on $G$ and the dilated metrics $d_\epsilon$. Then, there are a neighborhood $\Omega$ of $1$ in $G$ and $C>0$ such that 
 \begin{align}\label{thm: eq: quantitative mitchell0}
     -Cd_\epsilon(p,q)^{1/\beta}\epsilon^{\alpha_0/\beta}&\leq d_\epsilon(p,q)-d_0(p,q)
     \leq Cd_0(p,q)^{1/\beta}\epsilon^{\alpha_0/\beta}, \quad \forall \epsilon \in [0,1], \; \forall p,q \in \Omega,
 \end{align}
  where $\alpha_{0}$ and $\beta$ are the constants from Definition~\ref{def: constant alpha mitchell} and~\ref{def: constant beta}.
\end{thm}

For the rest of this section, we are in the assumptions of Theorem~\ref{thm: quantitative Mitchell}. In addition, we denote the sub-Finsler structure associated to the metrics $d_\epsilon$, for $\epsilon \in [0,1]$, by $(\Delta,\norm{\cdot})$, and we denote the tangent grading by $(W_1,\ldots,W_s)$.

Our proof of Theorem~\ref{thm: quantitative Mitchell} follows the same strategy as that of Theorem~\ref{thm: quantitative Pansu}. Therefore, we need to restate and prove the lemmata required for Theorem \ref{thm: quantitative Pansu} in the context of tangent gradings and dilated metrics.

\begin{lem}\label{lem: projection of curves mitchell}
 Let $\epsilon \in (0,1]$, let $\gamma_\epsilon$ be a $d_\epsilon$-admissible curve with control $u \in L^2([0,1],\Delta)$ such that the solution $\gamma_0$ of \begin{equation*}
 \begin{cases}
  \gamma_0(0) = \gamma_\epsilon(0)\\
  \dot{\gamma}_0(t) = (L^{\cdot_0}_{\gamma_0(t)})_*u(t)
 \end{cases}
\end{equation*}
is well defined. Let $\mathfrak{i}$ be a Carnot quotient ideal (according to Definition~\ref{def: Carnot quotient ideal}). Denote by $\pi : \mathfrak{g} \rightarrow \mathfrak{g}/\mathfrak{i}$ the quotient map. Then 
\begin{enumerate}
 \item $\pi \circ \gamma_0 = \pi \circ \gamma_\epsilon$.
 \item $\gamma_0(1)-\gamma_\epsilon(1), \; \gamma_0(1)^{-1} \cdot_0\gamma_\epsilon(1) , \;\gamma_0(1)^{-1} \cdot_\epsilon\gamma_\epsilon(1) \in \mathfrak{i}$.
\end{enumerate} 
\end{lem}
\begin{proof}We follow the same strategy of the proof of Lemma~\ref{lem: projection of curves pansu}.

(1). It is sufficient to show that the projected curves are solutions of the same Cauchy problem. Recall that $\mathfrak{i}$ is an ideal for each Lie algebra $(\mathfrak{g},\llbracket\cdot,\cdot\rrbracket^{(\epsilon)})$, see Lemma~\ref{lem: properties of CQI associated to ALG}\eqref{lem: properties of CQI associated to ALG 2}. Moreover, the quotient bracket is independent of $\epsilon$, see Lemma~\ref{lem: properties of CQI associated to ALG}\eqref{lem: properties of CQI associated to ALG 4}. We denote by $\cdot_\mathfrak{i}$ the Dynkin product on $\mathfrak{g}/\mathfrak{i}$, and we denote the left translation by an element $v \in \mathfrak{g}/\mathfrak{i}$ with respect to this product by $L_v^{\mathfrak{i}}$. Observe also that by Lemma~\ref{lem: properties of CQI associated to ALG}\eqref{lem: properties of CQI associated to ALG 5}, the quotient map $\pi : \mathfrak{g} \rightarrow \mathfrak{g}/\mathfrak{i}$ is a local Lie group homomorphism. Therefore, for almost every $t\in [0,1]$ we have
 \begin{align*}
  \frac{\dd}{\dd t}{(\pi \circ \gamma_\epsilon)}(t) &= (\dd\pi)\dot{\gamma_\epsilon}(t) \\
  &= (\dd\pi)(\dd L_{\gamma_\epsilon(t)}^{\cdot_\epsilon})u(t) \\
  &=\dd(\pi \circ L_{\gamma_\epsilon(t)}^{\cdot_\epsilon})u(t) \\
  &=\dd(L_{\pi(\gamma_\epsilon(t))}^{\mathfrak{i}}\circ \pi)u(t),
 \end{align*}
 and
 \begin{align*}
  \frac{\dd}{\dd t}{(\pi \circ \gamma_0)}(t) &= (\dd\pi)\dot{\gamma_0}(t) \\
  &= (\dd\pi)(\dd L_{\gamma_0(t)}^{\cdot_0})u(t) \\
  &=\dd(\pi \circ L_{\gamma_0(t)}^{\cdot_0})u(t) \\
  &= \dd(L_{\pi(\gamma_0(t))}^{\mathfrak{i}}\circ \pi) u(t).
 \end{align*}

 (2). The fact that $\gamma_0(1)-\gamma_\epsilon(1) \in \mathfrak{i}$ is an immediate consequence of (1). Using (1) and the fact that $\pi$ is a Lie group homomorphism, we obtain for $\tau \in \{0, \epsilon\}$:
 \begin{align*}
 \pi(\gamma_0(1)^{-1}&\cdot_\tau \gamma_\epsilon(1)) = (\pi \circ\gamma_0(1))^{-1} \cdot_{\mathfrak{i}} (\pi \circ\gamma_\epsilon(1)) = 0.
 \end{align*}
 Therefore $\gamma_0(1)^{-1} \cdot_\tau \gamma_\epsilon(1) \in \mathfrak{i}.$
\end{proof}

We also need a uniform bound similar to the one of Lemma~\ref{lemma: uniform bound on contracted metrics}, but for dilated metrics.
\begin{lem}\label{lemma: uniform bound on dilated metrics}
 For every norm $\norm{\cdot}$ on $\mathfrak{g}$, there are constants $R>0$, $C>0$ such that
 \begin{align}\label{eq: lemma unif bound on dilated}
  d_\epsilon(0,v) \leq C \norm{v}^{1/\beta}, \quad \forall \epsilon \in [0,1],\;\forall v \in B_{d_0}(0,R) \cap W_{\leq \beta}.
 \end{align}
 where $\beta$ is the constant from Definition~\ref{def: constant beta}.
\end{lem}
\begin{proof}Fix $R>0$ small enough so that $\exp|_{B_{d_0}(0,R)}$ is a diffeomorphism with its image. Let $B \coloneqq B_{d_0}(0,R) \cap W_{\leq \beta}.$ Since $B \subseteq W_{\leq \beta}$, by the Ball-Box Theorem for Carnot groups, see Theorem~\ref{thm: Ball-Box for Carnot}, there is a constant $C_1$ such that
 \begin{equation}\label{lem: bound on dilated metric equations}
  d_0(0,v) \leq C_1 \norm{v}^{1/\beta}, \quad \forall v \in B.
 \end{equation}
The claim then follows by additionally applying \eqref{eq: dilated metric are bilipschitz}: there exists a constant $C_2$ such that
\begin{equation*}
 d_\epsilon(0,v)\overset{\eqref{eq: dilated metric are bilipschitz}}{\leq}C_2d_0(0,v) \overset{\eqref{lem: bound on dilated metric equations}}{\leq}C_1C_2\norm{v}^{1/\beta}, \quad \forall v \in B.
\end{equation*}
\end{proof}
\subsubsection{A geodetically connected subset}\label{subsection: geodetically connected set}
For every $R>0$ small enough so that $\exp|_{B_{d_0}(0,R)}$ is a diffeomorphism, we construct a set $\tilde{\Omega}$ that is $d_\epsilon$-geodetically connected within $B_{d_0}(0,R)$, for all $\epsilon \in [0,1]$, in the sense that for every $p,q \in \tilde{\Omega}$ there is a $d_\epsilon$-geodesic valued in $B_{d_0}(0,R).$

Fix $R>0$ small enough, and let $C>1$ be the constant from~\eqref{eq: dilated metric are bilipschitz}. Define
\begin{equation*}\label{eq: def geodetically connected set}
 \tilde{\Omega} \coloneqq B_{d_0}(0,r), \quad \text{with } r=r(R)\coloneqq \frac{R}{4C^2}. 
\end{equation*}
We claim that $\tilde{\Omega}$ is $d_\epsilon$-geodetically connected within $B_{d_0}(0,R)$, for all $\epsilon \in [0,1]$. Indeed, fix $p,q \in \tilde{\Omega}$. Assume $\epsilon \neq 0$, and let $\tilde{\gamma} : [0,1] \rightarrow (G,d)$ be a geodesic from $\delta_\epsilon(p)$ to $\delta_\epsilon(q)$. We claim that 
\begin{equation}\label{eq: geodetically connected subset 1}
    \tilde{\gamma}(t) \in B_{d_0}(0, R\epsilon), \quad \forall t \in [0,1].
\end{equation}
Indeed, we show that as long as $\tilde{\gamma}(t)$ lies within $B_{d_0}(0,R)$, it actually lies in $B_{d_0}(0,3C^2r\epsilon)$: suppose that $\tilde{\gamma}(t) \in B_{d_0}(0,R)$, then
\begin{align*}
 d_0(0,\tilde{\gamma}(t)) &\overset{\eqref{eq: dilated metric are bilipschitz}}{\leq}Cd(0,\tilde{\gamma}(t))\\
 &\leq C(d(0,\delta_\epsilon p)+ L_{d}(\tilde{\gamma}))\\
 &=C(d(0,\delta_\epsilon p)+ d(\delta_\epsilon p,\delta_\epsilon q))\\
 &\leq C(2d(0,\delta_\epsilon p)+ d(0,\delta_\epsilon q))\\
 &\overset{\eqref{eq: dilated metric are bilipschitz}}{\leq}C^2(2d_0(0,\delta_\epsilon p)+d_0(0,\delta_\epsilon q))\\
 &\leq 3C^2r\epsilon.
\end{align*}
Since $\tilde{\gamma}$ is continous and $\tilde{\gamma}(0)=\delta_\epsilon p \in B_{d_0}(0,R)$, it follows that 
\[\tilde{\gamma}(t) \in B_{d_0}(0,3C^2r\epsilon) \subseteq B_{d_0}(0,R\epsilon), \quad \forall t \in [0,1],\]
proving \eqref{eq: geodetically connected subset 1}.
Therefore, since $\delta_\epsilon^{-1} : (B_{d_0}(0,R\epsilon),\epsilon^{-1}d) \rightarrow (B_{d_0}(0,R),d_\epsilon)$ is an isometry, the curve $\gamma \coloneqq \delta_\epsilon^{-1} \circ \tilde{\gamma}$ is a $d_\epsilon$-geodesic from $p$ to $q$ in $B_{d_0}(0,R)$. 
Next assume $\epsilon = 0$, and let $\gamma_0$ be a $d_0$-geodesic between $p$ and $q$. Then by triangle inequality:
\begin{align*}
 d_0(0, \gamma_0(t)) \leq 2r < R.
\end{align*}
We deduced that $\tilde{\Omega}$ has the specified property.

\subsection{Proofs of Theorem~\ref{thm: quantitative Mitchell} and Theorem~\ref{main thm: mitchell}}
\begin{proof}[Proof of Theorem~\ref{thm: quantitative Mitchell}]
We start by constructing the set $\Omega$. Fix $R$ small enough so that $\exp|_{B_{d_0}(0,R)}$ is a diffeomorphism and the map $(\epsilon, p, q) \mapsto p^{-1} \cdot_\epsilon q$ is analytic on $V \coloneqq [-1,1] \times \overline{B_{d_0}(0,R)}\times \overline{B_{d_0}(0,R)}$ (such $R$ exists by Lemma~\ref{lem: brackets extend to eps=0} and Lemma~\ref{lem: analyticity of product}). By definition of $\beta$, there is a Carnot quotient ideal (according to Definition~\ref{def: Carnot quotient ideal}) $\mathfrak{i} \triangleleft \mathfrak{g}$ with $\mathfrak{i} \subseteq W_{\leq \beta}$. Let $\| \cdot \|$ be a norm on $\mathfrak{g}$ with unit ball so large that $B_{d_0}(0,R) \subseteq B_{\|\cdot\|}(0,1)$. Let $C >1$ be the constant from~\eqref{eq: dilated metric are bilipschitz}.

We claim that there is an $r>0$ with $r(C^2+2C+1) <R$ such that:
\begin{equation}\label{claim: product lies in small ball Mitchell}
 p^{-1} \cdot_\epsilon q \in B_{d_0}(0,R), \quad \forall p,q \in B_{d_0}(0,r(C^2+2C+1)), \; \forall \epsilon \in [-1,1].
\end{equation}
Indeed, since the map $(\epsilon,p,q) \mapsto p^{-1}\cdot_\epsilon q$ is analytic on the compact set $V$, there exists a constant $\tilde{C}$ so that $(\epsilon,p,q) \mapsto p^{-1}\cdot_\epsilon q$ is $\tilde{C}$-Lipschitz on $V$. In addition, by the Ball-Box Theorem for Carnot group, see Theorem~\ref{thm: Ball-Box for Carnot}, there is a constant $\tilde{D}\geq 1$ such that $\forall (\epsilon,p,q) \in V$ we have
\begin{align*}
 d_0(0,p^{-1}\cdot_\epsilon q)&\overset{\eqref{thm: eq: Ball box for Carnot}}{\leq}\tilde{D}\norm{p^{-1} \cdot_\epsilon q}^{1/s} \\
 &= \tilde{D}\norm{p^{-1} \cdot_\epsilon q - 0 \cdot_\epsilon 0}^{1/s} \\
 &\leq \tilde{C}\tilde{D}(\norm{p^{-1}}+\norm{q})^{1/s}\\
 &\overset{\eqref{thm: eq: Ball box for Carnot}}{\leq}\tilde{C}\tilde{D}^2(d_0(0,p^{-1})+ d_0(0,q))^{1/s}\\
 &\leq 2\tilde{C}\tilde{D}^2\max\{d_0(0,p),d_0(0,q)\}^{1/s}.
\end{align*}
Therefore there is an $r>0$ small enough such that claim~\eqref{claim: product lies in small ball Mitchell} holds.

Let $\hat{C}$ be the constant from Lemma~\ref{lem: Mitchell Gronwall} and let $C_*$ be a constant such that 
\begin{align}\label{def: constant C_*}
    x<2Cr \implies \e^{\hat{C}x}-1< C_*x.
\end{align}

Next, let $r' \coloneqq \frac{r}{4C^2}$. Then the ball $B_{d_0}(0,r')$ is $d_\epsilon$-geodetically connected within $B_{d_0}(0,r)$ for every $\epsilon$, by the discussion of Section~\ref{subsection: geodetically connected set}. We define $\Omega \coloneqq B_{d_0}(0,r').$

Since $(\epsilon,p,q) \mapsto p^{-1}\cdot_\epsilon q$ is $\tilde{C}$-Lipschitz on $V$ and $r(C^2+2C+1) <R$, we have:
\begin{align}\label{eq: Mitchell: bound difference of products}
 \norm{v^{-1}\cdot_\tau w - v^{-1}\cdot_\tau z} \leq \tilde{C}\norm{w-z}, \quad \forall \tau \in [-1,1], \; \forall v,w,z \in B_{d_0}(0,r(C^2+2C+1)).
\end{align}

 From now on, fix $p,q \in \Omega$ and $\epsilon \in [0,1].$

We prove the \textbf{first inequality}: $d_0(p,q) \leq d_\epsilon(p,q) +C\epsilon^{\alpha_{0}/\beta}.$ Let $\gamma_\epsilon:[0,1]\rightarrow B_{d_0}(0,r)$ be a $d_\epsilon$-geodesic from $p$ to $q$ parametrized by constant speed with control $u \in L^2([0,1], \Delta)$, for which $d_\epsilon(p,q) = L_{d_\epsilon}(\gamma_\epsilon) = \norm{u}_2 \eqqcolon L< 2Cr$. Let $\gamma_0:[0,1] \rightarrow \mathfrak{g}$ be the solution of \begin{equation*}
 \begin{cases}
  \gamma_0(0) = \gamma_\epsilon(0)\\
  \dot{\gamma}_0(t) = (L^{\cdot_0}_{\gamma_0(t)})_*u(t).
 \end{cases}
\end{equation*}
Notice that
\begin{align}
 \label{eq: Mitchell: length of curves is the same 1}L_{d_0}(\gamma_0) &= L_{d_\epsilon}(\gamma_\epsilon)\\
 \nonumber&=d_\epsilon(p,q)\\
 \nonumber&\leq d_\epsilon(0,p)+d_\epsilon(0,q)\\
 \nonumber&\overset{\eqref{eq: dilated metric are bilipschitz}}{\leq}C(d_0(0,p)+d_0(0,q))\\
 \nonumber&\leq 2Cr'.
\end{align}
Hence $d_0(0,\gamma_0(t)) \leq d_0(0,p) + d_0(p,\gamma_0(t))\leq r' + 2Cr' < r(2C+1)<R$.

We have established that the curves $\gamma_0$ and $\gamma_\epsilon$ lie in the bounded set 
$$B_{d_0}(0,r(2C+1)) \subseteq B_{d_0}(0,R) \subseteq B_{\norm{\cdot}}(0,1).$$
In addition we observe that $\gamma_0$ and $\gamma_\epsilon$ are integral curves of the vector fields $Y^{u,0}$ and $Y^{u,\epsilon}$ respectively, as in Definition~\ref{def: VF Mitchell}.
Hence, we may use Grönwall Lemma~\ref{lem: Mitchell Gronwall} to bound
\begin{equation}\label{eq: bound of euclidean difference Mitchell1}
 \norm{\gamma_0(1)- \gamma_\epsilon(1)} \leq \epsilon^{\alpha_0}(\e^{\hat{C}L}-1) \overset{\eqref{def: constant C_*}}{\leq} C_*L\epsilon^{\alpha_0}, 
\end{equation}
for some constant $\hat{C}>0$.

Next, by Lemma~\ref{lem: projection of curves mitchell}: $\gamma_0^{-1}(1) \cdot_0 \gamma_\epsilon(1) \in \mathfrak{i} \subseteq W_{\leq \beta}$. Therefore, by Theorem~\ref{thm: Ball-Box for Carnot} there exists a constant $C'$ depending only on $G, (W_i)_{i=1}^s,$ and $\norm{\cdot}$ such that
\begin{equation}\label{eq: Mitchell: bound of dist of curves1}
 d_0(0,\gamma_0^{-1}(1) \cdot_0 \gamma_\epsilon(1)) \overset{\eqref{thm: eq: Ball box for Carnot}}{\leq} C'\norm{\gamma_0^{-1}(1) \cdot_0 \gamma_\epsilon(1)}^{1/\beta}.
\end{equation}

We need one last bound, which we obtain by applying~\eqref{eq: Mitchell: bound difference of products}:
\begin{align}\label{eq: Mitchell: bound of norm of curves1}
 \nonumber \norm{\gamma_0^{-1}(1) \cdot_0\gamma_\epsilon(1)} & = \norm{L_{\gamma_0^{-1}(1)}^{\cdot_0}( \gamma_0(1) \cdot_0 \gamma_0^{-1}(1) \cdot_0\gamma_\epsilon(1)) - L_{\gamma_0^{-1}(1)}^{\cdot_0} (\gamma_0(1)) } \\
 \nonumber&\overset{\eqref{eq: Mitchell: bound difference of products}}{\leq}\tilde{C}\norm{\gamma_0(1) \cdot_0 \gamma_0^{-1}(1) \cdot_0\gamma_\epsilon(1) - \gamma_0(1)}\\
 &= \tilde{C}\norm{\gamma_\epsilon(1)- \gamma_0(1)}.
\end{align}

Finally, we obtain the first inequality:
\begin{align*}
 d_0(p,q) &\leq d_0(p,\gamma_0(1))+ d_0(\gamma_0(1),\gamma_\epsilon(1)) \\
 &\leq L_{d_0}(\gamma_0) + d_0(0,\gamma_0^{-1}(1) \cdot_0\gamma_\epsilon(1)) \\
 &\overset{\eqref{eq: Mitchell: length of curves is the same 1},\eqref{eq: Mitchell: bound of dist of curves1}}{\leq} L_{d_\epsilon}(\gamma_\epsilon) + C'\|\gamma_0^{-1}(1) \cdot_0\gamma_\epsilon(1)\|^{1/\beta}\\
 &\overset{\eqref{eq: Mitchell: bound of norm of curves1}}{\leq} d_\epsilon(p,q) + C'\tilde{C}\|\gamma_0(1)-\gamma_\epsilon(1)\|^{1/\beta}\\
 &\overset{\eqref{eq: bound of euclidean difference Mitchell1}}{\leq} d_\epsilon(p,q) +  C'\tilde{C}C_*^{1/\beta}L^{1/\beta}\epsilon^{\alpha_{0}/\beta}.
\end{align*}

We prove the \textbf{second inequality}: $d_\epsilon(p,q) \leq d_0(p,q) +C\epsilon^{\alpha_{0}/\beta}.$ Let $\gamma_0:[0,1] \rightarrow B_{d_0}(0,r)$ be a $d_0$-geodesic from $p$ to $q$ parametrized by constant speed with control $v\in L^2([0,1],\Delta)$, for which $d_0(p,q) =L_{d_0}(\gamma_0) = \norm{v}_2 \eqqcolon L' < 2r$. We proceed as in the first inequality; let $\gamma_\epsilon:[0,1] \rightarrow \mathfrak{g}$ be the solution of \begin{equation*}
 \begin{cases}
  \gamma_\epsilon(0) = \gamma_0(0)\\
  \dot{\gamma}_\epsilon(t) = (L^{\cdot_\epsilon}_{\gamma_\epsilon(t)})_*v(t).
 \end{cases}
\end{equation*}
We claim that
\begin{equation}\label{eq: Mitchell: claim gamma epsilon contained}
 \gamma_\epsilon(t) \in B_{d_0}(0,r'(C^2+2C)), \; \forall t.
\end{equation}

Indeed,
\begin{align}
 \label{eq: Mithell: length of curves is the same2}L_{d_\epsilon}(\gamma_\epsilon) &= L_{d_0}(\gamma_0)\\
 \nonumber&\leq d_0(p,q)<2r',
\end{align}
and
\begin{align*}
 d_\epsilon(0,p) \overset{\eqref{eq: dilated metric are bilipschitz}}{\leq}Cd_0(0,p) < Cr'.
\end{align*}
Hence, \begin{align*}
 d_\epsilon(0, \gamma_\epsilon(t))
 \leq d_\epsilon(0, p) + d_\epsilon(p, \gamma_\epsilon(t))
 \leq Cr' + L_{d_\epsilon}(\gamma_\epsilon)\leq r'(C+2),
\end{align*}
and thus $\gamma_\epsilon$ is contained in $B_{d_\epsilon}(0,r'(C+2))$. Since $Cr'(C+2)=r'(C^2+2C) < R,$ by~\eqref{cor: equicontinuity of distances} we have
\[B_{d_\epsilon}(0,r'(C+2)) \overset{\eqref{cor: equicontinuity of distances}}{\subseteq} B_{d_0}(0,r'(C^2+2C)),\]
proving claim~\eqref{eq: Mitchell: claim gamma epsilon contained}.

We proceed by establishing the necessary bounds as in the first part of the proof.

We observe that both $\gamma_0$ and $\gamma_\epsilon$ are contained in $B_{d_0}(0,r(C^2+2C)) \subseteq B_{d_0}(0,R) \subseteq B_{\norm{\cdot}}(0,1)$, and that $\gamma_0$ and $\gamma_\epsilon$ are integral curves of the vector fields $Y^{u,0}$ and $Y^{u, \epsilon}$ respectively, as in Definition~\ref{def: VF Mitchell}. Hence, we may use Grönwall Lemma~\ref{lem: Mitchell Gronwall} to bound
\begin{equation}\label{eq: bound of euclidean difference Mitchell2}
 \norm{\gamma_0(1)- \gamma_\epsilon(1)} \leq  \epsilon^{\alpha_0}(\e^{\hat{C}L'}-1) \overset{\eqref{def: constant C_*}}{\leq} C_*L'\e^{\alpha_0}.  
\end{equation}

Next, since $\gamma_0$ and $\gamma_\epsilon$ are contained in $B_{d_0}(0,r(C^2+2C))$, we have
\[\gamma_0^{-1}(1) \cdot_0 \gamma_\epsilon(1) \overset{\eqref{claim: product lies in small ball Mitchell}}{\in} B_{d_0}(0,R).\] 
In addition, by Lemma~\ref{lem: projection of curves mitchell}, $\gamma_0^{-1}(1) \cdot_0 \gamma_\epsilon(1) \in \mathfrak{i} \subseteq W_{\leq \beta}$. Therefore, we may apply Lemma~\ref{lemma: uniform bound on dilated metrics} to obtain the bound
\begin{equation}\label{eq: Mitchell: bound of dist of curves2}
 d_0(0,\gamma_0^{-1}(1) \cdot_0 \gamma_\epsilon(1)) \overset{\eqref{eq: lemma unif bound on dilated}}{\leq} C'\norm{\gamma_0^{-1}(1) \cdot_0 \gamma_\epsilon(1)}^{1/\beta},
\end{equation}
for some constant $C'$ that does not depend on $\epsilon$.

We need one last bound, which we obtain by applying~\eqref{eq: Mitchell: bound difference of products}:
\begin{align}\label{eq: Mitchell: bound of norm of curves2}
 \nonumber \norm{\gamma_0^{-1}(1) \cdot_\epsilon\gamma_\epsilon(1)} & = \norm{L_{\gamma_0^{-1}(1)}^{\cdot_0}( \gamma_0(1) \cdot_0 \gamma_0^{-1}(1) \cdot_\epsilon\gamma_\epsilon(1)) - L_{\gamma_0^{-1}(1)}^{\cdot_0} (\gamma_0(1)) } \\
 \nonumber&\overset{\eqref{eq: Mitchell: bound difference of products}}{\leq}\tilde{C}\norm{\gamma_0(1) \cdot_0 \gamma_0^{-1}(1) \cdot_\epsilon\gamma_\epsilon(1) - \gamma_0(1)}\\
 &= \tilde{C}\norm{\gamma_\epsilon(1)- \gamma_0(1)}.
\end{align}
We obtain the second inequality:
\begin{align*}
 d_\epsilon(p,q) &\leq d_\epsilon(p,\gamma_\epsilon(1))+ d_\epsilon(\gamma_0(1),\gamma_\epsilon(1)) \\
 &\leq L_{d_\epsilon}(\gamma_\epsilon) + d_\epsilon(0,\gamma_0^{-1}(1) \cdot_\epsilon \gamma_\epsilon(1)) \\
 &\overset{\eqref{eq: Mithell: length of curves is the same2},\eqref{eq: Mitchell: bound of dist of curves2}}{\leq} L_{d_0}(\gamma_0) + C'\|\gamma_0^{-1}(1) \cdot_\epsilon \gamma_\epsilon(1)\|^{1/\beta}\\
 &\overset{\eqref{eq: Mitchell: bound of norm of curves2}}{\leq} d_0(p,q) + C' \tilde{C}\|\gamma_0(1)-\gamma_\epsilon(1)\|^{1/\beta}\\
  &\overset{\eqref{eq: bound of euclidean difference Mitchell2}}{\leq} d_0(p,q) + C'\tilde{C}C_*^{1/\beta}L'^{1/\beta} \epsilon^{\alpha_{0}/\beta}.
\end{align*}
\end{proof}

Next, we use Theorem~\ref{thm: quantitative Mitchell} to prove Theorem~\ref{main thm: mitchell}. We stress that an analogous argument can deduce Theorem~\ref{thm: quantitative Mitchell} from Theorem~\ref{main thm: mitchell}.
\begin{proof}[Proof of Theorem~\ref{main thm: mitchell}]
 Let $\Omega$ be the set from Theorem~\ref{thm: quantitative Mitchell}, let $r\in(0,1]$ such that $B_{d_0}(1,2r) \subseteq \Omega$, and let $p,q \in B_{d_0}(1,r)$ with $d_0(1,p) \geq d_0(1,q)$. We define $\epsilon \coloneqq \frac{1}{r}d_0(1,p)$. Observe that $\delta_\epsilon^{-1}p, \delta_\epsilon^{-1}q \in B_{d_0}(1,2r) \subseteq \Omega$. Indeed,
 \[d_0(1, \delta_\epsilon^{-1}p) = rd_0(1,p)^{-1}d_0(1,p) < 2r,\]
 and similarly 
 \[d_0(1, \delta_\epsilon^{-1}q) = rd_0(1,p)^{-1}d_0(1,q)<2r.\]Moreover we have $\epsilon \leq 1$. Therefore, by Theorem~\ref{thm: quantitative Mitchell} there is a $C>0$ so that
 \begin{align*}
     d(p,q)-d_0(p,q) &= \epsilon(d_\epsilon(\delta_\epsilon^{-1}p, \delta_\epsilon^{-1}q)-d_0(\delta_\epsilon^{-1}p, \delta_\epsilon^{-1}q))\\
     &\overset{\eqref{thm: eq: quantitative mitchell0}}{\geq} -Cd_\epsilon(\delta_\epsilon^{-1}p, \delta_\epsilon^{-1}q)^{1/\beta}\epsilon^{1+\alpha_0/\beta}\\
     &=-Cd(p,q)^{1/\beta}\epsilon^{1+(\alpha_0-1)/\beta}\\
     &=-Cr^{-1-(\alpha_0-1)/\beta}d(p,q)^{1/\beta} d_0(1,p)^{1+(\alpha_0-1)/\beta},
 \end{align*}
 and
  \begin{align*}
     d(p,q)-d_0(p,q) &= \epsilon(d_\epsilon(\delta_\epsilon^{-1}p, \delta_\epsilon^{-1}q)-d_0(\delta_\epsilon^{-1}p, \delta_\epsilon^{-1}q))\\
     &\overset{\eqref{thm: eq: quantitative mitchell0}}{\leq} Cd_0(\delta_\epsilon^{-1}p, \delta_\epsilon^{-1}q)^{1/\beta}\epsilon^{1+\alpha_0/\beta}\\
     &=Cd_0(p,q)^{1/\beta}\epsilon^{1+(\alpha_0-1)/\beta}\\
     &=Cr^{-1-(\alpha_0-1)/\beta}d_0(p,q)^{1/\beta} d_0(1,p)^{1+(\alpha_0-1)/\beta}.
 \end{align*}
\end{proof}
We also obtain the following corollary:
\begin{corollary}\label{cor: mitchell}
 Let $(G,d)$ be a sub-Finsler Lie group equipped with an $s$-step distribution. With respect to a tangent grading, consider the tangent metric $d_0$ on a neighborhood of the identity in $G$. Then there are constants $r\in (0,1]$ and $C>0$ such that
 \begin{align}\label{eq: cor mitchell2}
  |d(p,q)-d_0(p,q)| \leq C\max\{d_0(1,p),d_0(1,q)\}^{1+1/s}, \quad \forall p,q \in B_{d_0}(1,r).
 \end{align}
\end{corollary}
\begin{proof}
 It is a consequence of Theorem~\ref{main thm: mitchell} and the fact that by definitions of $\alpha_0$ and $\beta$, we have $\alpha_0/\beta \geq 1/s$ for every tangent grading.
\end{proof}

We now explain how Theorem~\ref{main thm: mitchell} implies Mitchell Tangent Theorem (Theorem~\ref{thm: Mitchell tangent thm}) and the quantitative estimate~\eqref{eq: thm Mitchell tangent thm}.
\begin{proof}[Proof of Theorem~\ref{thm: Mitchell tangent thm}]
Let $U \subseteq G$ be a neighborhood of the identity where the estimate~\eqref{eq: cor mitchell1} holds. For $\epsilon \in (0,1]$, consider the dilations 
\[\delta_{1/\epsilon}: U \rightarrow G_0,\]
as in Section~\ref{subsection: gradings}, where we identify $U$ with a neighborhood of the identity in $G_0$ as in Section~\ref{subsection: osculating carnot group}.

We show that $\forall R>0$, $\forall \tau >0$, the map 
\[\delta_{1/\epsilon}:\left(B_{\frac{1}{\epsilon}d}(1,R), \frac{1}{\epsilon} d\right) \rightarrow (G_0,d_0)\]
is a $(1,\tau)$-quasi-isometric embedding, for $\epsilon >0$ small enough.
Indeed, fix $R>0$. Then for $\epsilon >0$ small enough so that
\[B_{\frac{1}{\epsilon}d}(1,R) = B_d(1,\epsilon R) \subseteq U,\]
and for $p,q \in B_{\frac{1}{\epsilon}d}(1,R)$, by Theorem~\ref{main thm: mitchell} there is a $C$ such that
\begin{align*}
    \left| \frac{1}{\epsilon}d(p,q) - d_0(\delta_{1/\epsilon}p, \delta_{1/\epsilon}q)\right| &= \frac{1}{\epsilon}\left| d(p,q) - d_0(p,q)\right|\\
    &\overset{\eqref{eq: cor mitchell1}}{\leq} C\frac{1}{\epsilon}(\epsilon R)^{1+\alpha_0/\beta} \rightarrow 0, \quad \text{ as } \epsilon \rightarrow 0.
\end{align*}

Next, we show $\forall R >0$, $\forall \tau >0$, and for $\epsilon$ small enough
\[B_{d_0}(1,R) \subseteq \delta_{1/\epsilon}\left(B_{\frac{1}{\epsilon}d}(1,R + \tau)\right).\]
Indeed, let $q \in B_{d_0}(1,R)$ and consider $p\coloneqq \delta_\epsilon q \in B_{d_0}(1,\epsilon R)$, for $\epsilon >0 $ small enough so that $p\in \Omega$. Then by Theorem~\ref{main thm: mitchell}, there is a $C$ such that
\begin{align*}
    \frac{1}{\epsilon}d(1,p) \overset{\eqref{eq: cor mitchell1}}{\leq} d_0(1,q) + CR^{1+\alpha_0/\beta}\epsilon^{\alpha_0/\beta} \rightarrow d_0(1,q), \quad \text{ as } \epsilon \rightarrow 0.
\end{align*}

Therefore $(G, \frac{1}{\epsilon}d)\overset{\text{GH}}{\rightarrow}(G_0,d_0)$ as $\epsilon \rightarrow 0$.
\end{proof}

\section{Applications and examples}\label{section: applications}
In this final section, we provide examples illustrating how our new results yield sharper bounds than those obtained by previous results in the literature. The first family of examples estimates the rate of convergence to the asymptotic cone of products of the form $G \times H$ between a Carnot group $G$ and a simply connected nilpotent group $H$, in a way that does not depend on $G$.
\begin{prop}\label{prop: pansu convergence doesnt depend on carnot factor}
 Let $(G,(V_i^G)_{i=1}^r)$ be an $r$-step stratified group and $(H,\Delta_H)$ be a simply connected $s$-step nilpotent Lie group equipped with a bracket generating distribution. Let $d$ be a sub-Finsler metric on $G\times H$ associated to the distribution $V_1^G \times \Delta_H$, and let $d_\infty$ be the corresponding Pansu limit metric. Then there exists an asymptotic grading such that
 \[|d(p,q)-d_\infty(p,q)|=O(\max\{d_\infty(1,p), d_\infty(1,q)\}^{1-1/s}), \quad {\rm as} \; p,q \rightarrow \infty.\]
\end{prop}
\begin{proof}
  Let $(V_i^H)_{i=1}^s$ be an asymptotic grading of $H$. Then $(V_i)_{i=1}^{\max\{r,s\}}$ defined by $V_i = V_i^G \times V_i^H$ is an asymptotic grading of $G \times H$. With respect to this grading $H$ is a Carnot quotient ideal of $G\times H$ (according to Definition~\ref{def: Carnot quotient ideal}), and thus by Definition~\ref{def: constant beta} we have $\beta \leq s$. Since $\alpha_\infty \geq 1$ we have $\alpha_\infty/\beta \geq 1/s$, and therefore Theorem~\ref{main thm: pansu} implies 
 \[|d(p,q)-d_\infty(p,q)|=O(\max\{d_\infty(1,p), d_\infty(1,q)\}^{1-1/s}), \quad {\rm as} \; p,q \rightarrow \infty.\]
\end{proof}
We stress that the nilpotency step of $G\times H$ as in Proposition~\ref{prop: pansu convergence doesnt depend on carnot factor} is $\max\{r,s\}$, which could be larger than $s$.

Next, we consider the very specific product as in Example~\ref{example: beta}\eqref{example: CQI2}.
\begin{example}
 Let $\mathfrak{g} = N_{5,2,2} \times N_{5,2,1}$ and $(V_i)_{i=1}^4$ be the sub-Finsler group equipped with the asymptotic grading from Example~\ref{example: beta}\eqref{example: CQI2}.
 Then by Proposition~\ref{prop: pansu convergence doesnt depend on carnot factor}:
 \[|d(1,p) - d_\infty(1,p)| = O(d(1,p)^{2/3}), \quad {\rm as}\; p \rightarrow \infty,\]
 which is an improvement on the bound
 \[|d(1,p) - d_\infty(1,p)| = O(d(1,p)^{3/4}), \quad {\rm as}\; p \rightarrow \infty\]
 which one obtains by using \cite[Proposition 4.1]{BL12}.
\end{example}
We stress that different choices of gradings can give different estimates in both the quantitative version of Pansu Asymptotic Theorem and the quantitative version of Mitchell Tangent Theorem. Regarding the asymptotic theorem~\ref{main thm: pansu}, the exponent $1-\alpha_\infty/\beta$ clearly varies with the choice of the asymptotic grading. This number is at most $1-1/s$, and the smaller it is, the sharper the bound. Excluding the Carnot case, this number is at least $1/s$. In the next example, we show that this is the case for a specific grading.
\begin{example}
 Let $(G,(V_i)_{i=1}^s)$ be a stratified group. Let $d^{(k)}$ be a sub-Finsler distance on $G$ associated to the distribution $\Delta_{(k)} \coloneqq V_1 \oplus V_k$, for $k \in \{2,\ldots,s\}$, and let $d_\infty^{(k)}$ be the corresponding Pansu limit metric. With respect to the stratification, we have $\beta = s$ and $\alpha_\infty = k-1$. Thus Theorem~\ref{main thm: pansu} implies 
 \[|d^{(k)}(p,q)-d^{(k)}_\infty(p,q)|=O(\max\{d^{(k)}_\infty(1,p), d^{(k)}_\infty(1,q)\}^{(s-k+1)/s}), \quad {\rm as} \; p,q \rightarrow \infty.\]
 Note that for $k=s$, the exponent in the above equation is $1/s$, which by Proposition~\ref{prop: alpha_inf < beta} is the minimum value of $1-\alpha_\infty/\beta$ for groups which are not Carnot.
\end{example}

We also obtain a result similar to Proposition~\ref{prop: pansu convergence doesnt depend on carnot factor}, but for tangent spaces.
\begin{prop}\label{prop: mitchell convergence doesnt depend on carnot factor}
 Let $(G,(W_i^G)_{i=1}^r)$ be an $r$-step stratified group and $(H,\Delta_H)$ be a Lie group equipped with an $s$-step bracket generating distribution. Let $d$ be a sub-Finsler metric on $G\times H$ associated to the distribution $V_1^G \times \Delta_H$, and let $d_0$ be the corresponding tangent metric. Then there exists a tangent grading such that
 \[|d(p,q)-d_0(p,q)|=O(\max\{d_0(1,p), d_0(1,q)\}^{1+1/s}), \quad {\rm as} \; p,q \rightarrow 0.\]
\end{prop}
\begin{proof}
 The proof is similar to the one of Proposition~\ref{prop: pansu convergence doesnt depend on carnot factor}.
 Let $(W_i^H)_{i=1}^s$ be a tangent grading of $H$. Then $(W_i)_{i=1}^{\max\{r,s\}}$ defined by $W_i = W_i^G \times W_i^H$ is a tangent grading of $G \times H$. With respect to this grading $H$ is a Carnot quotient ideal of $G\times H$ (according to Definition~\ref{def: Carnot quotient ideal}), and thus by Definition~\ref{def: constant beta} we have $\beta \leq s$. Since $\alpha_0 \geq 1$ we have $\alpha_0/\beta \geq 1/s$, and therefore Theorem~\ref{main thm: mitchell} implies 
 \[|d(p,q)-d_0(p,q)|=O(\max\{d_0(1,p), d_0(1,q)\}^{1+1/s}), \quad {\rm as} \; p,q \rightarrow 0.\]
\end{proof}
\printbibliography
\end{document}